\documentclass[11pt,twoside,reqno,centertags]{amsart}
\usepackage{amsmath,amsthm,amsfonts,amssymb}
\advance\textheight by 6truemm
\newtheorem{Theorem}{Theorem}[section]

\newtheorem{theorem}[Theorem]{Theorem}
\newtheorem{corollary}{Corollary}[section]
\newtheorem{proposition}{Proposition}[section]
\newtheorem{lemma}{Lemma}[section]
\newtheorem{remark}{Remark}[section]
\newtheorem{example}{Example}[section]
\newcommand{\R}{{\mathbb R}}
\newcommand{\N}{{\mathbb N}}
\newcommand{\eps}{\varepsilon}
\def\+#1{{#1}^+}

\numberwithin{equation}{section}

\usepackage{color}

\newcommand{\be}{\begin{equation} \label}
\newcommand{\ee}{\end{equation}}

\def\e{\varepsilon}
\def\eps{\varepsilon}
\def\Rn{{\R}^n}
\def\Rm{{\R}^m}
\def\IR {\int_{D_R}}

\def\IS {\int_\Sigma}

\begin{document}

\title[Liouville theorems for elliptic problems]{Liouville theorems and universal estimates \\ 
for superlinear elliptic problems \\ without scale invariance} 
\author[Pavol Quittner and Philippe Souplet]{Pavol Quittner$^{(1)}$ and Philippe Souplet$^{(2)}$}

\thanks{$^{(1)}$Department of Applied Mathematics and Statistics, Comenius University,
Mlynsk\'a dolina, 84248 Bratislava, Slovakia. Email: quittner@fmph.uniba.sk} 

\thanks{$^{(2)}$Universit\'e Sorbonne Paris Nord, CNRS UMR 7539, LAGA,
93430 Villetaneuse, France. Email: souplet@math.univ-paris13.fr}

\date{}

\setcounter{tocdepth}{1}

\vspace*{-1cm}

\begin{abstract}
 We give applications of known and new Liouville type theorems to universal singularity and decay estimates
for non scale invariant elliptic problems,
including Lane-Emden and Schr\"odinger type systems.
 This applies to various classes of nonlinearities with regular variation
and possibly different behaviors at $0$ and $\infty$.
To this end, we adapt the method from \cite{SouDCDS} to elliptic systems,
which relies on a generalized rescaling technique and on doubling arguments from~\cite{PQS1}.

 This is in particular facilitated by new
Liouville type theorems in the whole space and in a half-space, for elliptic problems without scale invariance,
 that we obtain.
Our results apply to some non-cooperative systems, 
for which maximum principle based techniques such as moving planes do not apply.
 To prove these Liouville type theorems, we employ two methods, respectively based on
 Pohozaev-type identities combined with functional inequalities on
 the unit sphere, and on reduction to a scalar equation by proportionality of components.

  In turn we will survey the existing methods for proving Liouville-type theorems for superlinear elliptic equations and systems, 
 and list some of the typical existing results for (Sobolev subcritical) systems.
 
 In the case of scalar equations, we also revisit the classical Gidas-Spruck integral Bernstein method,
providing some improvements which turn out to be efficient for certain nonlinearities,
and we next compare the performances of various methods on a benchmark example.

\vskip 0.1cm

\end{abstract}

\maketitle
\tableofcontents

\section{Introduction}
\label{intro}

\vskip 0.5mm

\subsection{Background and motivation}
\label{intro1}

The present paper is concerned with Liouville type theorems and universal estimates for elliptic problems, especially systems.
Throughout this paper, by a {\it solution} we mean a nonnegative strong solution\footnote{Namely, $u\in W^{2,r}_{loc}$ for all $r\in(1,\infty)$ and the equation is satisfied a.e. 
Moreover, when boundary conditions are imposed, a strong solution is assumed to be continuous up the boundary. 
Of course strong solutions are classical if the nonlinearity is locally H\"older continuous (which we do not assume in general).}, unless otherwise specified. 
To begin with, denoting the Sobolev exponent by 
$$p_S=
\begin{cases}
\frac{n+2}{n-2},&\hbox{ if $n\ge 3$} \\
\noalign{\vskip 1mm}
\infty,&\hbox{ if $n\le 2$,}
\end{cases}
$$
let us first recall the celebrated Gidas-Spruck elliptic Liouville theorem \cite{GSa}
(see also \cite{BVV} for a simplified proof).

\medskip
\noindent {\bf Theorem~A.}
{\it Let $1<p<p_S$. Then the equation 
\be{ellup}
-\Delta u=u^p
\ee
has no nontrivial solution in $\R^n$.}
\medskip

\noindent The exponent $p_S$ is critical for nonexistence since, as is well known,
there exist positive solutions
of \eqref{ellup}
in $\R^n$ whenever $n\ge 3$ and $p\ge p_S$ 
(see \cite[Section~9]{QSb} and the references therein; the solution can even be taken bounded and radial).
Concerning Liouville type theorems for semilinear equations with more general nonlinearities, 
let us mention the following important 
generalization of Theorem~A, due to~\cite{LZ}.

\medskip
\noindent {\bf Theorem~B.}
{\it Assume that $f\in C([0,\infty))$, with $f(s)>0$ for all $s>0$ and, if $n\ge 3$,
\be{condLZ}
\hbox{$s^{-p_S}f(s)$ is nonincreasing and nonconstant.}
\ee
 Then the equation
 \be{S-LiouvLZ}
-\Delta u=f(u),\qquad x\in \Rn
\ee 
has no nontrivial solution.}
\medskip

Elliptic Liouville theorems, in combination with rescaling techniques, have many applications. 
Theorem~A (and its half-space analogue) was used in \cite{GSb}, as well as in many subsequent papers,  to show
a priori bounds and existence for elliptic Dirichlet problems.
In~\cite{PQS1}, by combining the rescaling method with a doubling argument,
it was shown that the Liouville property in Theorem~A leads to universal estimates
for singularities and decay in unbounded domains.
Liouville type theorems and their applications have been extended in various directions,
for instance systems or quasilinear problems (cf.,~e.g.,~\cite{Fel, RZ, SZ, PQS1, SouAdv, QSb, NNPY} and the references therein,
and see also subsection~\ref{survey} below).

\smallskip
In \cite{SouDCDS} (see also \cite{GIR} for a first step in that direction)
was investigated the possibility, in connection with 
Liouville type theorems, to generalize rescaling methods
so as to obtain a priori bounds and universal singularity or decay estimates
for elliptic (and parabolic) equations which are {\it not} scale invariant, even asymptotically,
and whose behavior can be quite far from power-like.
This led to several natural classes of nonlinearities, including nonlinearities with regular variation,
and motivated the need for further Liouville type theorems for equations with nonhomogeneous nonlinearities.
Let
\be{hypcontpos}
\hbox{$f\in C([0,\infty))$ with
$f(s)>0$ for all $s>0$,}
\ee
and recall that $f$ has {\it regular variation at $\infty$} (resp., $0$) with index $q\in\R$ if 
\be{hypregulvar}
L(s):= s^{-q} f(s) \hbox{ satisfies } 
\lim_{\lambda\to \infty \atop ({\rm resp.,}\, \lambda\to0)} \frac{L(\lambda s)}{L(\lambda)}=1\ \hbox{ for each $s>0$.}
\ee
 Such functions $L$ include, for instance, logarithms and their powers and iterates (see, e.g., \cite[Section~2.3]{SouDCDS} for more examples).
In particular the following connection was obtained in~\cite{SouDCDS} between Liouville theorems and universal bounds, 
extending \cite{PQS1} to semilinear elliptic equations without scale invariance.

\medskip

\goodbreak
\noindent {\bf Theorem~C.}
{\it Assume \eqref{hypcontpos}.
Suppose that
\be{HypLiouville}
\hbox{$-\Delta v=f(v)$
has no nontrivial bounded solution in $\Rn$}
\ee
and
\be{HypLiouvilleReg}
\hbox{$f$ has regular variation at $0$ and $\infty$ with indices in $(1,p_S)$.}
\ee
There exists a constant $C=C(n,f)>0$ such that, for any domain $\Omega\subset \Rn$ and any 
 solution of $-\Delta u=f(u)$ in $\Omega$, we have the estimate 
\be{EstimA}
{f(u(x))\over u(x)}\le {C\over{\rm dist}^2(x,\partial\Omega)},\quad\hbox{ for all $x\in \Omega$.}
\ee
}

 \smallskip
 
One of the aims of the present paper is to develop such tools and estimates
for non scale-invariant elliptic {\it systems}.
We will take this opportunity to summarize and compare the various existing methods to prove Liouville-type theorems
for elliptic equations and systems.

A second motivation is the companion paper \cite{QSparab23}, where we develop similar 
ideas for parabolic problems.
Since we use there the approach from \cite{Q16, Q21, Q21b}, which enables one to derive parabolic from elliptic Liouville theorems,
this will be facilitated by the results in the present paper.

Specifically, we shall consider elliptic systems of the form
\be{GenSystem}
-D\Delta U=f(U),
\quad x\in \Rn,
\ee
where $n\ge 2$, $m\ge 1$, $D$ is an $m\times m$ diagonal matrix with positive (constant) entries $d_i>0$,
$U:\Rn\to\Rm$, $f\in C(\R^m;\R^m)$ and
where we denote $\Delta U=(\Delta u_1,\dots,\Delta u_m)^{T}$.
The components of $f$ will be denoted by $f_i$. 
We shall also consider the corresponding half-space problem
\be{GenSystemB}
\left\{
  \begin{aligned}
 \hfill -D\Delta U&=f(U), \hfill&\quad x\in \Rn_+, \hfill \\ 
 \noalign{\vskip 1mm}
 \hfill U&=0, \hfill&\quad x\in \partial\Rn_+,\hfill 
   \end{aligned}
   \right.
\ee
where $\Rn_+:=\{x=(x_1,\dots,x_n)\in\R^n:x_n>0\}$.

\subsection{A brief survey of existing methods and results}
\label{survey}
We first give a brief survey of existing methods for proving 
Liouville-type theorems for superlinear elliptic equations and 
systems (we warn the reader that our list of references is not exhaustive).

\smallskip

$\bullet$ Scalar equations: whole space
\vskip 2pt

\begin{itemize}
\item[(a1)] Integral Bernstein method (nonlinear integral estimates obtained by using Boch\-ner's identity, power change of dependent variable, 
and multipliers involving powers of $u$ and cut-offs), and its variants 
 \cite{GSa, BVV, BVGHV, BVGHV2}

\vskip 2pt

\item[(a2)] Direct Bernstein method 
(application of the maximum principle to auxiliary functionals involving the unknown function and its gradient)
\cite{Lio85, JLi, BVGHV,  FPS2, BVGHV2, Lu}

\vskip 2pt

\item[(a3)]   Moving planes methods 

\vskip 2pt

\begin{itemize}
\item[(a3.1)]via symmetry, using the Kelvin transform 
 \cite{CLi,Bia}

\vskip 1pt

\item[(a3.2)]moving spheres \cite{LZ}
\end{itemize}

\end{itemize}
\smallskip

\goodbreak
$\bullet$  Scalar equations: half-space

\vskip 2pt

\begin{itemize}
\item[(b1)]  Moving planes methods 

\vskip 2pt
\begin{itemize}
\item[(b1.1)]via symmetry and reduction to a one-dimensional problem on a half-line,
sometimes using the Kelvin transform  \cite{GSb, BCN, FPS}
\vskip 1pt

\item[(b1.2)]via monotonicity in the normal direction and reduction to an $(n-1)$ dimensional problem in the whole space,
sometimes using notions of stability \cite{Dan, BCN, Far1, FV07}
\vskip 1pt

\item[(b1.3)]via monotonicity and convexity in the normal direction,
sometimes using notions of stability \cite{CLZ,DSS}
  \end{itemize}

\vskip 2pt

\item[(b2)] Stability estimates, combined with rescaling and monotonicity formulas  \cite{DFT}

\end{itemize}

\smallskip

$\bullet$ Systems: whole space
\vskip 2pt

\begin{itemize}
\item[(a1)] Integral Bernstein method \cite{BVR}

\vskip 2pt

\item[(a3)]   Moving planes methods 

\vskip 2pt

\begin{itemize}
\item[(a3.1)]via symmetry, sometimes using the Kelvin transform 
\cite{dFF,Zou,GuLi}

\vskip 1pt

\item[(a3.2)]moving spheres 
\cite{RZ, BM, Zou2}

\end{itemize}

\vskip 2pt

\item[(a4)] Pohozaev-type identities combined with functional inequalities 
on the unit sphere\footnote{This method also works for scalar equations, but its primary interest is for systems.}
 \cite{SZ96, SouAdv, Phan1, QS-CMP,  SouNHM, AYZ, Faz, FazGh, Phan2, LZZ,DYZ,YZ}

\vskip 2pt

\item[(a5)] Reduction to a scalar equation by proportionality of components
\cite{Lou, QS12, DAm, MSS, Far}

\end{itemize}

\smallskip
$\bullet$ Systems: half-space

\vskip 2pt

\begin{itemize}
\item[(b1)]  Moving planes methods 

\vskip 2pt
\begin{itemize}

\item[(b1.2)]via monotonicity in the normal direction and reduction to an $(n-1)$ dimensional problem in the whole space
 \cite{BiMi}
\vskip 1pt

\item[(b1.3)]via monotonicity and convexity in the normal direction
\cite{CLZ}

\vskip 1pt

\item[(b1.4)] moving spheres 
\cite{RZ}

 \end{itemize}

\vskip 2pt

\item[(b3)]Reduction to whole space by a Harnack inequality, a doubling argument and variational identities
 \cite{DWW,DW}

\vskip 2pt

\item[(b4)]  Reduction to a scalar equation by proportionality of components
\cite{Lou, MSS}
\end{itemize}

\begin{remark} \label{remOtherMeth}
(i) Some other methods are suitable for elliptic inequalities,
but usually require non-optimal growth restrictions on the nonlinearity when one considers equations.
Let us mention the rescaled test-function method 
(see \cite{MP} and, e.g., \cite[Theorem~8.4]{QSb}) and the spherical maximum method
(see \cite{AS} and \cite[Remark 8.5a(ii)]{QSb} and the references therein).
 Also one can sometimes prove a Liouville theorem by reducing 
to the radial case after applying a symmetry result in the whole space based on moving planes (see, e.g.,~\cite{BuSi}).

\smallskip

(ii) Some other methods apply to special classes of solutions, such as stable solutions or solutions 
stable outside of a compact (see \cite{Far1} and \cite[Remark 8.5(vi)]{QSb} and the references therein).
See also \cite[Appendix~H]{QSb} for some other, more specific, methods.
Also, there are other methods which essentially apply to dissipative nonlinearities or nonlinearities with sublinear or linear growth at infinity. 
We will not discuss them, since we are mainly interested in (non-dissipative) superlinear problems.
\end{remark}

 Let us next mention some of the typical existing results of Liouville type for (Sobolev subcritical) 
superlinear systems in the whole space.
\vskip 2pt

One of the best understood cases is that of subcritical cooperative nonlinearities, namely those such that 
$$\hbox{$t\to t^{-p_S}f(tu)$ is $>0$ nonincreasing in $[0,\infty)$ component-wise for each $u\in(0,\infty)^m$.}$$
  In particular the results in \cite{Fel,RZ,Zou2}, obtained by moving planes or moving spheres methods (a3), give a fairly complete picture for that class
  (some special cases have also been treated by the method (a1) in \cite{BVR}).
\vskip 2pt

The following class of Schr\"odinger type nonlinearities has also received a lot of attention:
\be{Schrod}
f_i(u)=\sum_{1\le i,j\le m}\beta_{ij}u_i^qu_j^{q+1},\quad i=1,\dots,m
\ee
 with $q>0$, $\beta_{ij}\in\R$  (note that this is a scale-invariant problem).
While the cooperative case (i.e.~$\beta_{ij}\ge 0$ for $i\ne j$) is essentially covered by \cite{RZ,Zou2}
(in the subcritical range $2q+1<p_S$), the noncooperative cases are only partially understood and depend on delicate 
properties of the interaction matrix $(\beta)$.
Optimal results (in the full Sobolev subcritical range) for $m=2$ and low dimensions were obtained in \cite{QS-CMP} by the method (a4),
which strongly relies on the variational structure of the problem.
Some noncooperative cases, as well as other systems involving sums and/or products of 
powers\footnote{possibly noncooperative and nonvariational, including some Lotka-Volterra type systems},
have been treated in the full Sobolev subcritical range by the method (a5) in \cite{Lou, QS12, MSS},
and in the half-space case by the method (b3) in \cite{DW}.
Other Liouville-type results for \eqref{Schrod} can be found in, e.g., \cite{DWW,TTVW,SouNHM} (but they do not cover the full Sobolev subcritical range).

\vskip 2pt

Another system of particular interest, as a model of Hamiltonian (as opposed to gradient) system,
is the Lane-Emden system $-\Delta u = v^p$, $-\Delta v=u^q$.
Although cooperative, a significant difficulty is that the two equations have different degrees of homogeneity, 
and the Sobolev subcritical condition is given by a hyperbola in $(p,q)$.
Several results have been obtained by methods (a3) and (a4) -- see Remark~\ref{remLE}(i) below for details.

\begin{remark} \label{remRestrictions}
(i) The method of Pohozaev-type identities combined with functional inequalities 
on the unit sphere (a4) doesn't seem to have been considered for non scale invariant problems nor for half-spaces so far.
As we will see, it will turn out to be applicable, after some modifications, and lead to some new Liouville-type theorems
for systems. 
We will also study the application to non scale invariant systems of the method 
of proportionality of components (cf.~(a5)-(b4)),
which was so far mostly considered for scale invariant nonlinearities.

\smallskip

(ii) On the contrary, the Bernstein methods (a1)-(a2) do not seem to be applicable to the half-space case. 
This may be due to the fact that, in order for them to be applied with full efficiency,
one has to use negative powers of $u$ as new variable and/or test function,
which may cause substantial difficulties near the boundary.  
As for the method (b2) for systems in the half-space, it seems to depend on scale invariance in an essential way.
Finally the direct Bernstein method (a2) does not seem to have been studied for systems.

\smallskip

 (iii) The optimality of the subcriticality condition for the Liouville property in the scalar case (cf.~after Theorem~A)
has a natural counterpart in many systems. However, as an interesting difference,
whereas any entire solution of the scalar equation \eqref{ellup} in the critical case $p=p_S$
is known to be radially symmetric (cf.~\cite{CGS}), the symmetry property may fail in the case of systems with critical nonlinearities of the form \eqref{Schrod}  (see ~\cite{GLW}).
\end{remark}

\subsection{Outline}
 Section~\ref{LTT} is devoted to our new Liouville type theorems.
In Subsection~\ref{LPoh} and \ref{LProp}, we will give Liouville type theorems 
for non scale invariant systems, both in the whole space and in the half-space,
obtained respectively by the method (a4), based on Pohozaev identities and functional inequalities on the unit sphere,
and by the method (a5)-(b4), based on of proportionality of components.
In Subsection~\ref{LGS}, we will get back to the case of scalar equations and (mildly) revisit the classical Gidas-Spruck integral Bernstein method (a1),
providing some improvements which turn out to be efficient for certain nonlinearities.
  In Section~\ref{sec-benchmark} we will compare the performances of  several methods on a benchmark scalar example.
In Section~\ref{secUB} we will give applications of 
 known and new Liouville type theorems to universal bounds for singularity and decay estimates.
This will concern non scale invariant systems with nonlinearities having regular variation, either with different indices on the
components (Lane-Emden type systems), or with equal indices.
 In Section~\ref{ProofPoh} we will prove Theorem~\ref{thm1} and
also provide some additional properties of possible independent interest.
The remaining sections will be devoted to the proofs of the other resuts.
Finally some properties of regularly varying functions will be given in appendix.

\section{Liouville type theorems}
\label{LTT}

\subsection{Method of Pohozaev identities and functional inequalities on the unit sphere}
\label{LPoh}
We assume that the system has gradient structure, i.e.
\be{hypGradSyst}
f=\nabla F,\hbox{ for some function $F\in C^1(\R^m,\R)$.}
\ee
Fixing $p\ge q>1$,
we next assume that, for each $M>0$, there exist constants $c_M,C_M>0$
such that the functions $f$ and $F$ satisfy the following conditions:
\be{hypGrowthB}
|f(U)|\leq C_M|U|^q, \ \hbox{ for all $U\in[0,M]^m$},  
\ee
\be{hypGrowthC}
2nF(U)-(n-2)U\cdot\nabla F(U)\geq c_M |U|^{p+1}, \ \hbox{ for all $U\in[0,M]^m$}.  
\ee
For some results we also assume that there exist constants $\xi\in(0,\infty)^N$ and $c_M>0$ such that
\be{hypGrowthD}
\xi\cdot f(U)\ge c_M |U|^p,\ \hbox{ for all $U\in[0,M]^m$}. 
\ee
Note that since $f(0)=0$ by \eqref{hypGrowthB}, $U\equiv 0$ is a solution of \eqref{GenSystem} and \eqref{GenSystemB}.
We introduce the exponent
\be{defpstar}
p^*=  p^*(n):=
\begin{cases}
(n+2)/(n-2),\hfill&\quad \hbox{if $n\leq 4$,}\hfill\\ 
\noalign{\vskip 1mm}
(n-1)/(n-3),\hfill&\quad \hbox{if $n\geq 5$.}\hfill
  \end{cases}
\ee

\begin{theorem} \label{thm1}
Let 
\be{Hypp0}
1<p_0<p^{**}:=
\begin{cases}
p^*,\hfill&\quad \hbox{in case of \eqref{GenSystem}}\hfill \\ 
\noalign{\vskip 1mm}
n/(n-2),\hfill&\quad \hbox{in case of \eqref{GenSystemB}.}\hfill
\end{cases}
\ee
There exists $\e_0=\e_0(n,p_0)>0$ such that, if $p_0-\e_0\le q\le p\le p_0+\e_0$,
\eqref{hypGrowthB}-\eqref{hypGrowthD} are satisfied and 
$U$ is a bounded solution 
of \eqref{GenSystem} or \eqref{GenSystemB}, then $U\equiv 0$.
 \end{theorem}

\goodbreak

\begin{remark}
If we only consider radial solutions of \eqref{GenSystem},
then using the same arguments as in the proof of \cite[Proposition 5(i)]{QS-CMP}
one can show that the condition $p<p^*$ in Theorem~\ref{thm1} can be replaced with $p<p_S$.
\end{remark}

 As a first application of Theorem~\ref{thm1}, we have the following result
concerning noncooperative, multi-power Schr\"odinger type systems,
which arise in models of Bose-Einstein condensates (see \cite{Fran, KL}).
Liouville theorems for systems of this type were established in \cite{DWW,TTVW,QS-CMP,SouNHM}
under various assumptions
(\cite{DWW,TTVW} being restricted to systems with single power homogeneity).
Corollary~\ref{thm1cor1} directly follows from Theorem~\ref{thm1} by taking $p=q=\alpha$.
We note that Corollary~\ref{thm1cor1} for $\Omega=\Rn$ could be deduced from \cite[Theorem~4]{QS-CMP}
at the expense of a bit of additional work, but the case $\Omega=\Rn_+$ is completely new.

\begin{corollary} \label{thm1cor1}
Let $m=2$, $1<\alpha<\beta\le p_S$, $\alpha<p^{**}$, $\lambda>-1$, $\mu\ge-1$, $b\ge\strut 0$.
Let $f=\nabla F$ with $F=G+H$, where
$$G(U)=\frac{1}{\alpha+1}\bigl(u^{\alpha+1}+v^{\alpha+1}+2\lambda u^{\frac{\alpha+1}{2}} v^{\frac{\alpha+1}{2}}\bigr)$$
and
$$H(U)=\frac{b}{\beta+1}\bigl(u^{\beta+1}+v^{\beta+1}+2\mu u^{\frac{\beta+1}{2}}v^{\frac{\beta+1}{2}}\bigr) \quad\hbox{or}\quad 
H(U)=bu^{\frac{\beta+1}{2}}v^{\frac{\beta+1}{2}}.$$
If $U$ is a bounded solution 
of \eqref{GenSystem} or \eqref{GenSystemB}, then $U\equiv 0$.
 \end{corollary}

As another typical application of Theorem~\ref{thm1}, we have the following result,
which involves nonlinearities with logarithmic behavior, either near $0$ or at infinity.

\begin{corollary} \label{thm1cor0}
Let $m=2$, $K\ge 1$, $\sigma\in\{-1,1\}$, assume \eqref{Hypp0} 
and let 
\be{DeffiCor0}
f=\nabla F,\quad F(U)=g^2(u)+g^2(v)-2\lambda g(u)g(v),\quad g(s)=s^{(p_0+1)/2}\log^a(K+s^\sigma).
\ee

 \smallskip
 
(i) Let $K=1$ and assume 
\be{HypaLCor0}
\sigma a>0,\quad |a|<a_0 
\quad\hbox{and}\quad
0<\lambda<\lambda_0:= {2\sqrt{1-\rho}\over 2-\rho} \ \hbox{with $\rho={|a|\over a_0}$}, 
\ee
where $a_0:={p_S-p_0\over 2}$ if $\sigma=-1$, $a_0:={p^{**}-p_0\over 2}$ if $\sigma=1$.
If $U$ is a bounded solution 
of \eqref{GenSystem} or \eqref{GenSystemB}, then $U\equiv 0$.

 \smallskip

 (ii) Let $K>1$. Assertion (i) remains valid with $a_0:={p^{**}-p_0\over 2}\theta(K)$, 
where $\theta:[1,\infty)\to [1,\infty)$ is the inverse bijection of $s\mapsto s^{-1}e^{s-1}$.
\end{corollary}

If $K=\sigma=1$, $n\leq4$, and we consider problem \eqref{GenSystem} in Corollary~\ref{thm1cor0},
then $a_0={p_S-p_0\over 2}$. 
The next proposition shows that in this case, the assumption $a<a_0$ 
for nonexistence in Corollary~\ref{thm1cor0} is essentially optimal. 

\begin{proposition} \label{thm1cor0prop}
Let $m=2$, $p_0\in(1,p_S)$, $K=\sigma=1$, $\lambda\in(0,1)$,
and let $f$ be given by~\eqref{DeffiCor0}.
If $a> {p_S-p_0\over 2}$,
then there exists a positive bounded solution of~\eqref{GenSystem}.
 \end{proposition}

\begin{remark} \label{rem1.1}
\smallskip

(a) In the half-space case, the condition $p<n/(n-2)$ is better than the well-known condition $p\le (n+1)/(n-1)$ for inequalities
\cite{BCDN}.

\smallskip

(b) System \eqref{GenSystem}   is {\bf not} cooperative in general.  
Therefore, maximum principle techniques, such as moving planes or moving spheres
(see,~e.g.,~\cite{CLi, dFF, RZ, BM}),
are not applicable here. 

\smallskip

(c) The value of $\e_0$ in Theorem~\ref{thm1} (also in Theorem~\ref{thm3} below) is explicit and could be computed from the proof.
However its expression is rather complicated.

\smallskip

(d) By using some arguments from \cite{QS-CMP, SouNHM}, the boundedness hypothesis in Theorem~\ref{thm1} could actually be 
replaced with an exponential growth condition, provided we assume  
 \eqref{hypGrowthC}, \eqref{hypGrowthD} with $c_M$ independent of $M$
 and 
 \be{Hypfa0}
|f(U)|\le c_1(|U|^q+|U|^p),
\ee
 with some $c_1>0$, instead of \eqref{hypGrowthB}.
\end{remark}

\begin{remark} \label{remSlow}
The logarithms in Corollary~\ref{thm1cor0} could be replaced, under suitable assumptions
on the parameters, by some more general slowly varying functions (cf.~\eqref{hypregulvar}), 
 such as iterated logarithms (see \cite[Section~2.3]{SouDCDS} for more examples).
\end{remark}

\subsection{Method of proportionality of components}
\label{LProp}

 Consider the system
\be{systProp}
\left\{
  \begin{aligned}
 \hfill -\Delta u&=\phi(u,v)k(u)(g(v)-\lambda g(u)), \quad x\in D\\ 
 \noalign{\vskip 1mm}
 \hfill -\Delta v&=\phi(u,v)k(v)(g(u)-\lambda g(v)), \quad x\in D\\ 
   \end{aligned}
   \right.
\ee
where 
\begin{equation} \label{ass-prop1}
\left\{
  \begin{aligned}
&\hbox{\ $D=\R^n$ or $\R^n_+$,}\\
&\hbox{\ $\phi,k,g$ continuous on $[0,\infty)$,\ $\phi\geq c_M>0$ for $u,v\leq M$,}\\
&\hbox{\ $\lambda \geq0$, $g(0)=0$, $g>0$ on $(0,\infty)$}
   \end{aligned}
   \right.
\end{equation}
($\phi$ could also depend on $x,\nabla u,\nabla v$).
Set $\varphi:=kg$ if $\lambda>0$, $\varphi:=g$ if $\lambda=0$, and assume that
\begin{equation} \label{ass-prop2}
 \hbox{$\varphi$ is $C^1$ on $(0,\infty)$ and $\varphi'(t)>0$ for $t>0$.}
\end{equation}
Fix also $\eps=0$ if $\lambda>0$, $\eps>0$ if $\lambda=0$, set $\tilde k:=k-\eps$
and assume that
\begin{equation} \label{ass-propk}
\hbox{$\tilde k(0)\geq0$ and $\tilde k/g$ is nonincreasing on $(0,\infty)$.}
\end{equation}
The following theorem guarantees the equality of components.
It is a special case of a result from \cite{QSparab23}, established there
for parabolic systems.
See \cite{Lou, QS12, DAm, MSS, Far} for results of this type in the case of 
elliptic systems with power nonlinearities.

\begin{theorem} \label{thm-proportional}
Assume \eqref{ass-prop1}--\eqref{ass-propk}. 
Let $(u,v)$ be a bounded solution of \eqref{systProp},
with homogeneous Dirichlet boundary conditions if $D=\R^n_+$. Then $u\equiv v$.
\end{theorem}

As a consequence of Theorem~\ref{thm-proportional} we obtain the following Liouville type theorem
 (see Examples~\ref{ex43} and \ref{ex44} below for applications).

\begin{theorem} \label{thm-proportional2}
Let $\lambda\in[0,\infty)\setminus\{1\}$, set $h(s)=\phi(s,s)k(s)g(s)$
and assume \eqref{ass-prop1}--\eqref{ass-propk} with $k(s)>0$ for $s>0$. If $\lambda\in[0,1)$ assume in addition
\be{HypProp2}
\begin{cases}
\hbox{$s^{-p_S}h(s)$ nonincreasing nonconstant,}&\hbox{if $D=\Rn$ and $n\ge 3$,}\\
\noalign{\vskip 2mm}
\hbox{$h\in C^1([0,\infty)\cap C^2(0,\infty)$ and $h$ convex,}&\hbox{if $D=\Rn_+$ and $n\ge 2$}.$$
\end{cases}
\ee
If $(u,v)$ is a bounded solution of \eqref{systProp},
with homogeneous Dirichlet boundary conditions if $D=\R^n_+$, then $u\equiv v\equiv 0$.
\end{theorem}

\begin{remark} 
(i) The assumption that $\phi\geq c_M>0$ for $u,v\leq M$ in Theorem~\ref{thm-proportional} is not technical,
since \eqref{systProp} admits the semitrivial solution $(u,v)=(0,C)$ whenever $\phi(0,C)=0$ for $C>0$.
Theorem~\ref{thm-proportional} in $\Rn$ may even fail in the class of positive solutions when 
$\phi(u,v)=u^pv^p$, $k(s)=1$ and $g(s)=s^q$
(see \cite[Theorems~2.1 and 2.4]{MSS} for positive and negative results depending on the values of $p$).

\smallskip

(ii) The boundedness assumption in Theorems~\ref{thm-proportional} and \ref{thm-proportional2}
 is not purely technical either.
Indeed, consider $\phi\equiv 1$, $k(s)=s^p$ and $g(s)=s^q$, with $q\ge p>0$, $p+q\le 1$ and $\lambda>0$.
Then assumptions \eqref{ass-prop1}--\eqref{ass-propk} are satisfied but 
\eqref{systProp} admits unbounded solutions in $D=\Rn$, or in $D=\Rn_+$ with homogeneous Dirichlet boundary conditions,
of the form $(w,0)$ and $(0,w)$ with $w\ge 0$ nontrivial.
Namely, for $p+q<1$, $w$ is given by
$w(x)=c|x_n|^a$, $a=2/(1-p-q)$ and some $c=c(p,q,\lambda)>0$ 
and, for $p+q=1$, by $w(x)=\cosh(\sqrt\lambda\,x_n)$ if $D=\Rn$ and $w(x)=\sinh(\sqrt\lambda \,x_n)$ if $D=\Rn_+$.

\smallskip

(iii) Under the assumptions of Theorem~\ref{thm-proportional2} with $\lambda=1$, it follows that $u$ is a bounded harmonic function,
hence $u=v=C\ge 0$ in case $D=\Rn$ by Liouville's Theorem,
and $u=v=0$ in case $D=\Rn_+$ by Phragmen-Lindel\"of's Theorem (see, e.g.,~\cite{PW}).
\smallskip

(iv) Assumption \eqref{HypProp2} for $D=\Rn_+$ and $n\ge 2$ can be replaced for instance by $\liminf_{s\to 0^+}$
$s^{-(n+1)/(n-1)}h(s)>0$ (see \cite[Corollary~5.6]{AS}).
\end{remark}

\subsection{Modified Gidas-Spruck method}
\label{LGS}

In this subsection we get back to the case of scalar equations and (mildly) revisit the classical Gidas-Spruck method.
It  appears that there is room for some improvements which  will turn out to be efficient for certain nonlinearities 
(see Section~\ref{sec-benchmark}).
We thus consider the semilinear equation
\be{scaleq}
-\Delta u=f(u),\quad x\in\Rn,
\ee
where 
\be{scaleq0}
\hbox{$f:[0,\infty)\to[0,\infty)$ is continuous, $f(0)=0$ and $n\geq3$}.
\ee
Denote $\omega_1=(0,1]$, $\omega_2=(1,\infty)$.

\begin{theorem} \label{thmGSstarCor}
Assume \eqref{scaleq0} and let
\be{hypDefq}
p_1,p_2\in(1,p_S),\quad 0<\bar\kappa<\kappa:=\frac{n}{n-2}.
\ee
Assume that, for any $\e>0$, there exist $c_\e, C_\e>0$ such that
\be{hypGS1a0}
c_\e s^{p_i-(-1)^i\e}\le f(s)\le C_\e s^{p_i+(-1)^i\e},\quad s\in\omega_i,\ i\in\{1,2\}
\ee
and that
\be{hypGS200}
f(s) \le\bar\kappa s^{\kappa-1}\phi(s),\quad  s>0, \qquad\hbox{where }\phi(s):=\int_0^s \sigma^{-\kappa}f(\sigma)d\sigma\le\infty.
\ee
If $u$ is a solution 
of \eqref{scaleq}, then $u\equiv 0$.
\end{theorem} 

\begin{remark} 
 The classical Theorem~A in Section~\ref{intro1} is a special case of Liouville type theorems obtained in \cite{GSa} for more general nonlinearities
of the form $f(x,u)$ (also on Riemannian manifolds). In the Euclidean case with $f=f(u)$,
the Liouville type theorem in \cite[Theorem~6.1]{GSa} requires that  $f\in C^1([0,\infty))$ and that
\be{hypGS61}
\hbox{$s^{-p}f(s)$ is nonincreasing for some $p\in(1,p_S)$}
\ee
(plus some additional assumptions).
Condition \eqref{hypGS61} is stronger than assumption~\eqref{condLZ} of Theorem~B,
and \eqref{hypGS61}  also implies \eqref{hypGS200}  (see Lemma~\ref{compHyp} below).
We refer to Section~\ref{sec-benchmark}  for further comparison on a benchmark example.
\end{remark} 

 Theorem~\ref{thmGSstarCor} is a consequence of the following more general result (but whose formulation is less transparent).

\begin{theorem} \label{thmGSstar}
 Assume~\eqref{scaleq0} and let $p\ge 0$, $q>-p$, with $q\ge -2$ and $n=3$ if $q=-2$.
Assume
\be{f-Lip}
 f(s)\le Cs^p,\quad s\in\omega_1,
\ee
  \be{hypGS20}
s^qf(s) \le c_qF_q(s),\quad\hbox{where } F_q(s):=\int_0^s \sigma^{q-1}f(\sigma)d\sigma<\infty,
\ee
\be{hypGS4}
\exists m_i\in(2/3, 2n/(3n-4)_+),\ s^{(1-(q/2))m_i} F_q^{2m_i-1}(s)\le Cf(s), \quad s\in\omega_i,
\ee
\be{hypGS5}
\alpha>0,\ -\beta+c_q\gamma>0,
\ee
where
\be{alphabetagamma}
\alpha=-\frac{n-1}{n}k^2+(q-1)k-\frac{q(q-1)}{2},\quad
\beta=\frac{n+2}{n}k-\frac{3q}{2}, \quad\gamma=-\frac{n-1}{n}, 
\ee
and $k\ne-1$.
 If $q>-2$ assume in addition that
\be{hypGS3}
\exists \gamma_i\in(1,n/(n-4)_+),\ s^{(2+q)\gamma_i}\le Cf(s)F_q(s), \quad s\in\omega_i.
\ee
If $u$ is a solution 
of \eqref{scaleq}, then $u\equiv 0$.
\end{theorem}

\goodbreak

\begin{lemma} \label{compHyp}
Assume  $f\in C^1([0,\infty))$ and \eqref{hypGS61}. Then \eqref{hypGS200} is true for some $\bar\kappa\in(0,\kappa)$.
\end{lemma}

\begin{proof}
We may assume $p\in(\kappa-1,p_S)$ without loss of generality
  (since \eqref{hypGS61} remains true for any larger $p<p_S$).
We have $\bar\kappa:=p+1-\kappa\in(0,\kappa)$.
We may also suppose that 
\be{cvphi}
\int_0^1 \sigma^{-\kappa}f(\sigma)d\sigma<\infty,
\ee
since otherwise $\phi(s)=\infty$ for all $s>0$ and there is nothing to prove.
Then letting $h(s):=s^{1-\kappa}f(s)-\bar\kappa \phi(s)$ for all $s>0$, we have $h\in C^1(0,\infty)$ and assumption \eqref{hypGS61} implies that 
\be{cvphi2}
h'(s)=s^{1-\kappa}f'(s)+(1-\kappa-\bar\kappa) s^{-\kappa}f(s)=s^{-\kappa}(sf'(s)-pf(s))\le 0,\quad s>0.
\ee
Since \eqref{cvphi} guarantees that $\liminf_{s\to 0}h(s)\le \liminf_{s\to 0}s^{1-\kappa}f(s)\le 0$,
it follows from \eqref{cvphi2} that $h\le 0$ on $(0,\infty)$, hence \eqref{hypGS200}.
\end{proof}

\section{Comparison of methods and further properties} \label{sec-benchmark}

\subsection{Comparison of methods on a benchmark example}
The goal of this subsection is to compare the efficiency of some of the available methods
on a simple example, for which existence and nonexistence depend on the value of  
a parameter.
Namely, for $n\geq3$, $1<n/(n-2)<p<p_S$, and $K>0$, 
we consider the scalar equation
\be{eqfu}
-\Delta u=f(u),\qquad x\in \Rn,
\ee
where 
\be{fuK}
 f(u)=(K+\min(1,u^{p-1}))u^p. 
\ee
Let also $F_q$ be defined by \eqref{hypGS20}.
If
$$K=K_0:=\frac{(n-2)p-n}{2p},$$
then we have the explicit solution
$$u(x)=(1+\lambda |x|^2)^{-1/(p-1)},\qquad \lambda:=\frac{(p-1)^2}{4p}.$$

$\bullet$ Theorem~B implies that
\be{nonex}
\hbox{equation \eqref{eqfu} does not possess nontrivial solutions}
\ee
provided
$s^{-p_S}f(s)$ is nonincreasing and nonconstant, 
i.e.~if 
$$K\geq K_1:=\frac{2((n-2)p-n)}{(n+2)-(n-2)p}.$$

$\bullet$ If $n\leq4$ or $p<\frac{n-1}{n-3}$, then Theorem~\ref{thm1}
implies \eqref{nonex} provided
$sf(s)\leq(p_S+1-\e)F(s)$ for some $\e>0$ (cf.~\eqref{hypGrowthC}),
i.e.~if 
$$K>K_2:=\frac{p+1}{2p}K_1.$$

$\bullet$ Finally, Theorem~\ref{thmGSstarCor} 
implies \eqref{nonex} provided
$$K>K_3:=\frac{(n-2)p-2}{2(n-2)p-n}K_1.$$
Notice that
$$ K_1>K_2>K_3>K_0,$$
hence the modified 
Gidas-Spruck estimates in Theorem~\ref{thmGSstarCor}
yield the best result.
Notice also that $K_3/K_0\to 2$ if $p\to\frac{n}{n-2}+$, but
$K_3/K_0\to\infty$ if $p\to\frac{n+2}{n-2}-$.

\begin{remark} \label{remConjecture}
If $u$ is a solution of \eqref{eqfu}--\eqref{fuK}, then $v(x):=u(\frac1{\sqrt K}x)$
is a solution of \eqref{eqfu} with
\be{fua}
 f(u)=(1+a\min(1,u^{p-1}))u^p,
\ee
where $a=1/K$. Consequently, problem \eqref{eqfu}, \eqref{fua}
has an explicit positive solution if $a=a_0:=1/K_0$, but
it does not possess nontrivial solutions if $a<a_3:=1/K_3$.
This suggests that the equations 
$-\Delta u=u^p\log^a(2+u)$ and $-\Delta u=u^p\log^{-a}(2+u^{-1})$
(which do not possess nontrivial solutions for $a>0$ small due to Theorem~B)
could possess positive solutions for some $a$ large enough.  
\end{remark}

\subsection{Non open range for the Liouville property}
In the case of scale-invariant problems, it is known
that the Liouville property has an open range with respect to the exponent
(see \cite[Proposition 21.2b and subsequent comments]{QSb}).

The following examples show that this is no longer the case for problems without scale invariance. 
Let $n\ge 3$, $\kappa=\frac{n}{n-2}$, $p_k:= \kappa+\frac1k$, $q_k:=2p_k-1=p_S+\frac2k$ and let
$f_k:[0,\infty)\to[0,\infty)$ be given by
$f_k(u)=u^{p_k}+u^{q_k}$, where $k\in\N$ is large. 
Then the equations $-\Delta u=f_k(u)$ in $\R^n$ possess explicit bounded radial solutions
$u_k=u_k(|x|)$ satisfying
$$ u_k^{p_k-1}(r)=\frac{2kp_k}{n-2}\frac{\xi_k^2}{\xi_k^2+r^2},\quad\hbox{where}\quad
\xi_k = \frac{n-2}{k\sqrt{p_k}(p_k-1)}$$
(see \cite{LN88}), but the limiting equation
$-\Delta u=f(u)$ with $f(u)=u^\kappa+u^{p_S}$
does not possess nontrivial solutions by Theorem~B.

 On the other hand, assume that 
$\lim_{\lambda\to 0^+}\frac{f(\lambda s)}{f(\lambda)}=f_0(s)$,
$\lim_{\lambda\to\infty}\frac{f(\lambda s)}{f(\lambda)}=f_\infty(s)$
for all $s>0$ and that the Liouville theorem is true for 
all problems
$-\Delta u=\varphi(u)$ with $\varphi\in\{f,f_0,f_\infty\}$.
If $f_k\to f$ in a suitable way, 
then a contradiction argument (see \cite[Theorem~6.1]{SouDCDS})
does show that
the Liouville theorem is true for $(f_k, (f_k)_0,(f_k)_\infty)$ with $k$ large enough.

Notice also that 
$$ M_k^{p_k-1}:=(\max u_k)^{p_k-1}=\frac{2kp_k}{n-2}\approx k\frac{2n}{(n-2)^2}\to\infty\quad\hbox{as}\ \ k\to\infty,$$
the functions 
$v_k(y):=\frac1{M_k}u_k(M_k^{(1-q_k)/2}y)$ 
converge to a positive solution
of $-\Delta v=h(v)$ with $h(v)=v^{p_S}$, and 
that the proof of Theorem~1 (in particular inequality \eqref{hypGrowthC}) 
guarantees 
the nonexistence of nontrivial solutions of $-\Delta u=f_k(u)$ satisfying
$$u^{p_k-1}\leq\Bigl(\frac{2n}{p_k+1}-(n-2)-\eps\Bigr)\Big/\Bigl(n-2-\frac{2n}{q_k+1}\Bigr)\approx k\frac{2n}{(n-2)(n-1)}.$$
See also \cite{BFdP00,F04,DG16} for results on radial positive solutions
of $-\Delta u=u^p+u^q$ with $p<p_S<q$:
In particular, \cite[Theorem~1.1(b)]{BFdP00} guarantees that if
$p\in(\kappa,p_S)$ is fixed, then the equation $-\Delta u=u^p+u^q$
possesses positive radial solutions for $q\in(p_S,p_S+\eps)$,
which again shows that the Liouville property for this equation
does not have an open range.

\section{Universal bounds}
\label{secUB}

In this section,  for systems with regularly varying nonlinearities, we will give applications of 
 known and new Liouville type theorems to universal bounds for singularity and decay estimates.

 Namely, we shall describe the connections between universal estimates of solutions of an elliptic system
 with given nonlinearity $f$ and the three natural associated Liouville type properties.
 The latter, that we denote by LP, LP$_0$, LP$_\infty$, 
 mean the nonexistence of nontrivial bounded entire solutions for the systems with nonlinearities $f, f_0, f_\infty$, where $f_0, f_\infty$ denote the generalized 
 rescaling limits\footnote{These limits can be defined in a vector-valued way, with a single index of regular variation (see subsection~\ref{SystEqual}), 
 or component-wise, allowing different indices of regular variation (see subsection~\ref{LEsec}).}
  of $f$, respectively near $0$ and $\infty$.

Stated in loose terms, under suitable assumptions, we have the implications:
\be{LPimplications}
\left\{
\begin{aligned}
&\ \hbox{LP$_\infty$}&\Longrightarrow &\ \hbox{ universal singularity estimates of all solutions} \\
&\ \hbox{LP}&\Longrightarrow &\ \hbox{ universal decay without rate for bounded solutions} \\
&\ \hbox{LP$_0$}& \Longrightarrow &\ \hbox{ universal decay rates for small solutions.}
\end{aligned}
\right.
\ee
Here singularity and decay will be expressed in terms of $d^{-2}(x)$, as $d(x)\to 0$ and $\infty$, respectively,
where $d(x)$ is the distance to the boundary
of the domain where the equation is satisfied.\footnote{Of course the decay property is relevant only in the case of unbounded domains
(or expanding families of bounded domains).}
Moreover, the above properties can be suitably combined to yield:
$$
\begin{aligned}
&\hbox{LP $\&$ LP$_0$}&\Longrightarrow &\ \hbox{ universal decay rates for bounded solutions} \\
&\hbox{LP $\&$ LP$_0$ $\&$ LP$_\infty$}& \Longrightarrow &\ \hbox{ universal decay rates for all solutions.}
\end{aligned}
$$

Before going to specific results, we record right away the following general proposition,
which describes the second implication in \eqref{LPimplications} 
(its proof will be rather simple, while the other two will require more assumptions and more work).

\begin{proposition}  \label{propdecay}
Let $f:K\to\R^m$ be continuous, $\Lambda>0$ and assume that the system $-\Delta V=f(V)$
 admits no nontrivial strong solution in $\Rn$ with $\|V\|_\infty\le \Lambda$.
Then 
$\lim_{R\to\infty} \eta(R)=0,$
where\footnote{with the convention $\sup(\emptyset)=0$}
$$
\begin{aligned}
\eta(R)&=\hbox{$\sup\Bigl\{|U(x)|,\ \Omega\subset\Rn$ is open, $U:\Omega\to K$ is a bounded strong solution}\\
&\qquad\qquad\hbox{of $-\Delta U=f(U)$, $\|U\|_\infty\le \Lambda$
and $x\in\Omega$ with ${\rm dist}(x,\partial\Omega)\ge R\Bigr\}$}.
\end{aligned}$$
\end{proposition} 

\goodbreak 
We note that some connections of the above type were described for scalar equations in \cite[Section~4.2]{SouDCDS},
but the extension to systems will require nontrivial additional arguments.
In the next two subsections, we will develop such considerations 
for non scale invariant systems with regularly varying nonlinearities, either with different indices on the
components (Lane-Emden type systems), or with equal indices.

\subsection{Lane-Emden type systems}
\label{LEsec}

We here study systems where the nonlinearities have regular variation with different indices on the components.
For simplicity we restrict ourselves to the following non-scale invariant Lane-Emden type systems
(more general systems could be considered): 
\be{LEsystem1}
\left\{
  \begin{aligned}
 \hfill -\Delta u&=f_1(v), \\ 
  \noalign{\vskip 1mm}
 \hfill -\Delta v&=f_2(u),   
   \end{aligned}
   \right.
\ee
which is a modification of the classical Lane-Emden system
\be{LEsystem2}
\left\{
  \begin{aligned}
 \hfill -\Delta u&=v^{p}, \\ 
 \noalign{\vskip 1mm}
 \hfill -\Delta v&=u^{q}.   
   \end{aligned}
   \right.
\ee
For $p, q>0$, we consider nonlinearities $f=(f_1,f_2)$ satisfying the following, slightly reinforced, regular variation assumptions:
\be{hypLE1}
\hbox{$f_i$ are continuous on $[0,\infty)$, $f_i$ are $C^1$ and positive for $s>0$ large,}
\ee
\be{hypLE2}
\hbox{$L_1(s):=s^{-p} f_1(s)$, $L_2(s):=s^{-q} f_2(s)$ satisfy }
\lim_{s\to\infty} {sL'_i(s)\over L_i(s)}=0.
\ee
Note that \eqref{hypLE1}-\eqref{hypLE2} imply that $f_1, f_2$ have regular variation at $\infty$ with respective indices~$p, q$
(see \cite[the paragraph containing (1.11)]{Se}).
For Lane-Emden type systems, the formulation of the universal bounds involves a suitable interaction between the components $f_1, f_2$. 
To this end we introduce the auxiliary functions:
$$h_i(s)=sf_i(s),\quad s\ge 0.$$
One can check (see \eqref{infphi1})
that there exists $s_0>0$ such that 
$h_1, h_2$ are continuous positive and increasing on $[s_0,\infty)$. 
Taking larger $s_0$ if necessary, we can thus ensure that all functions
\be{h1h2mon}
  \begin{aligned}
&\hbox{$h_1$, $h_2$, $h_1^{-1}$, $h_2^{-1}$, $h_1^{-1}\circ h_2$ and $h_2^{-1}\circ h_1$}\\
&\hbox{are defined, positive and increasing on $[s_0,\infty)$.}
   \end{aligned}
\ee

\begin{theorem} \label{thmLEpert1}
Assume \eqref{hypLE1}-\eqref{hypLE2}, where $p, q>0$ with $pq>1$ are such that system \eqref{LEsystem2}
 admits no nontrivial solution on $\Rn$.

\smallskip

 (i) There exists $C=C(n,f)>0$ such that, if $\Omega$ is an arbitrary domain of $\Rn$
and $(u,v)$ is a 
solution of~\eqref{LEsystem1} in $\Omega$, then we have the estimates
\be{estimLE1}
  \begin{aligned}
{f_2(u)\over (h_1^{-1}\circ h_2)(u)} &\le C(1+d^{-2}(x)),\quad x\in \Omega_1,\\
{f_1(v)\over (h_2^{-1}\circ h_1)(v)} &\le C(1+d^{-2}(x)),\quad x\in \Omega_2,
  \end{aligned}
\ee
where $\Omega_1=\{x\in\Omega, u\ge s_0\}$, $\Omega_2=\{x\in\Omega, v\ge s_0\}$
and $d(x)={\rm dist}(x,\partial\Omega)$.

\smallskip

 (ii) Let $\Omega\subset\R^n$ be a uniformly regular domain of class $C^2$.
There exists a constant $C=C(\Omega,f)>0$
such that, for any 
solution $(u,v)$ of~\eqref{LEsystem1} in $\Omega$
with $u=v=0$ on $\partial\Omega$, we have 
$u, v\le C$ in $\Omega$.
\end{theorem}

\begin{remark} \label{remLE}
(i) The optimal condition on $p, q>0$ for nonexistence of nontrivial solutions of \eqref{LEsystem2} on $\Rn$
is conjectured to be given by the so-called Lane-Emden hyperbola:
\be{hypLE3}
{1\over p+1}+{1\over q+1}>{n-2\over n}
\ee
or equivalenty $\alpha+\beta>n-2$ 
with $\alpha={2(p+1)\over pq-1}$, $\beta={2(q+1)\over pq-1}$ in the ``superlinear'' case $pq>1$.
So far the conjecture has been fully established only for $n\le 4$.
While the necessity of condition \eqref{hypLE3} is known in any dimension, 
nonexistence for $n\ge 5$ is known only under some additional restrictions on $p,q$.
For example, simple sufficient conditions in any dimension are given by
$p,q\le p_S$ with $(p,q)\ne (p_S,p_S)$, or $\max(\alpha,\beta)\ge n-2$.
See \cite{SouAdv} and \cite[Section 31.2]{QSb} for details.

\smallskip

(ii) In the special case of the Lane-Emden system ($f_1(v)=v^p$, $f_2(u)=u^q$)
and $p,q>1$, Theorem~\ref{thmLEpert1} was proved in 
\cite{PQS1}.
Note that in that case, estimate \eqref{estimLE1} can be rewritten as
$$u^{1/\alpha}+v^{1/\beta}\le C(1+d^{-1}(x)),\quad x\in \Omega.$$
\end{remark}

\begin{example} \rm
Consider the case of nonlinearities with logarithmic behavior at infinity:
$$f_1(v)=v^p\log^a(K+v),\quad f_2(u)=u^q\log^b(K+u),$$
with $a, b\in \R$ and $K\ge 1$.
Then Theorem~\ref{thmLEpert1}(i) applies and estimate \eqref{estimLE1} yields
$$u^{pq-1\over p+1}\log^k{\hskip -2pt}(1+u)+v^{pq-1\over q+1}\log^\ell{\hskip -2pt}(1+v)\le C(1+d^{-2}(x)),\quad x\in \Omega,$$
with $k=\frac{bp+a}{p+1}$, $\ell=\frac{aq+b}{q+1}$.
\end{example}

We next give a counterpart of Theorem~\ref{thmLEpert1} for decay estimates.
For $p, q>0$, we consider nonlinearities $f_1,f_2$ satisfying 
\eqref{hypLE1}, \eqref{hypLE2} as $s\to 0$ (instead of $s\to\infty$).
Then $f_1, f_2$ have regular variation at $0$ with respective indices~$p, q$,
and there exists $\eps_0>0$ small such that \eqref{h1h2mon} is true on $I=(0,\eps_0]$.

\begin{theorem} \label{thmLEpert1B}
(i) Assume that \eqref{hypLE1}, \eqref{hypLE2} are satisfied as $s\to 0$ (instead of $s\to\infty$),
where $p,q>0$ with $pq>1$ are such that the system $-\Delta U=\varphi(U)$
 admits no nontrivial bounded solution on $\Rn$ for $\varphi\in\{f,(s^p,s^q)\}$.
For any $\Lambda>0$, there exist constants $C,D>0$ depending only on $n,f,\Lambda$ such that, 
if $\Omega$ is an arbitrary open set of $\Rn$
and $U=(u,v)$ is a bounded solution of~\eqref{LEsystem1} in $\Omega$ with $\|U\|_\infty\leq\Lambda$, then 
\be{estimLE1BD}
u,v\le\eps_0 \quad\hbox{in }\tilde\Omega_D:=\bigl\{x\in\Omega;\ d(x):={\rm dist}(x,\partial\Omega)>D\bigr\},
\ee
with \eqref{h1h2mon} being satisfied on $I=(0,\eps_0]$,
and we have the estimates
\be{estimLE1B}
{f_2(u)\over (h_1^{-1}\circ h_2)(u)}+ {f_1(v)\over (h_2^{-1}\circ h_1)(v)} \le Cd^{-2}(x),\quad x\in \tilde\Omega_D.
\ee
(ii) Assume in addition that the hypotheses of Theorem~\ref{thmLEpert1}(i) are also satisfied,
with possibly different $p, q$. Then \eqref{estimLE1BD}, \eqref{estimLE1B} 
remain true for any (possibly unbounded) 
 solution of~\eqref{LEsystem1} in $\Omega$,
with $C, D$ depending only on $n,f$.
\end{theorem}

\begin{example} \rm
Consider the case of nonlinearities with logarithmic behavior at zero:
$$f_1(v)=v^p\log^a(K+v^{-1}),\quad f_2(u)=u^q\log^b(K+u^{-1}).$$
Assume $p,q\in(1,p_S)$, $a, b\in \R$, $K\ge 1$, $a\ge -\theta(K)(p_S-p)$,
and $b\ge -\theta(K)(p_S-q)$,
where $\theta(K)\ge 1$ is defined in Corollary~\ref{thm1cor0}.
Assume also $a<p-1$ and $b<q-1$ if $K=1$.
Then the conclusion of Theorem~\ref{thmLEpert1B}(ii) is true,
which leads to the estimate
$$u^{pq-1\over p+1}\log^k{\hskip -2pt}(1+u^{-1})+v^{pq-1\over q+1}\log^\ell{\hskip -2pt}(1+v^{-1})\le Cd^{-2}(x),\quad x\in \tilde\Omega_D,$$
with $k=\frac{bp+a}{p+1}$, $\ell=\frac{aq+b}{q+1}$.

Indeed, using \eqref{defhKphi}, \eqref{infsuphK} and \eqref{gprimehK0}, 
we see that, for $\varphi\in\{f,(s^p,s^q)\}$, the functions $s^{-p_S}\varphi_i(s)$ are decreasing on $(0,\infty)$,
so that the Liouville type assumptions are satisfied owing to \cite{RZ}.
\end{example}

\subsection{Systems with nonlinearities having regular variation with equal indices on the
components.}

\label{SystEqual}

In this subsection we will assume $f\in C(K)$, where $K:=[0,\infty)^m$,
 and we will denote $|X|=\max\{|X_1|,\dots,|X_N|\}$ for $X\in\R^N$ with $N=n$ or $N=m$.
In order to state our universal bounds for non scale invariant systems that can be deduced from Liouville type theorems,
we introduce the function
$$f^+:[0,\infty)\to[0,\infty):\lambda\mapsto\max\limits_{|U|=\lambda}|f(U)|.$$ 
This allows us to consider a suitable notion of regularly varying, vector valued functions,
whose generalized
rescaling limits have to be homogeneous (vector valued) functions, similarly as in the scalar case (see Lemma~\ref{lem-theta} below).

\begin{theorem} \label{thmUB}
Let $\Omega\subset\R^n$ be an arbitrary domain, $d(x):=\hbox{\rm dist}(x,\partial\Omega)$,
$f\in C(K,\R^m)$, $\Lambda>0$,
and let $U:\Omega\to K$ be a 
solution of $-\Delta U=f(U)$ in $\Omega$.
\smallskip

(i) Assume $f^+(\lambda)>0$ for $\lambda\geq\Lambda$ and
\begin{equation} \label{limh} 
  \lim_{\lambda\to\infty}\frac{f(\lambda \xi)}{\+f(\lambda)}=  f_\infty(\xi)
  \quad\hbox{ locally uniformly for }\xi\in K,
\end{equation}
where $f_\infty$ is homogeneous of order $p>1$.
Assume also that the problem $-\Delta V= f_\infty(V)$
does not possess nontrivial bounded entire solutions $V:\R^n\to K$.
Set $\Omega_\Lambda:=\{x\in \Omega:|U(x)|\geq\Lambda\}$.
Then there exists a constant $C=C(n,f,\Lambda)$ (independent of $\Omega$ and $U$) 
such that 
\begin{equation} \label{UBh}
 |f(U(x))|\leq C|U(x)|(1+d^{-2}(x))
\quad\hbox{for all } x\in \Omega_\Lambda.
\end{equation}
\smallskip

(ii) Assume $f^+(\lambda)>0$ for $\lambda\in(0,\Lambda]$ and
\begin{equation} \label{limg} 
  \lim_{\lambda\to0+}\frac{f(\lambda \xi)}{\+f(\lambda)}=  f_0(\xi)
  \quad\hbox{ locally uniformly for } \xi\in K,
\end{equation}
where $f_0$ is homogeneous of order $q>1$.
Assume also that problems $-\Delta V=\varphi(V)$
do not possess nontrivial bounded entire solutions $V:\R^n\to K$
for $\varphi\in\{f, f_0\}$. Let $\|U\|_\infty\leq\Lambda$.
Then there exists a constant $C=C(n,f,\Lambda)$ (independent of $\Omega$ and $U$) 
such that 
\begin{equation} \label{UB}
 |f(U(x))|\leq C|U(x)|d^{-2}(x)
\quad\hbox{for all } x\in \Omega.
\end{equation}
\smallskip

(iii)
Assume $\+f(\lambda)>0$ for $\lambda>0$.
Let \eqref{limh} and \eqref{limg} be true,
where $f_\infty$ and $f_0$ is homogeneous of order $p>1$ and $q>1$, respectively.
Assume also that problems $-\Delta V=\varphi(V)$ do not possess 
nontrivial bounded entire solutions $V:\R^n\to K$
for $\varphi\in\{f, f_0, f_\infty\}$.
Then there exists a constant $C=C(n,f)$ (independent of $\Omega$ and $U$) such that 
\eqref{UB} is true.
\end{theorem}

\begin{remark}
Theorem~\ref{thmUB}(iii) extends Theorem~C to systems.
More precisely, Theorem~C follows from Theorem~\ref{thmUB}(iii)
with $m=1$ by using the classical Liouville theorem in \cite{GSa} (and Lemma~\ref{LemUnifConv0} below).
\end{remark}

\begin{remark}
(A priori estimate for the Dirichlet problem)
Under the hypotheses of Theorem~\ref{thmUB}(i), assume in addition that $\Omega$ is uniformly regular of class $C^2$,
that $U\in C(\overline\Omega)$ satisfies the Dirichlet boundary conditions $U=0$ on $\partial\Omega$,
and that the Liouville property is also true in the half-space (i.e., the problem $-\Delta V= f_\infty(V)$
does not possess nontrivial bounded entire solutions $V:\R^n_+\to K$ with $V=0$ on $\partial\R^n_+$).
Then there exists a constant $C=C(n,f,\Omega)$ (independent of $U$) 
such that $|U(x)|\leq C$ for all $x\in \Omega$.
This follows by modifying the proof of Theorem~\ref{thmUB}(i)
along the lines of the proof of \cite[Theorem~4.1]{PQS2}.
\end{remark}

\begin{remark}
The third implication in \eqref{LPimplications} can be illustrated by the following variant of Theorem~\ref{thmUB}(ii).
Assume that $f^+(\lambda)>0$ for $\lambda>0$ small and
that \eqref{limg} is satisfied, where $f_0$ is homogeneous of order $q>1$.
Assume also that the problem $-\Delta V=f_0(V)$
does not possess nontrivial bounded entire solutions $V:\R^n\to K$.
Then there exist constants $\eps_0,C>0$, depending only on $n,f$, 
such that \eqref{UB} holds in $\hat\Omega=\{x\in\Omega; |U(x)|\le \eps_0\}$
(hence in particular in $\Omega$ for small solutions, namely if $\|U\|_\infty\le\eps_0$).
This follows from straightforward modifications
of the proof of Theorem~\ref{thmUB}(ii), where only the case $m_k\to 0$ will occur.
\end{remark}

\smallskip

We now give some examples illustrating Theorem~\ref{thmUB}.

\begin{example} \rm
 Consider the noncooperative system $-\Delta U=f(U)$ with $U=(u,v)$ and nonlinearity
\be{DefEx1}
f(U)=\bigl(2g'(u)[g(u)-\lambda g(v)],2g'(v)[g(v)-\lambda g(u)]\bigr),
\quad g(s)=s^{\frac{p+1}{2}}\log^a(K+s),
\ee
 with $p\in (1,p^*)$,  $K\ge 1$, $a\in(0,a_0)$ and $\lambda\in(0,\lambda_0)$, where $a_0:={p^*-p\over 2} \theta(K)$,
$\lambda_0={2\sqrt{1-\rho}\over 2-\rho}$, $\rho={a\over a_0}$, and $p^*$  and $\theta(K)$ are defined in \eqref{defpstar} 
and in Corollary~\ref{thm1cor0},
respectively.
Then Theorem~\ref{thmUB}(iii) guarantees the existence of a constant $C=C(n,f)$ such that 
\eqref{UB} is true for any domain $\Omega\subset\Rn$ and any solution $U: \Omega\to K$ of $-\Delta U=f(U)$ in $\Omega$,
 which yields
\be{EstEx1}
|U|^{p-1}\log^{2a}(K+|U|)\le C(n,f)d^{-2}(x),\quad x\in\Omega.
\ee

 Indeed, using the fact that $g',g''\ge 0$ (cf.~Lemma~\ref{lemhK}), we easily see that
$$\+f(t)=2g(t)g'(t),\quad f_\infty(U)=
\bigl(u^p-\lambda u^{\frac{p-1}{2}}v^{\frac{p+1}{2}},v^p-\lambda u^{\frac{p+1}{2}}v^{\frac{p-1}{2}}\bigr),$$
$$
f_0(U)=
\begin{cases}
f_\infty(U)&\quad\hbox{ if $K>1$,}\\ 
\noalign{\vskip 1mm}
\bigl(u^{p+2a}-\lambda u^{\frac{p-1}{2}+a}v^{\frac{p+1}{2}+a},v^{p+2a}-\lambda u^{\frac{p+1}{2}+a}v^{\frac{p-1}{2}+a}\bigr)
&\quad\hbox{ if $K=1$,}
\end{cases}
$$
and the systems  $-\Delta V=\varphi(V)$ do not possess 
nontrivial bounded entire solutions $V:\R^n\to K$
for $\varphi=f$ by Corollary~\ref{thm1cor0}, and for $\varphi\in\{f_0,f_\infty\}$ by Theorem~\ref{thm1} (or \cite[Theorem~3]{QS-CMP}).

In the case of nonlinearities with logarithmic behavior at $0$, namely
$g(s)=s^{\frac{p+1}{2}}\log^{-a}(K+s^{-1})$ in \eqref{DefEx1} (the other assumptions being unchanged),
estimate \eqref{EstEx1} is replaced by
$$|U|^{p-1}\log^{-2a}(K+|U|^{-1})\le C(n,f)d^{-2}(x),\quad x\in\Omega.$$

 On the other hand, the logarithms in this example could be replaced, under suitable assumptions
on the parameters,
 by some more general slowly varying functions (cf.~Remark~\ref{remSlow}).
\end{example}

\begin{example} \label{cubic-quintic} \rm
 Consider the system
 $-\Delta U=f(U)$ with $U=(u,v)$ and nonlinearity
 $f=\nabla F$ with $F= H_1+H_2$,
$$H_1(U)=\frac{1}{p+1}\bigl(u^{p+1}+v^{p+1}+2\lambda u^{\frac{p+1}{2}} v^{\frac{p+1}{2}}\bigr)$$
and
$$H_2(U)=\frac{b}{q+1}\bigl(u^{q+1}+v^{q+1}+2\mu u^{\frac{q+1}{2}}v^{\frac{q+1}{2}}\bigr),$$
 where $1<p<q\le p_S$, $\lambda>-1$, $\mu\ge-1$, $b>0$.
 If $n\ge 5$ and $\min(\lambda,\mu)<0$, we assume in addition $p<(n-1)/(n-3)$.
If $q<q_S$  and $\mu>-1$, 
then Theorem~\ref{thmUB}(iii) guarantees the existence of a constant $C=C(n,f)$ such that 
\eqref{UB} is true for any domain $\Omega\subset\Rn$ and any 
solution $U:\Omega\to K$ of $-\Delta U=f(U)$ in $\Omega$,
 which yields
\be{boundUpq}
 |U|^{p-1}+|U|^{q-1}\le C(n,f)d^{-2}(x),\quad x\in\Omega.
\ee
If $q=q_S$ then, for each $\Lambda>0$, 
Theorem~\ref{thmUB}(ii) still guarantees the same conclusion with $C=C(n,f,\Lambda)$
provided we consider bounded solutions with $\|U\|_\infty\leq\Lambda$.

 Indeed, elementary computations show that $\+f(\lambda)>0$ for $\lambda>0$, with
$$\lim_{t\to 0^+}t^{-p}{\+f(t)}=c_1>0,\quad\
\lim_{t\to\infty}t^{-q}{\+f(t)}=c_2>0,$$
and that $f_0(U)=c_1^{-1}\nabla H_1$, $f_\infty(U)=c_2^{-1}\nabla H_2$.
Moreover, the systems  $-\Delta V=\varphi(V)$ do not possess 
nontrivial bounded entire solutions $V:\R^n\to K$
for $\varphi\in\{f, f_0\}$ by Corollary~\ref{thm1cor1}.
f $q<q_S$ and $\mu>-1$, 
hen the system $-\Delta V= f_\infty(V)$ does not possess 
nontrivial bounded entire solutions $V:\R^n\to K$, 
owing to Corollary~\ref{thm1cor1}. 

 Finally, if $H_2$ is replaced
with $H_2(U)=bu^{\frac{q+1}{2}}v^{\frac{q+1}{2}}$ with $q\in(p,p_S]$ 
and $b>0$ then,  for each $\Lambda>0$, Theorem~\ref{thmUB}(ii)  guarantees \eqref{boundUpq} with $C=C(n,f,\Lambda)$
for any domain $\Omega\subset\Rn$ and any bounded 
solution $U:\Omega\to K$ of $-\Delta U=f(U)$ in $\Omega$
such that $\|U\|_\infty\leq\Lambda$.
\end{example}

\begin{example} \label{ex43}
\rm 
 Consider the system $-\Delta U=f(U)$ with $U=(u,v)$ and nonlinearity
$$f(U)=\bigl((u^r+bu^q)(v^p-\lambda u^p),(v^r+bv^q)(u^p-\lambda v^p)\bigr),$$
where $p>1$, $0\le r<q\le p$ with $p+q\le p_S$, $b>0$ and $\lambda\in[0,1)$.
If $p+q<p_S$ and $\lambda\in(0,1)$, then Theorem~\ref{thmUB}(iii) 
guarantees the existence of a constant $C=C(n,f)$ such that 
\eqref{UB} is true for any domain $\Omega\subset\Rn$ and any  
solution $U:\Omega\to K$ of $-\Delta U=f(U)$ in $\Omega$,
which yields 
$$|U|^{p+r-1}+|U|^{p+q-1}\le C(n,f)d^{-2}(x),\quad x\in\Omega.$$
If  $p+q=p_S$ and $\lambda\in(0,1)$, or $\lambda=r=0$ then, for each $\Lambda>0$, 
Theorem~\ref{thmUB}(ii) yields $|U|\le C(n,f,\Lambda)d^{-2/(p-1)}(x)$ for all $x\in\Omega$,
provided we consider bounded solutions with $\|U\|_\infty\leq\Lambda$.

 Indeed, one can check that $\+f(t)>0$ for $t>0$, with
$$\lim_{t\to 0^+}t^{-(p+r)}{\+f(t)}=c_1>0,\quad\
\lim_{t\to\infty}t^{-(p+q)}{\+f(t)}=c_2>0,$$
and that 
$$f_0(U)=c_1^{-1}\bigl(u^r(v^p-\lambda u^p),v^r(u^p-\lambda v^p)\bigr),\quad
f_\infty(U)=bc_2^{-1}\bigl(u^q(v^p-\lambda u^p),v^q(u^p-\lambda v^p)\bigr).$$
Moreover, it follows from Theorem~\ref{thm-proportional2} that the systems  $-\Delta V=\varphi(V)$ do not possess 
nontrivial bounded entire solutions $V:\R^n\to K$
for $\varphi\in\{f,f_0,f_\infty\}$ if  $p+q<p_S$ and $\lambda\in(0,1)$, and for $\varphi\in\{f,f_0\}$ if 
 $p+q=p_S$ and $\lambda\in(0,1)$, or $\lambda=r=0$.
\end{example}

\begin{example} \label{ex44}
 \rm
Consider the system $-\Delta U=f(U)$ with $U=(u,v)$ and nonlinearity
$$f(U)=\bigl(\phi(u,v)k(u)(g(v)-\lambda g(u)),\phi(u,v)k(v)(g(u)-\lambda g(v))\bigr).$$
where
$$k(s)=s^r\log^a(K+s),\quad g(s)=s^p\log^b(K+s),\quad \phi(u,v)=\log^{-d}(K+u+v),$$
with $0\le r\le p$, $1<p+r<p_S$, $0\le a\le b$, $d\ge a+b$, $K>1$,
and $\lambda\in(0,1)$ or $\lambda=r=0$.
Theorem~\ref{thmUB}(iii) then guarantees the existence of a constant $C=C(n,f)$ such that 
\eqref{UB} is true for any domain $\Omega\subset\Rn$ and any 
solution $U:\Omega\to K$ of $-\Delta U=f(U)$ in $\Omega$,
which yields 
$$|U|^{p+r-1}\log^{a+b-d}(K+|U|)\le C(n,f)d^{-2}(x),\quad x\in\Omega.$$
Indeed, one can check that $\+f(t)>0$ for $t>0$, with
$$\lim_{t\to 0^+}t^{-(p+r)}{\+f(t)}=c_1>0,\quad\
\lim_{t\to\infty}t^{-(p+r)} (\log t)^{d-a-b} {\+f(t)}=c_2>0,$$
and that 
$$f_0(U)=c_1^{-1}\bigl(u^r(v^p-\lambda u^p),v^r(u^p-\lambda v^p)\bigr),\quad
f_\infty(U)= c_1c_2^{-1}f_0(U).$$
Moreover, it follows from Theorem~\ref{thm-proportional2} that the systems  $-\Delta V=\varphi(V)$ do not possess 
nontrivial bounded entire solutions $V:\R^n\to K$
for $\varphi\in\{f,f_0,f_\infty\}$.

 Again, the logarithms in this example could be replaced, under suitable assumptions
on the parameters,
 by some more general slowly varying functions (cf.~Remark~\ref{remSlow}).
\end{example}

\section{Proof of Theorem~\ref{thm1}}

\label{ProofPoh}

\subsection{A conditional Liouville theorem.}
Theorem~\ref{thm1} will be obtained as a consequence of the following conditional Liouville theorem,
which is an extension of \cite[Theorem~2]{SouNHM}.
The necessary a priori integral estimate \eqref{BasicEstAnews} to apply Theorem~\ref{thm3}
will be given in Subsections \ref{subAEs1}-\ref{subAEs2} below.
 In what follows we set $B_R=\{x\in\R^n : |x|<R\}$.

\begin{theorem} \label{thm3}
Let $p_0\in(1,p^*)$, $s\ge 1$ and assume
\be{hypsp}
s>{(n-1)(p_0-1)\over 2}.
\ee
There exists $\e_0=\e_0(n,p_0,s)>0$ 
such that, if 
$$p_0-\e_0\le q\le p\le p_0+\e_0,\qquad p\ge s,\qquad 0\le \delta_0\le \e_0,$$
\eqref{hypGrowthB}-\eqref{hypGrowthC} are satisfied and $U\geq 0$ is
 a bounded solution of \eqref{GenSystem} or \eqref{GenSystemB} verifying 
\be{BasicEstAnews}
\int_{B_R} |U|^s\leq CR^{n-{2s\over p-1}+\delta_0},\qquad R\geq 1, 
\ee
 then $U\equiv 0$.
 \end{theorem} 
 
\goodbreak

\begin{remark} \label{rem3.1}
The quantity $\delta_0$ 
in \eqref{BasicEstAnews} allows for a possible ``defect of scale invariance'' caused by 
the gap between $p$ and $q$ in assumptions \eqref{hypGrowthB}-\eqref{hypGrowthC}
(see \eqref{Hypuab} below). 
As for the restrictions $s\le p$ and $s\ge 1$, they could be relaxed at the expense of further complications in the proof,
using additional arguments from \cite{SouNHM}. We have refrained from expanding on this since we always have $s\in[1,p]$ in our applications.
Also, as in \cite{SouNHM}, we could handle solutions with some exponential growth,
provided we assume \eqref{hypGrowthC} with $c_M$ independent of $M$, 
 and \eqref{Hypfa0} instead of \eqref{hypGrowthB},
instead of just bounded solutions.
\end{remark}

The proof of Theorem~\ref{thm3} is long and technical.
 It relies on an extension of a feedback procedure from \cite{SouAdv, QS-CMP, SouNHM}.
More precisely, by means of a Rellich-Pohozaev type identity, for any $r>0$, the volume integral 
$$H(r):=\int_{|x|<r}|U|^{p+1}\,dx$$
can be controled by surface integrals involving $|U|^{p+1}$, $|\nabla U|^2$ and $|U|^2$ at $|x|= r$. Fix $R>0$. If, for {\it some} $r\in \bigl(R,2R\bigr)$, one can estimate these surface terms by $CR^{-a}H^b(3R)$ with $C,a>0$, $b\in [0,1)$ independent of $R$, then one easily infers that $H(R)$ has very fast growth as $R\to\infty$, hence contradicting the boundedness assumption on $U$. It turns out that such estimation can be achieved by a careful analysis using Sobolev imbeddings and interpolation inequalities on $S^{n-1}$, elliptic estimates and a measure argument.
In this scheme, the primary integral bound \eqref{BasicEstAnews} is used in the interpolation step, where the surface terms are interpolated between 
various auxiliary norms.  Heuristically, the efficiency of the method comes from the fact that one space dimension is ``gained'' via the Pohozaev type identity because, by applying functional analytic arguments on the $n-1$ dimensional unit sphere rather than directly on $B_R$,  one can bootstrap from \eqref{BasicEstAnews} under less stringent growth restrictions on the nonlinearity.
Related arguments, without interpolation and feedback, and restricted to $n=3$, first appeared in \cite{SZ}.
Both \cite{SZ} and \cite{SouAdv} were exclusively concerned with the Lane-Emden system 
$-\Delta u=v^p$, $-\Delta v=u^q$. 
Although many parts are similar to the proofs in \cite{SouAdv, QS-CMP, SouNHM}, it requires many nontrivial changes and 
we give a complete proof for self-containedness.

\subsection{Notation and preliminaries.}
 For $R>0$, we denote
$$B_R=\{x\in\Rn;\,|x|<R\},\quad B_R^+=\{x\in B_R,\ x_n>0\},\quad B_R^0=\{x\in B_R,\ x_n=0\},$$
$$S_R=\{x\in\Rn;\,|x|=R\},\quad S_R^+=\{x\in S_R,\ x_n>0\}$$
and
$$S^{n-1}=\{x\in\Rn;\,|x|=1\},\quad S_+^{n-1}=\{x\in\Rn;\,|x|=1,\ x_n>0\}.$$
We then set
 $$
D_R=\begin{cases}
B_R &\\ 
\noalign{\vskip 2mm}
B_R^+ &\\ 
\end{cases}
\quad
\Sigma_R=
\begin{cases}
S_R &\\ 
\noalign{\vskip 2mm}
S_R^+ &\\
\end{cases}
\quad
\Sigma=\Sigma_1=
\begin{cases}
S^{n-1}, &\quad\hbox{ in case of problem \eqref{GenSystem},}\\ 
\noalign{\vskip 2mm}
S^{n-1}_+, &\quad\hbox{ in case of problem \eqref{GenSystemB}.}\\ 
\end{cases}
$$

We shall use the spherical coordinates $(r,\theta)$ with $r=|x|$, $\theta=x/|x|\in \Sigma$
(for $x\neq 0$). For a given function $w$ of $x\in\Rn$, we write $w(x)=w(r,\theta)$
(using the same symbol $w$, without risk of confusion).
 The norm in $L^\lambda$ will be denoted by $\|\cdot\|_\lambda$.
For brevity we will use the following notation
for volume and surface integrals
$$\IR w=\IR w(x)\, dx,\qquad
\IS w(R)=\IS w(R,\theta)\, d\theta.$$

Our first two lemmas provide Sobolev and interpolation inequalities on $\Sigma$
and scale invariant elliptic estimates, respectively, which will be used repeatedly in the proof of Theorem~\ref{thm3}.
See \cite{SouNHM} for $\Sigma=S^{n-1}$ and $D_R=B_R$. The case $\Sigma=S^{n-1}_+$ and $D_R=B_R^+$ can be treated similarly.

\begin{lemma} \label{LemFive}
\smallskip
(i) Let $n\geq 2$, $j\geq 1$ be integers, let $1<z<\infty$ and let $\lambda$ satisfy
$$\left\{\begin{aligned}
\displaystyle{1\over z}-{1\over\lambda}\le {j\over n-1}, \hfill&\ \hbox{ if } z<(n-1)/j,\hfill\\ 
\noalign{\vskip 1mm}
1<\lambda<\infty, \hfill&\ \hbox{ if } z=(n-1)/j,\hfill\\ 
\noalign{\vskip 2mm}
\lambda=\infty, \hfill&\ \hbox{ if } z>(n-1)/j.\hfill\\ 
\end{aligned}\right.$$
Then, for any $w=w(\theta)\in W^{j,z}(\Sigma)$, we have
$$\|w\|_\lambda\leq C(\|D^j_\theta w\|_z+\|w\|_1).$$
\smallskip
(ii) Let $n=2$ and $1\leq z<\infty$.
For any $w=w(\theta)\in W^{2,z}(\Sigma)$, we have
$$\|w\|_\infty\leq C\bigl(\|w\|_1+\|D^2_\theta w\|_1\bigr)^{1/(z+1)}\|w\|_z^{z/(z+1)}.$$
\smallskip
(iii) Let $n=2$. For any $w=w(\theta)\in W^{1,1}(\Sigma)$, we have
$$\|w\|_\infty\leq C\inf_\Sigma |w|+C\|D_\theta w\|_1.$$  
\end{lemma}

\begin{lemma} \label{LemSix} 
Let $1<z<\infty$ and $\delta>0$. 
There exists $C>0$ such that for all $R>0$ and $v=v(x)\in W^{2,z}(D_{(1+\delta)R})$,
with $v=0$ on $B^0_{(1+\delta)R}$
in case $D_R=B_R^+$, we have
$$\int_{D_R}|D^2_xv|^z+R^{-z}\int_{D_R}|D_xv|^z\leq 
C\Bigl(\int_{D_{(1+\delta)R}}|\Delta v|^z+R^{-2z}\int_{D_{(1+\delta)R}}|v|^z\Bigr).$$
\end{lemma}

The following Rellich-Pohozaev type identity plays a key role
in the proof of Theorem~\ref{thm3}. It is essentially known (see \cite{PuS}, and cf.~also \cite{QS-CMP}, for the whole space case).
We give the proof in the half-space case for convenience.

\begin{lemma} \label{LemFour}
Assume \eqref{hypGradSyst}. Then, for any solution $U$ of \eqref{GenSystem}
and any $R>0$, there holds
$$\begin{aligned}
\IR \bigl(2nF(U)-(n-2)&U\cdot\nabla F(U)\bigr) = 2R^n\IS F(U(R)) \\ 
&+R^n\sum_i d_i\IS \Bigl(|\partial_r u_i|^2-|\partial_\tau u_i|^2+(n-2)R^{-1}u_i\partial_ru_i\Bigr)(R),
\end{aligned}$$
where $\partial_r$ and $\partial_\tau$ respectively denote the components of the gradient along $x/|x|$ and along its orthogonal complement,
 and $\partial_\tau u_i:=0$ if $n=1$.
\end{lemma}

\begin{proof}
We first claim that
\be{VarIdent}
2nF(U)-(n-2) U\cdot \nabla F(U)=\nabla\cdot\Psi,
\ee
where 
$$\Psi=2xF(U)+\sum_i d_i\bigl(2(x\cdot\nabla u_i)\nabla u_i-x|\nabla u_i|^2+(n-2)u_i\nabla u_i\bigr).$$
To show \eqref{VarIdent}, we start from the well-known identity (see, e.g., \cite[p.18]{QSb})
$$(x\cdot\nabla u_i)\Delta u_i-{n-2\over 2}|\nabla u_i|^2
=\nabla\cdot\Bigl((x\cdot\nabla u_i)\nabla u_i-{x\over 2}|\nabla u_i|^2\Bigr),$$
which implies
$$\begin{aligned}
nF(U)-\nabla\cdot\bigl(xF(U))
&=-x\cdot\nabla(F(U))
=-\sum_{i,j} x_j{\partial u_i\over \partial x_j}{\partial F\over \partial u_i}(U)
=\sum_i (x\cdot\nabla u_i)d_i\Delta u_i\\ 
&={n-2\over 2}\sum_i d_i|\nabla u_i|^2
+\nabla\cdot\sum_i d_i\Bigl((x\cdot\nabla u_i)\nabla u_i-{x\over 2}|\nabla u_i|^2\Bigr).
\end{aligned}$$
On the other hand, we have
$$U\cdot \nabla F(U)=-\sum_i d_iu_i\Delta u_i
=\sum_i d_i|\nabla u_i|^2-\nabla\cdot\Bigl(\sum_i d_iu_i\nabla u_i\Bigr).
$$
Multiplying the last identity by $(n-2)$ and subtracting to twice the previous one, we get~\eqref{VarIdent}.

Now, on $\partial\R^n_+$, we have $u_i=0$, $\nabla u_i=(\partial_\nu u_i)\nu$ and $x\cdot\nu=0$, hence $\Psi \cdot\nu=0$.
Denoting by $d\sigma$ the surface measure on $\{|x|=R\}$, we deduce from the divergence theorem that
$$\begin{aligned}
&\IR \Bigl(2nF(U)-(n-2)U\cdot\nabla F(U)\Bigr)\, dx\cr
&\ = \int_{S_R^+}\Bigl[2xF(U)+\sum_i d_i\bigl(2(x\cdot\nabla u_i\bigr)\nabla u_i-x|\nabla u_i|^2+(n-2)u_i\nabla u_i\bigr)\Bigr]\cdot\nu\, d\sigma\cr
&\ =2R\int_{\Sigma_R}F(U)\, d\sigma 
 +R \sum_i d_i\int_{\Sigma_R}\Bigl[2\Bigl({\partial u_i\over \partial\nu}\Bigr)^2
-|\nabla u_i|^2+(n-2)R^{-1}u_i{\partial u_i\over \partial\nu}\Bigr]\, d\sigma \cr
&\ =2R^n\IS F(U(R))+R^n\sum_i d_i\IS \Bigl(|\partial_r u_i|^2-|\partial_\tau u_i|^2+(n-2)R^{-1}u_i\,\partial_ru_i\Bigr)(R),
\end{aligned}$$
which proves the lemma.
\end{proof}

We will also need the following lemma.

\begin{lemma} \label{LemSeven}
Assume \eqref{hypGrowthB} and let $\delta>0$, $\beta\in (0,1)$.
Then there exists $C>0$ such that, for any bounded solution $U$ of \eqref{GenSystem} or \eqref{GenSystemB}
and any $R>0$, there holds
\be{EstimNablaSquare}
\begin{aligned}
\int_{D_R}|\nabla u_i|^2u_i^{-\beta}
&\leq C\int_{D_{(1+\delta)R}} |U|^{q+1-\beta}\\ 
&\qquad\qquad+CR^{-2}\int_{D_{(1+\delta)R}} |U|^{2-\beta},\quad R>0,\quad i=1,\cdots,m,
\end{aligned}
\ee
where the LHS is understood to be $0$ at points where $\nabla u_i=0$ and $u_i=0$.
\end{lemma}

\begin{proof}
 We only consider the case of \eqref{GenSystemB} (hence $D_R=B_R^+$),
the case of \eqref{GenSystem} being essentially the same as \cite[Lemma~2.4]{SouNHM}. Fix a cut-off function 
$0\leq \chi\in C^\infty_0(\Rn)$, such that $\chi=1$ for $|x|\leq 1$ and $\chi=0$ for $|x|\geq 1+\delta$.
 For $R>0$ we set $\varphi(x)=\varphi_R(x)=\chi(x/R)$. Let $\eps>0$.
 Multiplying the $i$-th equation in 
 \eqref{GenSystemB} with $v:=d_i^{-1}[(u_i+\eps)^{1-\beta}-\eps^{1-\beta}]\varphi^2$, integrating by parts
 and using the fact that $v=0$ on $\partial B_{(1+\delta)R}^+$, we obtain,
 setting  $\int=\int_{B_{(1+\delta)R}^+}$
 and $u=u_i$:
 $$
 \begin{aligned}d_i^{-1}
 &\int  f_i(U)\bigl[(u+\eps)^{1-\beta}-\eps^{1-\beta}]\varphi^2 \\ 
 &=\int\nabla u\cdot \nabla([(u+\eps)^{1-\beta}-\eps^{1-\beta}]\varphi^2) \\ 
 &=(1-\beta)\int|\nabla u|^2(u+\eps)^{-\beta}\varphi^2
   +2\int [(u+\eps)^{1-\beta}-\eps^{1-\beta}]\varphi (\nabla u\cdot \nabla \varphi).
 \end{aligned}$$
Estimating the last term via
$$\int (u+\eps)^{1-\beta}\varphi |\nabla u\cdot \nabla \varphi|
\leq {1-\beta\over 4}\int|\nabla u|^2(u+\eps)^{-\beta}\varphi^2+C\int|\nabla \varphi|^2(u+\eps)^{2-\beta},$$
we get
$$\int|\nabla u|^2(u+\eps)^{-\beta}\varphi^2
\leq C \int |f_i(U)|\bigl[(u+\eps)^{1-\beta}-\eps^{1-\beta}]\varphi^2+C\int|\nabla \varphi|^2(u+\eps)^{2-\beta}.$$
Letting $\eps\to 0$ (applying monotone convergence on the LHS), and then 
using \eqref{hypGrowthB}, we obtain \eqref{EstimNablaSquare}.
\end{proof}

\subsection{Proof of Theorem~\ref{thm3}}
 The case $n = 1$ can be treated by rather simple ODE arguments,
which are straightforward modifications of those in \cite[Appendix]{SouNHM}.
We thus assume $n\ge 2$ throughout this subsection.
 For sake of clarity, since the proof is rather long and technical, we split it into several steps and lemmas.

\smallskip
{\bf Step 1.} {\it Preparations.}
\smallskip

Suppose for contradiction that \eqref{GenSystem} (resp., \eqref{GenSystemB}) admits a nontrivial bounded 
solution $U$ 
and set $M:=\|U\|_\infty$.
Define
$$H(R):=\IR |U|^{p+1},\quad R>0.$$
Picking $R_0\geq 1$ such that $U\not\equiv 0$ on $B_{R_0}$, we have
\be{LowerEstimF}
H(R)\geq H(R_0)>0,\quad R\geq R_0.
\ee

In this proof, the letter $u$ will stand for any component of $U$ and
$C$ will denote generic positive constants which are {\bf independent of $R$} (but may possibly depend on the solution $U$
and on all the other parameters). 
Assuming $F(0)=0$ without loss of generality, \eqref{hypGrowthB} implies
that $F(U)\leq C|U|^{q+1}$. 
It then follows from \eqref{hypGrowthC} and the Rellich-Pohozaev identity in Lemma~\ref{LemFour} that
\be{controlFG}
H(R) \leq CG_1(R)+CG_2(R),\quad R>0,
\ee
where
\be{defFG}
G_1(R)=R^n\IS |U(R)|^{{q}+1},
\ \quad
G_2(R)=R^n\IS \bigl(|D_x U(R)|^2+R^{-2}|U(R)|^2\bigr).
\ee
In order to reach a contradiction, our goal is to find constants $a,C>0$ and $0\le b<1$ such that,
for all $R\geq R_0$, the feedback estimate
$$G_1(\tilde R),\, G_2(\tilde R)\leq CR^{-a}H^b\bigl(3R\bigr)$$
holds for some $\tilde R\in \bigl(R,3R\bigr)$.

For given function $w=w(r,\theta)$, $0<z\leq \infty$
and $R>0$, we denote
\be{DefLz}
\|w\|_z=\|w(R,\cdot)\|_{L^z(\Sigma)}, 
\ee
when no risk of confusion arises. We note that $\|\cdot\|_z$ is not a norm when $0<z<1$,
but we keep the same notation for convenience.
Finally, we put
$$I_z(R)=\|D^2_xU(R,\cdot)\|_z+R^{-1}\|D_xU(R,\cdot)\|_z+R^{-2}\|U(R,\cdot)\|_z.$$

For convenience of the readers, let us outline the rest of the proof, which is as follows:

\smallskip
 
\par\noindent Step 2: Estimation of $G_1(R)$ and $G_2(R)$ in terms of auxiliary norms for $n\ge 3$

\par\noindent Step 3: Estimation of $G_1(R)$ and $G_2(R)$ in terms of auxiliary norms for $n=2$

\par\noindent Step 4: Control of averages of auxiliary norms in terms of $R$ and $H(3R)$

\par\noindent Step 5: Measure argument

\par\noindent Step 6: Feedback estimate for $G_1$

\par\noindent Step 7: Feedback estimate for $G_2$

\par\noindent Step 8: Conclusion.

\medskip
{\bf Step 2.} {\it Estimation of $G_1(R)$ and $G_2(R)$ in terms of auxiliary norms for $n\ge 3$.} 
\smallskip
We fix a number $\eps\in(0,1)$, which will be chosen suitably small in subsequent steps of the proof.
We set 
\be{DefEll}
\alpha=2/(p-1), \qquad k=(p+1)/p,\qquad \ell=1+\eps,\quad \omega=s\alpha-\delta_0\
\ee
Let 
\be{DefGamma}
\lambda=\displaystyle{n-1\over n-3},\qquad
\rho:=
\displaystyle{n-1\over n-2}\leq2, \qquad 
\gamma:=\Bigl({p\over p+1}-{1\over n-1} \Bigr)^{-1} \in (2,\infty).
\ee
Next, if $p+1>\lambda$, then ${p\over p+1}>{2\over n-1}$ and we may define
\be{DefMu}
\mu:=\Bigl({p\over p+1}-{2\over n-1} \Bigr)^{-1} \in (p+1,\infty).
\ee
Note that the inequalities $\gamma>2$ and $\mu>p+1$ in \eqref{DefGamma}-\eqref{DefMu}
follow from $p<p_S$, which is satisfied for $\eps_0=\eps_0(n,p_0)>0$ sufficiently small.

\begin{lemma} \label{Lemma3.1}
Let $n\ge 3$. We have the estimates
\be{EstimGoneCaseB}
G_1(R)\leq\left\{\begin{aligned}
CR^n\bigl(R^2 I_\ell^{1-\nu} I_k^{\nu}\bigr)^{{q}+1}, \hfill
&\hbox{ if $p+1\geq\lambda$},\\ 
\noalign{\vskip 2mm}
CR^n \bigl(R^{2(1-\nu)} I_\ell^{1-\nu} \|U\|_s^{\nu}\bigr)^{{q}+1}, \hfill
&\hbox{ if $p+1<\lambda$},
\end{aligned}\right.
\ee
and
\be{EstimGtwo}
G_2(R)\leq CR^{n+2}I_\ell^{2(1-\tau)}I_k^{2\tau},
\ee
where $\nu, \tau\in [0,1)$ are given by
\be{DefNu}
\nu=\left\{\begin{aligned}
\Bigl(\displaystyle{1\over \lambda}-{1\over p+1}\Bigr)\Bigl({1\over \lambda}-{1\over \mu}\Bigr)^{-1},\hfill
&\hbox{ if $p+1\geq\lambda$},\\ 
\noalign{\vskip 2mm}
\Bigl(\displaystyle{1\over \lambda}-{1\over p+1}\Bigr)\Bigl({1\over \lambda}-{1\over s}\Bigr)^{-1},\hfill
&\hbox{ if $p+1<\lambda$},
\end{aligned}\right.
\ee
\be{DefTauOne}
\tau:=
\Bigl({1\over \rho}-{1\over \gamma}\Bigr)^{-1}\Bigl({1\over \rho}-{1\over 2}\Bigr)_+.
\ee
\end{lemma}

\begin{proof} It is proved in \cite[Lemma 3.1]{SouNHM} in the case $\Sigma=S^{n-1}$ that
$$\IS |U(R)|^{p+1}
\leq\left\{\begin{aligned}
C\bigl(R^2 I_\ell^{1-\nu} I_k^{\nu}\bigr)^{p+1}, \hfill
&\hbox{ if $p+1\geq\lambda$},\\ 
\noalign{\vskip 2mm}
C \bigl(R^{2(1-\nu)} I_\ell^{1-\nu} \|U\|_s^{\nu}\bigr)^{p+1}, \hfill
&\hbox{ if $p+1<\lambda$},
\end{aligned}\right.$$
and the proof is exactly the same when $\Sigma=S^{n-1}_+$.
Since $G_1(R)=R^n\IS |U(R)|^{{q}+1}\le R^n\bigl(\IS |U(R)|^{p+1}\bigr)^{q+1\over p+1}$ by H\"older's inequality,
\eqref{EstimGoneCaseB} follows.
 As for \eqref{EstimGtwo}, it is  proved in \cite[Lemma 3.2]{SouNHM} for $\Sigma=S^{n-1}$ 
 and the proof is also exactly the same when $\Sigma=S^{n-1}_+$. \end{proof}

\medskip
{\bf Step 3.} {\it Estimation of $G_1(R)$ and $G_2(R)$ in terms of auxiliary norms for $n=2$.} 

\smallskip
A similar procedure as for $n\ge 3$ could still be used,
 but this would eventually require a condition on $s$ stronger than \eqref{hypsp}, because 
the Sobolev injection $W^{1,1}(S^1)\subset L^\infty(S^1)$ ``loses'' too much information.
Instead, we shall estimate the $L^\infty(S^1)$ norm through the Gagliardo-Nirenberg type inequalities from Lemma~\ref{LemFive}(i)(ii)
and then control $G_1(R)$ by interpolating between $L^\infty$ and $L^s$.
As for the estimate of $G_2(R)$, relying on a modification of an idea in \cite{SZ}, it is based on the inequality
\be{gradBeta}
\|D_xU\|_2^2\leq \|U\|_\infty^\beta\sum_{i=1}^m\|u_i^{-\beta/2}D_x u_i\|^2_2,\qquad \beta>0,
\ee
which will make it possible to appeal to Lemma~\ref{LemSeven} (in the subsequent averaging step).
Note that in \eqref{gradBeta}, the quantity $u_i^{-\beta/2}D_x u_i$ is understood to be $0$ at points where $\nabla u_i=0$ and $u_i=0$.
To this end, from now on, we fix
$$\beta\in (0,1)\ \hbox{ with } \beta\le q+1-s$$
(assuming $\e_0\le 1/4$ without loss of generality, we can for instance choose $\beta=1/2$, 
 since $q+1-s\ge p_0+1-s-\e_0\ge p-s+1-2\e_0\ge 1-2\e_0$; we however keep track of the parameter $\beta$ for clarity).

\begin{lemma} \label{Lemma3.3}
Let $n=2$. We have
\be{EstimUinftyC}
\|U\|_\infty\leq 
C\Bigl(R^2\,I_\ell\|U\|_s^s\Bigr)^{1/(s+1)}.
\ee
Moreover, we have
\be{EstimGtwoBone}
G_1(R)\leq CR^2 \|U\|_s^s\,\|U\|_\infty^{q+1-s}
\ee
and, for any $\beta>0$, 
\be{EstimGtwoB}
G_2(R)\leq CR^2\|U\|_\infty^\beta\sum_{i=1}^m\|u_i^{-\beta/2}D_x u_i\|^2_2
+C\|U\|_s^2.
\ee
\end{lemma}

\begin{proof}
Lemma~\ref{LemFive}(ii) with $z=s$ implies
$$\|u\|_\infty\leq C\bigl(\|u\|_1+\|D^2_\theta u\|_1\bigr)^{1\over s+1}\|u\|_s^{s\over s+1}
\leq C\bigl(\|u\|_\ell+R^2\|D^2_x u\|_\ell\bigr)^{1\over s+1}\|u\|_s^{s\over s+1}.$$
This yields estimate \eqref{EstimUinftyC}. 
Next, \eqref{EstimGtwoBone} is obvious from \eqref{defFG}. To check \eqref{EstimGtwoB}, we first use Lemma~\ref{LemFive}(iii) to estimate
$$\|u\|_2 \leq C\|u\|_\infty\leq C\inf_\Sigma |u|+C\|D_\theta u\|_1   
\leq C\|u\|_s+C\|D_\theta u\|_2\leq C\|u\|_s+CR\|D_x u\|_2.$$
From \eqref{defFG} and $n=2$, we deduce that
$$G_2(R)\leq CR^2\|D_xU\|_2^2+C\|U\|_s^2,$$
hence \eqref{EstimGtwoB}, in view of \eqref{gradBeta}.
 \end{proof}

\medskip
{\bf Step 4.} {\it Control of averages of auxiliary norms in terms of $R$ and $H( 3R)$.}
\smallskip
For any $z\in (1,\infty)$, we set $J_z(R)=\int_R^{2R}I_z^z(r)r^{n-1}\,dr$.

\begin{lemma} \label{Lemma3.4}
 We have
\be{estimfk}
J_k+\int_R^{ 2R}\|U(r)\|_{p+1}^{p+1}r^{n-1}\,dr \leq CR^{{n(p-q)\over p}}H^{q/p}(3R),\qquad R\geq R_0,
\ee
\be{estimDtwo}
J_\ell\leq CR^{\tilde \mu}H^{\tilde\eta}( 3R),\hfill\qquad R\geq R_0
\ee
\be{defDeltaEta}
\tilde\mu={(n-\omega)(1-p\eps)\over p+1-s}+\bigl(n\eta+(1-\eta)\omega\bigr){p-q\over p},
\quad \eta=\frac{p(1+\eps)-s}{p+1-s},
\quad \tilde\eta={q\eta\over p}.
\ee
Moreover, when $n=2$, for any $\beta\in (0,1)$ satisfying $\beta\le q+1-s$,
we have
\be{EstimBeta}
\int_R^{ 2R}\sum_{i=1}^m\|u_i^{-\beta/2}D_x u_i(r)\|^2_2 \, r\, dr
\leq CR^dH^e(3R),\qquad R\geq R_0,
\ee
(where $C$ also depends on $\beta$), with
\be{DefDE}
d={(p-q+\beta)(2-\omega )\over p+1-s},\quad e={q+1-\beta-s\over p+1-s}.
\ee
\end{lemma}

\bigskip

\begin{proof}
We first note that, for any $z\in (1,\infty)$, Lemma~\ref{LemSix}, \eqref{GenSystem}, \eqref{GenSystemB} 
and \eqref{hypGrowthB} imply
$$\begin{aligned}
J_z&\leq C\int_{D_{2R}} \bigl(|D^2_x u|^z+R^{-z}|D_x u|^z+R^{-2z} u^z\bigr)
\leq C\Bigl(\int_{D_{3R}}|\Delta u|^z+R^{-2z}\int_{D_{3R}}|u|^z\Bigr)\\ 
&\leq C\Bigl(\int_{D_{ 3R}}|U|^{{q}z}+R^{-2z}\int_{D_{ 3R}}|U|^z\Bigr)\\ 
&\leq CR^{{n(p-q)\over p}}\Bigl(\int_{D_{3R}}|U|^{pz}\Bigr)^{q/p}
+CR^{-2z+{n(p-1)\over p}}\Bigl(\int_{D_{3R}}|U|^{pz}\Bigr)^{1\over p}\\ 
&\leq C\Bigl(R^{{n(p-q)\over q}}\int_{D_{3R}}|U|^{pz}\Bigr)^{q/p}
+CR^{-2z+{n(q-1)\over q}}\Bigl(R^{{n(p-q)\over q}}\int_{D_{ 3R}}|U|^{pz}\Bigr)^{1/p}
\end{aligned}$$
hence, 
\be{CompJz}
J_z\leq CR^{{n(p-q)\over p}}\Bigl(\int_{D_{3R}}|U|^{pz}\Bigr)^{q/p}
+CR^{n-{2zq\over q-1}}
\ee

First take $z=k$ in \eqref{CompJz}. Noting that ${n-2qk/(q-1)\le\ } n-2pk/(p-1)=n-2(p+1)/(p-1)<0$ due to $p<p_S$,
and using \eqref{LowerEstimF}, we deduce \eqref{estimfk}.

Applying \eqref{CompJz} with $z=\ell=1+\eps$, and using
H\"older's inequality and assumption \eqref{BasicEstAnews},
we obtain
$$\begin{aligned}
J_\ell
&\leq CR^{{n(p-q)\over p}}\Bigl(\int_{D_{3R}}|U|^{p(1+\eps)}\Bigr)^{q/p}+CR^{n-{2q(1+\eps)\over q-1}} \\ 
&\leq CR^{{n(p-q)\over p}}\Bigl(\int_{D_{3R}}|U|^s\Bigr)^{(1-\eta) q/p}\Bigl(\int_{D_{3R}}|U|^{p+1}\Bigr)^{\eta q/p}
+CR^{n-{2q(1+\eps)\over q-1}} \\ 
&\leq CR^{{n(p-q)\over p}}R^{(n-\omega)(1-\eta) q/p} H^{\eta q/p}( 3R)+CR^{n-{2q(1+\eps)\over q-1}}\\ 
&\leq CR^n\bigl\{R^{-(\omega+(n-\omega)\eta)}H^{\eta}( 3R)+CR^{-{2p(1+\eps)\over q-1}}\bigr\}^{q/p}.
\end{aligned}$$
Assume $\eps\in(0,1/p)$. We claim that
\be{CompJzA}
A:={2p(1+\eps)\over q-1}-\bigl(n\eta+(1-\eta)\omega\bigr)>0.
\ee
To this end, using $\omega\le s\alpha$, we compute
$$\begin{aligned}(p+1-s)A
&=(p+1-s)\Bigl({2p(1+\eps)\over q-1}-\omega\Bigr)-(n-\omega)(p-s+p\e)\\ 
&=\Bigl({2(p+1-s)\over q-1}-n\Bigr)p+ns+\Bigl\{{2(p+1-s)\over q-1}-n\Bigr\}p\e
-\omega(1-p\e) \\ 
&\ge\Bigl({2(p+1-s)\over q-1}-n\Bigr)p+ns+\Bigl\{{2(p+1-s)\over q-1}-n\Bigr\}p\e
-s\alpha (1-p\e) \\ 
&=\Bigl({2(p+1-s)\over q-1}-n\Bigr)p+(n-\alpha)s+\Bigl\{{2(p+1-s)\over q-1}-n+s\alpha\Bigr\}p\e \\ 
& \equiv A'+A''p\e.
\end{aligned}$$
Next, since $2(p+1)-n(q-1)>0$ owing to $1<q\le p<p_S$, and using $s\le p$, we have
$$\begin{aligned}
A'
&= {(p+1-s)2p-(\alpha s+n(p-s))(q-1)\over q-1} \\ 
&= {2p(p+1)-[2p+(\alpha-n)(q-1)]s-np(q-1)\over q-1} \\ 
&= {s\over q-1}\Bigl\{\bigl[2(p+1)-n(q-1)]{p\over s}-[(\alpha-n)(q-1)+2p\bigr]\Bigr\} \\ 
&\ge {s\over q-1}\Bigl\{[2(p+1)-n(q-1)]-[(\alpha-n)(q-1)+2p]\Bigr\} \\ 
&= {s\over q-1}\bigl(2-\alpha(q-1)\bigr) 
={s\over q-1}{2(p-q)\over p-1}.
\end{aligned}$$
Since also
$$\begin{aligned}A''
&={2(p+1-s)+(q-1)(\alpha s-n)\over q-1}
={2(p+1)-n(q-1)-{2s\over p-1}(p-q)\over q-1},
\end{aligned}$$
it follows that
$$(p+1-s)A\ge{2s(p-q)\over (p-1)(q-1)}(1-p\e)+{2(p+1)-n(q-1)\over q-1}p\e>0$$
i.e., \eqref{CompJzA}. 
In view of \eqref{LowerEstimF}, we deduce~\eqref{estimDtwo}.

Let us turn to \eqref{EstimBeta}.
Using Lemma~\ref{LemSeven}, H\"older's inequality, $n=2$,
$U\not\equiv 0$ on $D_{R_0}$ and $q\ge 1$, 
we obtain, for $R\ge R_0$,
$$\begin{aligned}
\sum_{i=1}^m\int_{D_{2R}}|\nabla u_i|^2u_i^{-\beta}
&\leq C\int_{D_{3R}} |U|^{q+1-\beta}
+CR^{-2+n{q-1\over q+1-\beta}} \Bigl(\int_{D_{ 3R}} |U|^{q+1-\beta}\Bigr)^{2-\beta\over q+1-\beta}\\ 
&\leq C\int_{D_{3R}} |U|^{q+1-\beta}
\leq C\Bigl(\int_{D_{3R}} |U|^{p+1}\Bigr)^{q+1-\beta-s\over p+1-s}
\Bigl(\int_{D_{3R}} |U|^s\Bigr)^{p-q+\beta\over p+1-s}.
\end{aligned}$$
Estimate \eqref{EstimBeta} then follows from assumption \eqref{BasicEstAnews}.
\end{proof}

\medskip\goodbreak
{\bf Step 5.} {\it Measure argument.}
\smallskip

For a given constant $K>0$, let us define the sets
\be{DefGammaThree}
\Gamma_1(R):=\bigl\{r\in (R, 2R);\, \|U(r)\|_s > KR^{-\omega/s}\bigr\}, 
\ee
\be{DefGammaOne}
\Gamma_2(R):=\bigl\{r\in (R, 2R);\, I_k^k(r)+\|U(r)\|_{p+1}^{p+1}>K(R^{-n}H(3R))^{q/p}\bigr\},
\ee
\be{DefGammaTwo}
\Gamma_3(R):=
\ \ \bigl\{r\in (R, 2R);\, I_\ell^\ell(r)>KR^{\tilde \mu-n} H^{\tilde\eta}(3R)\bigr\},
\ee
\be{DefGammaFour}
\Gamma_4(R):=
\left\{\begin{aligned}
\Bigl\{r\in (R, 2R);\, \displaystyle\sum_{i=1}^m\|u_i^{-\beta/2}D_x u_i(r)\|_2^2>KR^{d-2} H^e( 3R)\Bigr\},\hfill
&\quad\hbox{if $n=2$,}\hfill \\ 
\noalign{\vskip 2mm}
\  \emptyset,\hfill
&\quad\hbox{if $n\geq 3$.}\hfill
\end{aligned}\right.
\ee

\begin{lemma} \label{Lemma3.5}
For each $R\geq R_0$, we can find
\be{CapGamma}
 \tilde R\in (R, 2R)\setminus\bigcup_{i=1}^4 \Gamma_i(R)\neq\emptyset.
\ee
\end{lemma}

\begin{proof}
For $R\geq R_0$, by assumption \eqref{BasicEstAnews}, 
we have
$$CR^{n-\omega}\geq \int_R^{ 2R}\|U(r)\|_s^s\,r^{n-1}\,dr
\geq |\Gamma_1(R)|R^{n-1}K^sR^{-\omega}
=|\Gamma_1(R)| K^sR^{n-\omega-1}$$
and, by estimate \eqref{estimfk} in Lemma~3.4, 
$$
\begin{aligned}
CR^{{n(p-q)\over p}}H^{q/p}(3R)
&\geq\int_R^{ 2R} \bigl(I_k^k(r)+\|U(r)\|_{p+1}^{p+1}\bigr)r^{n-1}\,dr \\ 
&\geq |\Gamma_2(R)|R^{n-1}K(R^{-n}H(3R))^{q/p}\\ 
&=|\Gamma_2(R)|KR^{-1}R^{{n(p-q)\over p}}H^{q/p}( 3R). \\ 
\end{aligned}
$$
Consequently, $|\Gamma_1(R)|\leq  R/5$ for $K\geq (5\, C)^{1/s}$ and $|\Gamma_2(R)|\leq  R/5$ for $K\geq 5\, C$.
In a similar way, it follows from estimates \eqref{estimDtwo} and \eqref{EstimBeta} in Lemma~3.4 that
$|\Gamma_3(R)|, |\Gamma_4(R)| \leq  R/5$, for $K>0$ large enough (independent of $R\geq R_0$).
The lemma follows.
\end{proof}

\medskip\goodbreak
{\bf Step 6.} {\it Feedback estimate for $G_1$.}
\smallskip

Building on the results of the previous steps, we shall prove the following feedback estimate.

\begin{lemma} \label{Lemma3.6}
If $\e_0=\e_0(n,p_0)>0$ is sufficiently small, 
then there exist numbers $a>0, b\in [0,1)$ and $\eps\in (0,1)$ in \eqref{DefEll}, such that
\be{EstimGtwoFeedB}
G_1( \tilde R)\leq CR^{-a}H^b( 3R),\quad R\geq R_0,
\ee
where $\tilde R$ is given by Lemma~\ref{Lemma3.5}. 
Moreover, $a, b, \eps$ depend only on $n,p$ and $s$.
\end{lemma}

\begin{proof}
 The proof involves only elementary but long calculations.
Recall that $\tilde\eta=\tilde\eta_\eps$ and $\tilde\mu=\tilde\mu_\eps$ are defined in \eqref{defDeltaEta}.
\smallskip

{\it Case 1: $n\ge 3$ and $p+1\geq \lambda$} (hence $n\geq 4$ by \eqref{DefGamma}).
We deduce from Lemmas~\ref{Lemma3.1} and \ref{Lemma3.5},  
\eqref{DefGammaOne}, \eqref{DefGammaTwo} and $\ell=1+\eps$ that
$$G_1( \tilde R)
\le CR^n\bigl[R^2 (R^{\tilde \mu-n} H^{\tilde\eta}(3R))^{(1-\nu)/\ell} (R^{-n}H(3R))^{\nu q/pk}\bigr]^{{q}+1},$$
where,
in view of \eqref{DefGamma}, \eqref{DefMu} and \eqref{DefNu}, 
$\nu=\frac{(n-3)(p+1)}{n-1}-1$.
Therefore, we have \eqref{EstimGtwoFeedB} with
$$\tilde a=\tilde a_\eps:=(q+1)\Bigl[\frac{(1-\nu)(n-\tilde\mu_\eps)}{1+\eps}+\frac{n\nu q}{p+1}-2-\frac{n}{q+1}\Bigr],$$
$$\tilde b=\tilde b_\eps:=(q+1)\Bigl[\frac{\nu q}{p+1}+\frac{(1-\nu)\tilde\eta}{1+\eps}\Bigr].$$
Now, as $\eps,\delta_0\to 0$ (hence $\omega\to s\alpha$) and $p, q\to p_0$
(hence $\alpha\to\alpha_0:=2/(p_0-1)$ and $\nu\to\nu_0:=\frac{(n-3)(p_0+1)}{n-1}-1$), 
we have
$$\tilde a_\eps\to a_0:=(p_0+1)\Bigl[(1-\nu_0)(n-\mu_0)+\frac{n\nu_0 p_0}{p_0+1}-2-\frac{n}{p_0+1}\Bigr],
\quad \mu_0=\frac{n-\alpha_0 s}{p_0+1-s},$$
$$\tilde b_\eps\to b_0:=(p_0+1)\Bigl[\frac{\nu_0 p_0}{p_0+1}+(1-\nu_0)\frac{p_0-s}{p_0+1-s}\Bigr],
\quad \eta_0=\frac{p_0-s}{p_0+1-s}.
$$
It thus suffices to check that $a_0>0$ and $b_0<1$. We have
$$\begin{aligned}
1-b_0
&=1-\nu_0 p_0-\frac{(1-\nu_0)(p_0+1)(p_0-s)}{p_0+1-s} \\ 
&=\frac{(p_0+1-s)(1-\nu_0 p_0)+(1-\nu_0)(p_0+1)(s- p_0)}{p_0+1-s}\\ 
&=\frac{(p_0-\nu_0)s+1-p_0^2}{p_0+1-s}=\frac{2(p_0+1)s/(n-1)+1-p_0^2}{p_0+1-s}\\
&=\frac{p_0+1}{p_0+1-s}\Bigl[\frac{2s}{n-1}+1-p_0\Bigr]>0.\\ 
\end{aligned}$$
On the other hand, after some computation, we observe that
\be{RelNdelta}
n-\mu_0=(n-\alpha_0(p_0+1))\eta_0+\alpha_0 p_0.
\ee
Consequently,
$$\begin{aligned}
\frac{a_0}{p_0+1}
&=(1-\nu_0)\bigl((n-\alpha_0(p_0+1))\eta_0+\alpha_0 p_0\bigr)+\frac{n\nu_0 p_0}{p_0+1}-2-\frac{n}{p_0+1}\\ 
&=\Bigl[\frac{2}{p_0-1}-\frac{n}{p_0+1}\Bigr](1-b_0)>0.
\end{aligned}$$

\bigskip

{\it Case 2: $n\ge 3$ and $p+1<\lambda$.}
We deduce from Lemmas~\ref{Lemma3.1} and \ref{Lemma3.5},  
\eqref{DefGammaThree} and \eqref{DefGammaTwo} that
$$G_1( \tilde R)
\leq CR^n\Bigl[R^{-\omega\nu/s}\bigl(R^{\tilde\mu+2(1+\eps)-n} H^{\tilde\eta}(3R)\bigr)^\frac{1-\nu}{1+\eps}\Bigr]^{q+1},$$
where $\nu$ is given by \eqref{DefNu}.
We deduce \eqref{EstimGtwoFeedB} with
$$\tilde a=\tilde a_\eps:=(q+1)\Bigl[\frac{\omega\nu}{s}
+{(1-\nu)(n-2(1+\eps)-\tilde\mu)\over 1+\eps}-\frac{n}{q+1}\Bigr],
\ \quad \tilde b=\tilde b_\eps:=\frac{(1-\nu)(q+1)\tilde\eta}{1+\eps}.$$
As $\eps,\delta_0\to 0$ (hence $\omega\to s\alpha$) and $p, q\to p_0$ (hence $\alpha\to\alpha_0$ and $\nu\to\nu_0$), 
we have
$$\tilde a_\eps\to a_0:=(p_0+1)\Bigl[\nu_0\alpha_0+(1-\nu_0)(n-2-\mu_0)-\frac{n}{p_0+1}\Bigr],
\quad \mu_0={n-\alpha_0 s\over p_0+1-s},$$
$$\tilde b_\eps\to b_0:=(1-\nu_0)(p_0+1)\eta_0,
\quad \eta_0=\frac{p_0-s}{p_0+1-s},$$
where (recalling $\lambda=(n-1)/(n-3)$), 
we have 
$$1-\nu_0=\begin{cases}
\displaystyle\frac{(p_0+1-s)\lambda}{(p_0+1)(\lambda-s)}, &\quad\hbox{ if $n\ge 4$,}\\ 
\noalign{\vskip 2mm}
\displaystyle\frac{p_0+1-s}{p_0+1},&\quad\hbox{ if $n=3$.}\\ 
\end{cases}
$$
It thus suffices to check that $a_0>0$ and $b_0<1$. 
Therefore, owing to \eqref{DefGamma} and \eqref{hypsp}, we have
$$
1-b_0
=1-\frac{(p_0-s)\lambda}{\lambda-s}=\frac{\lambda-1}{\lambda-s}\Bigl(s-\frac{\lambda(p_0-1)}{\lambda-1}\Bigr)
=\frac{\lambda-1}{\lambda-s}\Bigl(s-\frac{(n-1)(p_0-1)}{2}\Bigr)>0
$$
if $n\ge 4$ and $1-b_0=s-p_0+1>0$ if $n=3$.
On the other hand, after some computation, we observe that 
$$\displaystyle\frac{p_0+1-s}{p_0+1}(n-2-\mu_0-\alpha_0)=\Bigl(\frac{n}{p_0+1}-\alpha_0\Bigr)(p_0-s).$$
Therefore we get (replacing $\lambda/(\lambda-s)$ by $1$ when $n=3$)
$$\begin{aligned}
\frac{a_0}{p+1}
&=(1-\nu_0)(n-2-\mu_0-\alpha_0)+\alpha_0-\frac{n}{p_0+1}\\
&=\Bigl(\frac{n}{p_0+1}-\alpha_0\Bigr){(p_0-s)\lambda\over\lambda-s}+\alpha_0-\frac{n}{p_0+1}\\ 
&=\Bigl(\frac{2}{p_0-1}-\frac{n}{p_0+1}\Bigr)\Bigl(1-\frac{(p_0-s)\lambda}{\lambda-s}\Bigr)
=\Bigl[\frac{2}{p_0-1}-\frac{n}{p_0+1}\Bigr](1-b_0)>0.
\end{aligned}$$

\smallskip
{\it Case 3: $n=2$.} 
Since $s\le p$, it follows from \eqref{EstimUinftyC} in Lemma~\ref{Lemma3.3}, 
Lemma~\ref{Lemma3.5},  
\eqref{DefGammaThree} and \eqref{DefGammaTwo}  that
\be{UinftyTwodimA}
\|U(\tilde R)\|_\infty\leq C\Bigl(R^{2-\omega}\,\bigl(R^{\tilde\mu-2} H^{\tilde\eta}(3R)\bigr)^{1/(1+\eps)}\Bigr)^{1/(s+1)},
\ee
hence
$$G_1(\tilde R)\leq CR^{2-\omega}\Bigl(R^{2-\omega}\,\bigl(R^{\tilde\mu_\e-2} H^{\tilde\eta_\e}( 3R)\bigr)^{1/(1+\eps)}\Bigr)^{(q+1-s)/(s+1)}$$
by \eqref{EstimGtwoBone}. We deduce \eqref{EstimGtwoFeedB} with
$$a=a_\eps=\omega-2+\frac{q+1-s}{s+1}\Bigl(\omega-2+\frac{2-\tilde\mu_\e}{1+\eps}\Bigr), \qquad 
b=b_\eps=\frac{q+1-s}{(s+1)(1+\eps)}\,\tilde\eta_\e.$$
As $\eps,\delta\to 0$ (hence $\omega\to s\alpha$) and $p, q\to p_0$ (hence $\alpha\to\alpha_0$), we have
$$a_\eps\to a_0:=\alpha_0 s-2+\frac{p_0+1-s}{s+1}(\alpha_0 s-\mu_0), \qquad 
\mu_0=\frac{2-\alpha_0 s}{p_0+1-s},$$
$$b_\eps\to b_0:=\frac{p_0+1-s}{s+1}\,\eta_0=\frac{p_0-s}{s+1}.$$
It thus suffices to check that $a_0>0$ and $b_0<1$, which is true owing to \eqref{hypsp} and
$$\begin{aligned}
a_0&=\frac{(\alpha_0 s-2)(s+1)+\alpha_0 s(p_0+1-s)+\alpha_0 s-2}{s+1}\\
&=\frac{\bigl(\alpha_0(p_0+3)-2\bigr)s-4}{s+1}=\frac{4(\alpha_0 s-1)}{s+1}.\end{aligned}$$
\end{proof}

\medskip\goodbreak
{\bf Step 7.} {\it Feedback estimate for $G_2$.}
\smallskip

Similarly as for $G_1$ in Step 6, we shall prove the following feedback estimate for $G_2$.

\begin{lemma} \label{Lemma3.7}
If $\e_0=\e_0(n,p_0)>0$ is sufficiently small, 
then there exist numbers $\tilde a'>0, \tilde b'\in [0,1)$ and $\eps\in (0,1)$ in \eqref{DefEll}, such that
\be{EstimGtwoFeedC}
G_2( \tilde R)\leq CR^{-\tilde a'}H^{\tilde b'}( 3R),\quad R\geq R_0,
\ee
where $\tilde R$ is given by Lemma 3.5.
Moreover, $\tilde a', \tilde b', \eps$ depend only on $n,p$ and $s$.
\end{lemma}

\begin{proof}
 {\it Case 1:} $n\geq 3$.
We deduce from Lemmas~\ref{Lemma3.1} and \ref{Lemma3.5},   
\eqref{DefGammaOne}, \eqref{DefGammaTwo} and $\ell=1+\eps$ that
$$G_2( \tilde R)\leq CR^{n+2}(R^{\tilde\mu-n} H^{\tilde\eta}(3R))^\frac{2(1-\tau)}{1+\eps} (R^{-n}H(3R))^{2\tau q/(p+1)},$$ 
where $\tau$ is defined in \eqref{DefTauOne} and $\tilde\eta=\tilde\eta_\eps, \tilde\mu=\tilde\mu_\eps$ are defined in \eqref{defDeltaEta}.
This yields \eqref{EstimGtwoFeedC} with
$$\tilde a'=\tilde a'_\eps=\frac{2(1-\tau)(n-\tilde\mu)}{1+\eps} +\frac{2n\tau q}{p+1}-(n+2),
\qquad \tilde b'=\tilde b'_\eps=\frac{2\tau q}{p+1}+\frac{2(1-\tau)\tilde\eta}{1+\eps},$$
where $\tau=(p+1)(n-3)/2(n-1)$ by \eqref{DefGamma}, \eqref{DefTauOne}.
As $\eps,\delta\to 0$ (hence $\omega\to s\alpha$) and $p, q\to p_0$ (hence $\alpha\to\alpha_0$), 
we have
$$\tilde a_\eps\to a'_0=2(1-\tau_0)(n-\mu_0)+\frac{2n\tau_0 p_0}{p_0+1}-(n+2),
\quad \mu_0=\frac{n-\alpha_0 s}{p_0+1-s},$$
$$\tilde b'_\eps\to b'_0=\frac{2\tau_0 p_0}{p_0+1}+2(1-\tau_0)\eta_0,
\quad\eta_0=\frac{p_0-s}{p_0+1-s},$$
where $\tau_0=(p_0+1)(n-3)/2(n-1)$.
It thus suffices to show that $a'_0>0$ and $b'_0<1$. 
We compute
$$\begin{aligned}1-b'_0
&=1-\frac{2\tau_0 p_0}{p_0+1}-\frac{2(1-\tau_0)(p_0-s)}{p_0+1-s}\\ 
&=\frac{(p_0+1-s)(1-(2\tau_0\frac{p_0}{p_0+1}))+2(1-\tau_0)(s-p_0)}{p_0+1-s}\\ 
&=\frac{(1-\frac{2\tau_0}{p_0+1})s+1-p_0}{p_0+1-s}=\frac{\frac{2s}{n-1}+1-p_0}{p_0+1-s}>0,
\end{aligned}$$
by assumption \eqref{hypsp}. On the other hand, using \eqref{RelNdelta}, we find that
$$\begin{aligned}
a'_0
&=2(1-\tau_0)\bigl[(n-\alpha_0(p_0+1))\eta_0+\alpha_0 p_0\bigr]+2n\tau_0\frac{p_0}{p_0+1}-(n+2)\\ 
&=\Bigl[\frac{2(p_0+1)}{p_0-1}-n\Bigr]\Bigl(1-2(1-\tau_0)\eta_0-2\tau_0\frac{p_0}{p_0+1}\Bigr)
=\Bigl[\frac{2(p_0+1)}{p_0-1}-n\Bigr](1-b'_0)>0.
\end{aligned}$$

 \smallskip
{\it Case 2:} $n=2$. 
By \eqref{EstimGtwoB} in Lemma~\ref{Lemma3.3}, Lemma~\ref{Lemma3.5}, \eqref{DefGammaThree} and \eqref{DefGammaFour}, we have
\be{defGtwoone}
\begin{aligned}
G_2( \tilde R)
&\leq CR^2\|U( \tilde R)\|_\infty^\beta\sum_{i=1}^m\|u_i^{-\beta/2}D_x u_i( \tilde R)\|^2_2
+C\|U( \tilde R)\|_s^2 \\ 
&\leq CR^d H^e( 3R)\,\|U( \tilde R)\|_\infty^\beta+CR^{-2\omega/s}=:G_{2,1}+CR^{-2\omega/s},
\end{aligned}
\ee
where $d,e$ are defined in \eqref{DefDE}.
 Then \eqref{UinftyTwodimA} implies
$$G_{2,1}\leq CR^d H^e(3R)\Bigl(R^{2-\omega}\,\bigl(R^{\tilde\mu-2} H^{\tilde\eta}(3R)\bigr)^{1/(1+\eps)}\Bigr)^{\beta/(s+1)}.$$
By \eqref{defGtwoone} and \eqref{LowerEstimF}, we deduce \eqref{EstimGtwoFeedC} with
$$\tilde a'=\tilde a'_\eps=\min(\bar a'_\eps,2\alpha),
\qquad \bar a'_\eps=-d+\frac{\beta}{s+1}\Bigl(\omega-2+\frac{2-\tilde\mu}{1+\eps}\Bigr)$$
and
$$\tilde b'=\tilde b'_\eps=e+\frac{\tilde\eta\beta}{(s+1)(1+\eps)}.$$
As $\eps,\delta\to 0$ (hence $\omega\to s\alpha$) and $p, q\to p_0$ (hence $\alpha\to\alpha_0$),
we have
$$\bar a'_\eps\to \bar a'_0:=-\frac{\beta(2-s\alpha_0)}{p_0+1-s}+\frac{\beta}{s+1}(s\alpha_0-\mu_0),
\quad \mu_0=\frac{2-\alpha_0 s}{p_0+1-s},$$
$$\tilde b'_\eps\to b'_0:=\frac{p_0+1-\beta-s}{p_0+1-s}+\frac{\eta_0\beta}{s+1},
\quad\eta_0=\frac{p_0-s}{p_0+1-s}.$$
It thus suffices to check that $\bar a'_0>0$ and $b'_0<1$. 
Using \eqref{hypsp}, we obtain
$$\beta^{-1}(1-b'_0)=\frac{1}{p_0+1-s}-\frac{p_0-s}{(p_0+1-s)(s+1)}
=\frac{2s-p_0+1}{(p_0+1-s)(s+1)}>0$$
and
$$\begin{aligned}\beta^{-1}\bar a'_0
&=\frac{\alpha_0 s-2}{p_0+1-s}+\frac{\alpha_0 s-\mu_0}{s+1}
={(\alpha_0 s-2)(s+1)+\alpha_0 s(p_0+1-s)+\alpha_0 s-2\over(p_0+1-s)(s+1)}\\ 
&=\frac{2(\alpha_0 s-2)-2s+\alpha_0 s(p_0+1)}{(s+1)(p_0+1-s)}
=\frac{4(\frac{2s}{p_0-1}-1)}{(s+1)(p_0+1-s)}>0.
\end{aligned}$$
This completes the proof of the Lemma.
\end{proof}

\medskip\goodbreak
{\bf Step 8.} {\it Conclusion.} Inequalities \eqref{LowerEstimF}, \eqref{controlFG}, 
Lemmas~\ref{Lemma3.6}--\ref{Lemma3.7}  
and the nondecreasing property of $H$ imply that
\be{Condhatb}
H(R)\leq C_0R^{-\hat a}H^{\hat b}\bigl(3R\bigr), \quad R\geq R_0,
\ee
for some constant $C_0>0$, with $\hat a=\min(a,\tilde a')>0$ and $\hat b=\max(b,\tilde b')\in [0,1)$.
It can be shown (see \cite{SouNHM} for details) that \eqref{Condhatb} entails an exponential lower bound on $H(R)$.
But the boundedness assumption on $U$ implies 
$H(R)=O(R^n)$ as $R>1$. This is a contradiction and Theorem~\ref{thm3} is proved.

\subsection{A priori estimates and proof of Theorem~\ref{thm1}} \label{subAEs1}

It is well known (see, e.g.,~\cite[Lemma 3.4]{QS-CMP}) that, under assumption \eqref{hypGrowthD}, any 
 bounded solution of \eqref{GenSystem}
satisfies
\eqref{BasicEstAnews} with $s=p$, $\delta_0=0$ (and $C$ depending on the bound of the solution).
Theorem~\ref{thm1} in the whole space case is then a direct consequence of Theorem~\ref{thm3}.

\smallskip
To prove Theorem~\ref{thm1} in the half-space case, we shall rely on the following a priori estimate,
whose proof is postponed to the next section.

\begin{proposition} \label{PropresD}
Let $1\le q<p$ and $\alpha=\frac{2}{p-1}$. We assume that $f:\R^N\to \R^N$ is continuous and satisfies
\be{hypGrowthD2}
 \xi\cdot f(U)\ge c|U|^p
\ee
and 
\be{Hypfa}
|f(U)|\le c_1(|U|^q+|U|^p),
\ee
 with some $\xi\in(0,\infty)^N$, $c, c_1>0$.
Let $U$ 
be a solution of \eqref{GenSystemB}
such that
\be{Hypua}
|U(x)|=O(\exp(|x|^{p+1-\eta})),\quad |x|\to\infty,
\ee
for some $\eta>0$.
There exists $\e_0\in(0,1)$ depending only on $p,\eta$ such that, 
for any $\e\in(0,\e_0)$, there holds
\be{Hypuab}
\int_{B_R^+} |U|^{\frac{p+1}{2}-\e}dx
\le CR^{n-k},\quad R\ge 1,\quad\hbox{ where } k:=\Bigl(\frac{p+1}{2}-\e\Bigr)\alpha-\frac{p-q}{p-1}(1+\eps\alpha).
\ee
\end{proposition}

Theorem~\ref{thm1} in the half-space case is a direct consequence of Theorem~\ref{thm3} and Proposition~\ref{PropresD}.
Indeed, by Proposition~\ref{PropresD}, estimate \eqref{BasicEstAnews} is true for any $s\in(1,(p+1)/2)$ 
with\footnote{Notice that if $p_0-\eps_0'\leq q\leq p\leq p_0+\eps_0'$ with $\eps_0'\leq\eps_0$
small enough, then $\delta_0\leq\eps_0$.} $\delta_0=\frac{p-q}{p-1}(1+((p+1)/2-s)\alpha)$.
Since $p<n/(n-2)$, 
condition \eqref{hypsp} is satisfied by taking $s$ close enough to $(p+1)/2$. 

\subsection{Proof of Proposition~\ref{PropresD}} \label{subAEs2}

In view of the proof we prepare the following two propositions.

\begin{proposition} \label{PropresA}
 Let $\alpha\in(0,1)$, $R>0$, $g\in C(\overline B_R^+;\R)$ and 
$u\in C^1(\overline B_R^+)\cap C^2(B_R^+)$ be a solution of
\be{InAA}
-\Delta u=g,\quad x\in B_R^+,\qquad \hbox{with $u=0$ on $B^0_R$.}
\ee
There exists a constant $C=C(n,\alpha)>0$ such that
$$R^{\alpha-1} \Bigl|\int_{B_R^+} u(x)x_n^{-\alpha}dx\Bigr|\le C\int_{B_R^+} |g(x)| x_n dx + C\int_{S_R^+} |u(x)| d\sigma_R.$$
\end{proposition}

\begin{proposition} \label{PropresC}
Let $p>1$, $c, R>0$, set $\alpha=2/(p-1)$ and let
$0\le v\in C^1(\overline B_{2R}^+)\cap C^2(B_{2R}^+)$ satisfy
\be{InD}
-\Delta v\ge cv^p,\quad x\in B_{2R}^+.
\ee
There exists a constant $C=C(n,c)>0$ such that 
$$\int_{B_R^+} v^px_n dx \le CR^{n+1-p\alpha}.$$
\end{proposition}

The proof of Proposition~\ref{PropresA} relies on the following lemma.

\begin{lemma} \label{lemresB}
Let $\e>0$, $\alpha\in(0,1)$ and let $w_\e\in C^1(\overline B_1^+)\cap C^2(\overline B_1^+\setminus\Sigma_1)$ 
be the unique solution of the auxiliary problem 
$$-\Delta w_\e= (x_n+\e)^{-\alpha},\quad x\in B_1^+,\qquad\hbox{with $w_\e=0$ on $\partial B_1^+$.}$$
We have
\be{InA}
w_\e(x)\le Cx_n,\quad x\in B_1^+,
\ee
\be{InB}
|\partial_\nu w_\e(x)|\le C,\quad x\in S_1^+,
\ee
with $C>0$ independent of $\e$.
\end{lemma}

\begin{proof}
$\bullet$ Proof of \eqref{InA}. This follows from the maximum principle, noting that $\overline w(x):=c(x_n-x_n^{2-\alpha})$
with $c:=((2-\alpha)(1-\alpha))^{-1}$
satisfies $\overline w\ge 0$ on $\partial B_1^+$ and
$$-\Delta\overline w =c(2-\alpha)(1-\alpha) x_n^{-\alpha}=x_n^{-\alpha},\quad x_n>0.$$

$\bullet$ Proof of \eqref{InB}. We first note that $w_\e\le W_\e$ in $B_1^+$, where $W_\e\in W^{2,q}(B_1)$
($1<q<\infty$) is the (strong) solution of 
$$-\Delta W_\e= (x_n+\e)^{-\alpha}\xi_{B_1^+},\quad x\in B_1,\qquad\hbox{with $W_\e=0$ on $\partial B_1$,}$$
hence
\be{InC}
|\partial_\nu w_\e(x)|\le |\partial_\nu W_\e(x)|,\quad x\in S_1^+.
\ee
We claim that
\be{EstWeps}
|\partial_\nu W_\e(x)|\le C,\quad x\in S_1^+,
\ee
which will prove the lemma.

To show \eqref{EstWeps} we use the representation formula
$$\nabla W_\e=\int_{B_1^+} \nabla_x G(x,y) (y_n+\e)^{-\alpha}dy,$$
where $G$ is the Dirichlet Green kernel of $B_1$. 
Let $x\in S_1^+$. It follows that 
$$\partial_\nu W_\e(x)=\int_{B_1^+} \nu(x)\cdot \nabla_xG(x,y) (y_n+\e)^{-\alpha}dy.$$
Using the estimate
$G(z,y)\le C|z-y|^{-n}d(z)d(y)$ (where $d(x)=\hbox{\rm dist}(x,\partial B_1)$, see \cite[Theorem 5]{Dav}), we have
$$|\nu(x)\cdot \nabla_xG(x,y)|=\lim_{z\to x}\frac{G(z,y)}{d(z)}\le C|x-y|^{-n}d(y)$$
hence, writing $x=(x',x_n)$ and using $d(y)\le |x-y|$,
$$\begin{aligned}
|\partial_\nu W_\e(x)|
&\le C\int_{B_1^+} |x-y|^{-n}d(y) (y_n+\e)^{-\alpha}dy
\le C\int_{B_1^+} |x-y|^{1-n} y_n^{-\alpha}dy\\ 
&\le C\int_0^1  s^{-\alpha}\Bigl(\int_{|y'|\le 1} (|x'-y'|^2+|x_n-s|^2)^{(1-n)/2}dy'\Bigr)ds \\ 
&\equiv C\int_0^1  s^{-\alpha}J(x',|x_n-s|)ds.
\end{aligned}$$
For all $h\in(0,1]$, by the change of variable $\xi=h\zeta$ ($\xi,\zeta\in R^{n-1}$), we have
$$
J(x',h)\le \int_{|\xi|\le 2} (|\xi|^2+h^2)^{(1-n)/2}d\xi 
= \int_{|\zeta|\le 2/h} (\zeta^2+1)^{(1-n)/2} d\zeta,
$$
 hence
$$
J(x',h)\le
\begin{cases}
C,&\hbox{if $n=1$,}\\
\noalign{\vskip 1mm}
C\int_0^{2/h} (r^2+1)^{(1-n)/2} r^{n-2}dr\le C(1+|\log h|), &\hbox{if $n\ge 2$}
\end{cases}
$$
where, in the second case, we used 
$$(r^2+1)^{(1-n)/2} r^{n-2}\le (r^2+1)^{\frac{1-n}{2}+\frac{n-2}{2}}  =(r^2+1)^{-1/2}\le  (r^2+1)^{-1} (r+1).$$
 We thus get (for any $n\ge 1$)
$$
|\partial_\nu W_\e(x)|\le C+\int_0^1  s^{-\alpha}|\log |x_n-s||ds.
$$
Now,  for $0<x_n<1$, we have
$$\begin{aligned}
\int_0^1  s^{-\alpha}|\log |x_n-s||ds
&=\int_0^{x_n/2}  s^{-\alpha} |\log |x_n-s|| ds+\int_{x_n/2}^1  s^{-\alpha} |\log |x_n-s|| ds \\ 
&\le C |\log (x_n/2)|x_n^{1-\alpha} +\int_{x_n/2}^1  |x_n-s|^{-\alpha} |\log |x_n-s|| ds \\ 
&\le C |\log (x_n/2)|x_n^{1-\alpha} +2\int_0^1  t^{-\alpha} |\log t| dt \le C,
\end{aligned}
$$
 hence \eqref{EstWeps}, and we conclude that \eqref{InB} holds. 
\end{proof}

\begin{proof}[Proof of Proposition~\ref{PropresA}]
It suffices to prove the estimate for $R=1$. Indeed, 
applying the estimate to $\tilde u(x)=u(Rx)$, $\tilde g(x)=R^2g(Rx)$ for $x\in B_1^+$, we then have
$$\int_{B_1^+} u(Rx) x_n^{-\alpha}dx\le CR^2\int_{B_1^+} g(Rx) x_n dx + C\int_{S_1^+} u(Rx) d\sigma_1$$
hence, with the change of variable $y=Rx$,
$$R^{\alpha-n}\int_{B_R^+} u(y) y_n^{-\alpha}dy\le CR^{1-n}\int_{B_R^+} g(y) y_n dy + CR^{1-n}\int_{S_R^+} u(y) d\sigma_R,$$
which yields the estimate in the general case.

Let us thus assume $R=1$.
Let $\e>0$. Multiplying \eqref{InAA} with $w_\e$ and integrating by parts, we have
$$\int_{B_1^+} u(x_n+\e)^{-\alpha} dx=\int_{B_1^+} u(-\Delta w_\e)dx
=\int_{B_1^+} w_\e gdx+\int_{\partial B_1^+} \bigl(w_\e\partial_\nu u-u\partial_\nu w_\e\bigr)d\sigma.$$
Using $u=w_\e=0$ on  $B^0_1$,
$w_\e=0$ on $S_1^+$, and estimates \eqref{InA}, \eqref{InB}, we obtain
$$\Bigl|\int_{B_1^+} u(x_n+\e)^{-\alpha} dx\Bigr|=\Bigl|\int_{B_1^+} w_\e gdx-\int_{S_1^+}u\partial_\nu w_\e d\sigma\Bigr|
\le C\int_{B_1^+} |g|x_n dx+C\int_{S_1^+}|u| d\sigma.$$
Passing to the limit $\e\to 0$ by dominated convergence, the desired conclusion follows.
\end{proof}

\begin{proof}[Proof of Proposition~\ref{PropresC}]
Let $\varphi\in C^\infty_0(\R)$ be radially symmetric, nonincreasing for $r>0$, with $\varphi=1$ on $[0,1]$ and $\varphi=0$ on $[2,\infty)$.
Set $\phi_R(x)=\varphi(|x|/R)$ and $\gamma=2p/(p-1)$.
We compute
$$\begin{aligned}
\nabla(\phi^\gamma_R)&=\gamma R^{-1}\phi_R^{\gamma-1}\textstyle\frac{x}{|x|}\varphi'(|x|/R),\\ 
|\Delta(\phi^\gamma_R)| &=\gamma R^{-2}\bigl|\phi^{\gamma-1}_R \varphi''(x/R)+(\gamma-1)\phi^{\gamma-2}_R{\varphi'}^2(x/R)\\
&\qquad \phi_R^{\gamma-1}(n-1)(R/|x|)\varphi'(|x|/R) \bigr|
\le CR^{-2}\phi^{\gamma-2}_R, \\ 
\Delta\bigl(x_n\phi^\gamma_R\bigr)
&=x_n\Delta(\phi^\gamma_R)+2e_n\cdot\nabla(\phi^\gamma_R)
=x_n\Delta(\phi^\gamma_R)+2\gamma R^{-1}\phi_R^{\gamma-1}\textstyle\frac{x_n}{|x|}\varphi'(|x|/R).
\end{aligned}$$
Using $|\varphi'(s)|\le Cs$ for all $s\ge 0$ (owing to $\varphi'(0)=0$ and the boundedness of $\varphi''$), 
it follows that
$$|\Delta(x_n\phi^\gamma_R)|\le CR^{-2}x_n\phi^{\gamma-2}_R=CR^{-2}x_n\phi^{\gamma/p}_R.$$
Now multiplying \eqref{InD} with $x_n\phi_R^\gamma$, integrating by parts
and using $\partial_\nu\bigl(\phi_R^\gamma(x)x_n)\bigr)\le 0$ on $B^0_{2R}$,
we have
$$\begin{aligned}
c\int_{B_{2R}^+} v^p\phi_R^\gamma x_n dx
&\le -\int_{B_{2R}^+}\phi_R^\gamma x_n \Delta v dx \\ 
&=-\int_{B_{2R}^+} \Delta(\phi_R^\gamma x_n)vdx
+\int_{B^0_{2R}}
v\partial_\nu\bigl(\phi_R^\gamma(x)x_n)\bigr)d\sigma  \\ 
&\le CR^{-2}\int_{B_{2R}^+}v\phi^\frac{\gamma}{p}_R x_ndx 
\le CR^{-2}\Bigl(\int_{B_{2R}^+}v^p\phi^\gamma_R x_ndx\Bigr)^\frac{1}{p}
\Bigl(\int_{B_{2R}^+}x_ndx\Bigr)^\frac{p-1}{p},
\end{aligned}$$
hence
$$\int_{B_{2R}^+} v^p\phi_R^\gamma x_n dx\le CR^{n+1-2p/(p-1)},$$
and the estimate follows.
\end{proof}

\begin{proof}[Proof of Proposition~\ref{PropresD}]
Fix any $\theta\in(0,1)$ and $\delta>0$.  In this proof $c, C$ denote generic positive contants 
depending only on $f$ and $\theta, \delta$.

\smallskip

{\bf Step 1.} We claim that, for all $R\ge 1$,
\be{interpaa}
\int_{B_R^+} |U|^\frac{p\theta+1}{\theta+1}dx \le CR^{n-{(p\theta+q)\alpha-2\over\theta+1}}+ 
CR^{\frac{p\theta}{p\theta+1}(n-{p\theta+1\over\theta+1}\alpha)}\Bigl(\int_{B_{(1+\delta)R}^+} |U|^\frac{p\theta+1}{\theta+1}dx\Bigr)^\frac{1}{p\theta+1}.
\ee

Applying Proposition~\ref{PropresC}
 to $v=\xi\cdot U$, which satisfies $-\Delta v=\xi\cdot f(U)\ge cv^p$, we have
\be{interpa}
\int_{B_R^+} |U|^px_n dx \le CR^{n+1-p\alpha},\quad R\ge 1.
\ee
By Proposition~\ref{PropresA}, for all $r>1$, using 
$$\int_{B_r^+} |U|^qx_n dx\le Cr^{(n+1)(1-(q/p))}\Bigl(\int_{B_r^+} |U|^px_n dx\Bigr)^{q/p}\le Cr^{n+1-q\alpha},$$
 and \eqref{interpa} we get
$$\begin{aligned}
r^{\theta-1}\int_{B_r^+} |U|x_n^{-\theta}dx
&\le C\int_{B_r^+} (|U|^q+|U|^p)x_n dx + C\int_{S_r^+} |U| d\sigma_r\\
&\le Cr^{n+1-q\alpha}+ C\int_{S_r^+} |U| d\sigma_r.
\end{aligned}$$
Integrating in $r$ over $(R,(1+\delta)R)$ and dividing by $\delta R$, we obtain
\be{interpb}
\begin{aligned}
R^{\theta-1}\int_{B_R^+} |U|x_n^{-\theta}dx
\le CR^{n+1-q\alpha}+ CR^{-1}\int_{B_{(1+\delta)R}^+} |U| dx.
\end{aligned}
\ee
We next interpolate between \eqref{interpa} and \eqref{interpb}, to write
$$\begin{aligned}
\int_{B_R^+} |U|^\frac{p\theta+1}{\theta+1}dx
&=\int_{B_R^+} \bigl(|U|^px_n\bigr)^{\theta\over\theta+1}
\bigl(|U|x_n^{-\theta}\bigr)^{1\over\theta+1}dx\\ 
&\le CR^{(n+1-p\alpha){\theta\over\theta+1}}\Bigl(R^{n+2-\theta-q\alpha}+ R^{-\theta}\int_{B_{(1+\delta)R}^+} |U| dx\Bigr)^{1\over\theta+1}.
\end{aligned}$$
Using
$$\int_{B_{(1+\delta)R}^+} |U| dx\le CR^\frac{n(p-1)\theta}{p\theta+1}\Bigl(\int_{B_{(1+\delta)R}^+} |U|^\frac{p\theta+1}{\theta+1}dx\Bigr)^{\theta+1\over p\theta+1},$$
it follows that
$$\begin{aligned}
&\int_{B_R^+} |U|^\frac{p\theta+1}{\theta+1}dx \\ 
&\le CR^{(n+1-p\alpha){\theta\over\theta+1}+{n+2-\theta-q\alpha\over\theta+1}}+ 
CR^{(n-p\alpha+{n(p-1)\over p\theta+1}){\theta\over\theta+1}}\Bigl(\int_{B_{(1+\delta)R}^+} |U|^\frac{p\theta+1}{\theta+1}dx\Bigr)^\frac{1}{p\theta+1}\\ 
\end{aligned}$$
hence \eqref{interpaa}.

\smallskip

{\bf Step 2.} Set
$$H(R):=R^{-\ell}\int_{B_R^+}  |U|^\frac{p\theta+1}{\theta+1}dx,\qquad \ell:=n-{(p\theta+q)\alpha-2\over\theta+1}.$$
If follows from \eqref{interpaa} that 
$$H(R)\le C+CR^{\frac{p\theta}{p\theta+1}(n-{p\theta+1\over\theta+1}\alpha)-\ell}\bigl(R^\ell H((1+\delta)R)\bigr)^\frac{1}{p\theta+1},
\quad R\ge 1.
$$
Since
$$\begin{aligned}
\textstyle\frac{p\theta}{p\theta+1}(n-{p\theta+1\over\theta+1}\alpha)-\ell+\frac{\ell}{p\theta+1}
&=\textstyle\frac{p\theta}{p\theta+1}(n-{p\theta+1\over\theta+1}\alpha-\ell)
=\textstyle{p\theta\over p\theta+1}\frac{(p\theta+q)\alpha-2-(p\theta+1)\alpha}{\theta+1}\\ 
&=\textstyle\frac{p\theta}{p\theta+1}\frac{(q-1)\alpha-2}{\theta+1}
=\textstyle\frac{2p\theta}{p\theta+1} \frac{q-p}{(p-1)(\theta+1)}\le 0,
\end{aligned}$$
it follows that
$H(R)\le C+C\bigl(H((1+\delta)R)\bigr)^{1/(p\theta+1)}$,
hence $H((1+\delta)R)\ge c_0 (H(R))^{p\theta+1}-1$
for some $c_0=c_0(f,\theta,\delta)\in(0,2]$.
We deduce that 
$\tilde H(R):=(c_0/2)^\frac{1}{ p\theta}H(R)$ satisfies
$$\tilde H((1+\delta)R)\ge c_0(c_0/2)^\frac{1}{p\theta}((c_0/2)^\frac{-1}{p\theta}\tilde H(R))^{p\theta+1}-1
=2(\tilde H(R))^{p\theta+1}-1,\quad R\ge 1.$$
Now assume $\theta>1-(\eta/p)$ and suppose for contradiction that $\tilde H(R)\ge 2$ for some $R\ge 1$.
Letting $R_i:=(1+\delta)^iR$, an immediate induction yields
$\tilde H(R_i)\ge 2$ and $\log\tilde H(R_{i+1})\ge (p\theta+1)\log (\tilde H(R_i))$
for all integer $i\ge 1$,
hence
$\log\tilde H(R_i)\ge \log\tilde H(R)(p\theta+1)^i$
and then
$$\int_{B_{R_i}^+}  |U|^\frac{p\theta+1}{\theta+1}dx\ge c \tilde H(R)e^{(p\theta+1)^i}
\ge c_R\exp\bigl(R_i^\frac{p\theta+1}{1+\delta}\bigr),$$
 provided $\delta$ is small enough.
But, by choosing $\delta>0$ sufficiently small, depending only on $\theta, p, \eta$, this 
produces a contradiction with assumption \eqref{Hypua} due to $\theta>1-(\eta/p)$.
It follows that
$$\int_{B_R^+}  |U|^\frac{p\theta+1}{\theta+1}dx \le CR^{n-\frac{(p\theta+q)\alpha-2}{\theta+1}},\quad R\ge 1.$$
Finally, for $\e>0$ small (with smallness depending only on $p,\eta$), we may choose
$\theta=(1-\alpha\e)/(1+\alpha\e)>1-(\eta/p)$, 
so that $\frac{p\theta+1}{\theta+1}=\frac{p+1}{2}-\e$.
Noting that $\frac{(p\theta+q)\alpha-2}{\theta+1}
=(\frac{p+1}{2}-\e)\alpha-\frac{p-q}{ p-1}(1+\eps\alpha)$
we obtain the desired conclusion.
\end{proof}

\subsection{Proof of Corollary~\ref{thm1cor0} and Proposition~\ref{thm1cor0prop}} 
 \label{subVerif}
 
 In order to apply Theorem~\ref{thm1}, 
 we shall use the following two lemmas which collect the needed properties of the function $g$.
 
 \begin{lemma} \label{lemf1f2}
 Let $g\in C^1([0,\infty))$ be convex, nonnegative and satisfy $g'>0$ on $(0,\infty)$.
Let $f$ be defined by \eqref{DeffiCor0} with $0<\lambda\le 1$.
 Then 
\be{lemf1f2a}
f_1+f_2\ge 2(1-\lambda)\bigl(g'(u_1)g(u_1)+g'(u_2)g(u_2)\bigr)\ge 0
  \quad \hbox{on $[0,\infty)^2$}
  \ee
 and
\be{lemf1f2b}
u_1f_1+u_2f_2\ge 2(1-\lambda)\bigl(u_1g'(u_1)g(u_1)+u_2g'(u_2)g(u_2)\bigr)\ge 0
 \quad \hbox{on $[0,\infty)^2$.}
 \ee
 \end{lemma} 
 
 \begin{proof}
   Denote $g_i=g(u_i)$, $g'_i=g'(u_i)$ for conciseness.
 We have
\be{eqf1f2}
f_1(U)=2g'_1[g_1-\lambda g_2],\quad f_2(U)=2g'_2[g_2-\lambda g_1],
\ee
hence 
$$ \begin{aligned} 
(f_1+f_2)/2
&=(1-\lambda)\bigl(g'_1g_1+g'_2g_2\bigr) +\lambda(g'_1g_1+g'_2g_2 -g'_1g_2-g'_2g_1\bigr)\\
&=(1-\lambda)\bigl(g'_1g_1+g'_2g_2\bigr) +\lambda(g_1-g_2)(g'_1-g'_2)
\end{aligned}$$
 and
$$ \begin{aligned} 
(u_1f_1+u_2f_2)/2
&=(1-\lambda)\bigl(u_1g'_1g_1+u_2g'_2g_2\bigr) +\lambda(u_1g'_1g_1+u_2g'_2g_2 -u_1g'_1g_2-u_2g'_2g_1\bigr)\\
&=(1-\lambda)\bigl(u_1g'_1g_1+u_2g'_2g_2\bigr) +\lambda(g_1-g_2)(u_1g'_1-u_2g'_2).
\end{aligned}$$
Since $g, g'$ are nondecreasing by assumption, it follows that
$g_1-g_2$, $g'_1-g'_2$  and $u_1g'_1-u_2g'_2$ have the same sign.
Properties \eqref{lemf1f2a} and \eqref{lemf1f2b} follow.
 \end{proof}
 
 \begin{lemma} \label{lemhK}
Let $p_0>1$, $K\ge 1$, $a\in\R$, $\sigma\in\{-1,1\}$
and let $g$ be defined by \eqref{DeffiCor0}. We have
\be{sghK0}
\frac{sg'(s)}{g(s)}=b+ah_\sigma(s),\quad s>0,
\ee
where $b=(p_0+1)/2$, 
\be{defhKphi}
h_{\sigma}(s)=\frac{\sigma}{\phi_K(s^{\sigma})},\quad \phi_K(s)=(1+K{s^{-1}}) \log(K+s)>0,\quad s>0,
\ee
and
\be{infsuphK}
\inf_{s>0}\phi_K(s)=\theta(K),
\ee
where $\theta:[1,\infty)\to [1,\infty)$ is the inverse bijection of $s\mapsto s^{-1}e^{s-1}$.
If, moreover, $\sigma a>0$, then
\be{gconvex0}
\hbox{$g'>0$ and $g''\ge 0$ on $(0,\infty)$.}
\ee
\end{lemma}

 \begin{proof}
We have
$$ \phi_K'(s)=-Ks^{-2} \log(K+s)+(1+Ks^{-1}) (K+s)^{-1}=s^{-2}\bigl(s-K \log(K+s)\bigr).$$
For $K=1$, $\phi_1$ is increasing with $\lim_{t\to 0^+}\phi_1(t)=1$ and $\lim_{t\to \infty}\phi_1(t)=\infty$,
hence $\inf_{s>0}\phi_1(s)=1=\theta(1)$.

Assume $K>1$. Then $\lim_{t\to 0^+}\phi_K(t)=\lim_{t\to \infty}\phi_K(t)=\infty$.
Setting $\tau=K^{-1}s$, we see that $\phi_K'(s)=0$ is equivalent to
$K=\zeta(\tau):=(\tau+1)^{-1} e^\tau$. 
Since $\zeta'(\tau)>0$ for all $\tau>0$ with $\zeta(0)=1$ and $\lim_{\tau\to\infty}\zeta(\tau)=\infty$,
this equation has a unique solution $\tau=\tau_K>0$, given by $\tau_K=\theta(K)-1$.
At the point $s=s_K=K\tau_K$ we have $\log(K+s)=K^{-1}s$
hence $\phi_K(s)=(1+Ks^{-1}) \log(K+s)
=1+K^{-1}s=1+\tau_K=\theta(K)$.
This proves property~\eqref{infsuphK}.

 Next, for $\sigma=1$, we compute
\be{gprimehK}
g'(s)=bs^{b-1}\log^a(K+s)+a(K+s)^{-1}s^b\log^{a-1}(K+s)
=s^{-1}g(s)\bigl[b+ah_1(s)\bigr]
\ee
and, for $\sigma=-1$,
\be{gprimehK0}
g'(s)=bs^{b-1}\log^a(K+s^{-1})-\frac{s^{-2}}{K+s^{-1}}as^b\log^{a-1}(K+s^{-1})
=s^{-1}g(s)\bigl[b+ah_{-1}(s)\bigr],
\ee
hence \eqref{sghK0}. This in particular yields the first part of \eqref{gconvex0}.

We turn to the second part of \eqref{gconvex0}.
In the case $\sigma=1$, $a>0$, we compute
$$ \begin{aligned}
g''(s)
&=b(b-1)s^{b-2}\log^a(K+s)+2ab(K+s)^{-1}s^{b-1}\log^{a-1}(K+s)\\
&\qquad +as^b[(K+s)^{-1}\log^{a-1}(K+s)]'\\
&\ge 2ab(K+s)^{-1}s^{b-1}\log^{a-1}(K+s)\\
&\qquad +as^b(K+s)^{-2}\bigl[(a-1)\log^{a-2}(K+s)-\log^{a-1}(K+s)\bigr],
\end{aligned}  $$
hence
$$ \begin{aligned}
{s^{1-b}(K+s)^2 g''(s)\over a\log^{a-2}(K+s)}
&\ge 2b(K+s)\log(K+s)+s\bigl[(a-1)-\log(K+s)\bigr] \\
&= \bigl(2bK+(2b-1)s\bigr)\log(K+s)+(a-1)s\\
&\ge (2+s)\log(1+s)-s=:\xi(s).
\end{aligned}  $$
Since $\xi(0)=0$ and $\xi'(s)=\log(1+s)+(2+s)(1+s)^{-1}-1\ge 0$ for $s\ge 0$,
it follows that $\xi\ge 0$, 
hence $g''\ge 0$, for $s>0$.

\smallskip

In the case $\sigma=-1$, $a<0$, setting $\psi(s)=[(1+Ks) \log(K+s^{-1})]^{-1}$, we have
$$ \begin{aligned} 
\psi'(s)
&=-\bigl[(1+Ks) \log(K+s^{-1})\bigr]'\psi^2\\
&=-\bigl[K \log(K+s^{-1})-s^{-2}(1+Ks) (K+s^{-1})^{-1}\bigr]\psi^2\\
&=\bigl[s^{-1}-K \log(K+s^{-1})\bigr]\psi^2
=\bigl[s^{-1}\psi-K (1+Ks)^{-1}\bigr]\psi
\end{aligned}  $$
and, using \eqref{gprimehK0},
$$g''(s)=(s^{-1}g'(s)-s^{-2}g(s))\bigl[b-a\psi(s)\bigr]-as^{-1}g(s)\psi'(s).$$
Consequently,
$$ \begin{aligned} 
{s^2g''(s)\over g(s)}
&=\Bigl({sg'(s)\over g(s)}-1\Bigr)\bigl[b-a\psi(s)\bigr]-as\psi'(s)\\
&=\bigl(b-1-a\psi(s)\bigl)(b-a\psi(s))+a\psi(s)\bigl(Ks (1+Ks)^{-1}-\psi(s)\bigr) \\
&=b(b-1)-a\psi(s)\Bigl\{(1-a)\psi(s)+2b-1-Ks (1+Ks)^{-1}\Bigr\}\\
&=b(b-1)-a\psi(s)\Bigl\{(1-a)\psi(s)+2(b-1)+(1+Ks)^{-1}\Bigr\}\ge 0.
\end{aligned}  $$
This completes the proof of \eqref{gconvex0}. \end{proof}

 \begin{proof}[Proof of Corollary~\ref{thm1cor0}]
Since $\theta(1)=1$, it suffices to prove assertion (ii) (for all $K\ge 1$).
The nonlinearities in \eqref{DeffiCor0} are given by \eqref{eqf1f2}.
 Denote $g_i=g(u_i)$, $g'_i=g'(u_i)$ for conciseness. 
We shall check the applicability of Theorem~\ref{thm1}.

 We claim that conditions \eqref{hypGrowthB} and \eqref{hypGrowthD} hold with
\be{defpqa}
\begin{cases}
q=p=p_0,&\hbox{ if $\sigma=1$ and $K>1$} \\
\noalign{\vskip 1mm}
q=p=p_0+2a,&\hbox{ if $\sigma=1$ and $K=1$} \\
\noalign{\vskip 1mm}
q=p_0, \ p=p_0+\eps,&\hbox{ if $\sigma=-1$ (for arbitrary $\eps>0$).}
\end{cases}
\ee
(We will use Theorem~\ref{thm1} with $p_0$ replaced by $p_0+2a$ if $\sigma=K=1$.)

To this end, we note that, by \eqref{eqf1f2}, \eqref{sghK0} and the boundedness of $h_\sigma$ (cf.~\eqref{defhKphi}-\eqref{infsuphK}), we have
\be{boundf1f2}
|f_1|+|f_2|\le C\bigl(u_1^{-1}g_1^2+u_2^{-1}g_2^2+(u_1^{-1}+u_2^{-1})|g_1g_2|\bigr).
\ee
In the case $\sigma=1$, $a>0$, for $s\le M$, we have
$g(s)=s^b\log^a(K+s)\le C_M s^b$ if $K>1$
and $g(s)=s^b\log^a(1+s)\le C_M s^{b+a}$ if $K=1$.
In the case $\sigma=-1$, $a<0$, we have
$g(s)=s^b\log^a(K+s^{-1})\le C_M s^b$ for $s\le M$.
The claim for $q$ in \eqref{hypGrowthB} then follows from~\eqref{boundf1f2}.

Next, by  \eqref{lemf1f2a} 
we have
$f_1+f_2
\ge 2(1-\lambda)\bigl(g'_1g_1+g'_2g_2\bigr).$
 Owing to \eqref{sghK0} and the assumption $\sigma a>0$,  it follows that
$$ f_1+f_2\ge 2b(1-\lambda)\bigl(u_1^{p_0}\log^{2a}(K+u_1^\sigma)+u_2^{p_0}\log^{2a}(K+u_2^\sigma)\bigr).$$
This yields the claim for $p$ in \eqref{hypGrowthD}.

We turn to condition \eqref{hypGrowthC}.
For $p$ given by \eqref{defpqa}, we note that, for all $U\in [0,M]^2$, we have
$g_1^2+g_2^2\ge \tilde c_M|U|^{p+1}$, 
hence $F(U)\ge (1-\lambda)\tilde c_M|U|^{p+1}$, 
for some $\tilde c_M>0$.
 Since $U\cdot\nabla F(U)\ge 0$ owing to \eqref{lemf1f2b}, this in particular guarantees \eqref{hypGrowthC} for $n\le 2$.
Assume $n\ge 3$.
A sufficient condition for \eqref{hypGrowthC} is thus
the existence of $\eta>0$ such that, for all $U\in [0,M]^2$,
\be{hypGrowthC2}
(p_S+1-\eta)[g_1^2+g_2^2-2\lambda g_1g_2]
\ge 2u_1g'_1[g_1-\lambda g_2]+2u_2g'_2[g_2-\lambda g_1].
\ee
For $0<\eta<p_S-p_0$, denoting $\zeta_i=h_\sigma(u_i)$, setting $\delta:=p_S-p_0=p_S+1-2b$, 
$A:=\delta-\eta-2a\zeta_1$, $B:=\delta-\eta-2a\zeta_2$
and using \eqref{sghK0}, we have
$$ \begin{aligned} 
\eqref{hypGrowthC2}
&\Leftrightarrow
(p_S+1-\eta)[g_1^2+g_2^2]-2u_1g'_1g_1-2u_2g'_2g_2 \\
&\qquad\qquad\qquad\qquad \ge
2\lambda \bigl\{(p_S+1-\eta)g_1g_2-u_1g'_1g_2-u_2g'_2g_1\bigr\}\\
&\Leftrightarrow\bigl(\delta-\eta-2a\zeta_1\bigr)g_1^2+\bigl(\delta-\eta-2a\zeta_2\bigr)g_2^2
\ge 2\lambda \bigl\{\delta-\eta-a\zeta_1-a\zeta_2\bigr\}g_1g_2 \\
&\Leftrightarrow Ag_1^2+Bg_2^2-\lambda (A+B)g_1g_2\ge 0 \\
&\Leftrightarrow\bigl(\sqrt{A}g_1-\sqrt{B}g_2\bigl)^2+\bigl(2\sqrt{AB}-\lambda (A+B)\bigr)g_1g_2 \ge 0,
\end{aligned}  $$
provided $A,B\geq0$.
Now, by \eqref{infsuphK}, we have $ah_\sigma(s)\ge 0$ and
$\sup_{s>0}|h_\sigma(s)|=\beta:={1\over \theta(K)}$. 
By our assumption that $|a|<a_0={p_S-p_0\over 2}\theta(K)={\delta\over 2\beta}$, for $0<\eta<\delta-2|a|\beta$,
we get  $A, B\in \Sigma_\eta:=[\delta-\eta-2|a|\beta,\delta-\eta]$,
hence  $A, B>0$ and $\sqrt{A/B}\in\tilde \Sigma_\eta:=\bigl[(1-\rho_\eta)^{1/2},(1-\rho_\eta)^{-1/2}\bigr]$,
where $\rho_\eta={2|a|\beta\over\delta-\eta}\in(0,1)$.
On the other hand, we have
$$ \begin{aligned} 
\sup_{A,B\in\Sigma_\eta} {A+B\over \sqrt{AB}}
&=\sup_{A,B\in\Sigma_\eta} \sqrt{A\over B}+\sqrt{B\over A}=\sup_{X\in\tilde \Sigma_\eta} X+X^{-1}\\
&=\sqrt{1-\rho_\eta}+{1\over \sqrt{1-\rho_\eta}}={2-\rho_\eta\over \sqrt{1-\rho_\eta}}\to {2-\rho\over \sqrt{1-\rho}} = {2\over\lambda_0},
\quad\hbox{as $\eta\to 0$.}
\end{aligned}  $$
Since $0<\lambda<\lambda_0$, by choosing $\eta>0$ small enough, we get
$2\sqrt{AB}-\lambda (A+B)\ge 0$, and condition \eqref{hypGrowthC2} is satisfied, hence \eqref{hypGrowthC} is true. 
\end{proof}

\begin{proof}[Proof of Proposition~\ref{thm1cor0prop}]
We show that there exists a radial, positive bounded solution of 
$-\Delta u=f(u)\equiv cg(u)g'(u)$ with $c=2(1-\lambda)$, which will immediately provide a solution $(u_1,u_2)=(u,u)$ of
\eqref{GenSystem}.

To this end we use a small modification of an argument from \cite{QS12, SouDCDS} 
(see \cite[Theorem~1.4]{QS12} and \cite[Proposition~1]{SouDCDS}).
Let $H(s)=\int_0^s f(z)dz$ and $\psi(s)=sf(s)-(p_S+1)H(s)$. 
We first claim that there exists $s_0>0$ such that
\be{signphi}
\psi(s)\ge 0\quad\hbox{for all $s\in [0,s_0]$}.
\ee
Indeed, we have $g(s)\sim s^{b+a}$ as $s\to 0$, with $b=(p_0+1)/2>1$. By \eqref{gprimehK}
and the fact that $\lim_ {s\to 0^+}h_1(s)=1$, we deduce that $g'(s)\sim (b+a)s^{b+a-1}$, hence
$f(s)\sim \tilde cs^{2(b+a)-1}$ with $\tilde c=(b+a)c$.
Consequently, $H(s)\sim (2(b+a))^{-1}\tilde cs^{2(b+a)}$ and, since $a>a_0={p_S+1\over 2}-b$, it follows that 
$\psi(s)\sim [1-(p_S+1)(2(b+a))^{-1}]\tilde cs^{2(b+a)}>0$ as $s\to 0$, hence \eqref{signphi}.

The end of the proof is then the same as in \cite[Proposition~1]{SouDCDS}, which we repeat for convenience.
Extend $f$ by $0$ for $s<0$ and consider the initial value problem
$$-(r^{n-1}v')'= r^{n-1}f(v),\  r > 0, \ \quad v(0) = s_0, \quad v'(0) = 0.$$
Let $R^*$ be its maximal existence time. Since $f\ge 0$, we have $v'\le 0$ on $[0,R^*)$. 
Assume for contradiction that $v$ has a (first) zero $R\in (0,R^*)$. Then $v$ is a 
solution of $-\Delta v=f(v)$ on the ball $B_R$ with Dirichlet boundary conditions and $0\le v\le s_0$.
But in view of \eqref{signphi} and Pohozaev's inequality (see, e.g.,~\cite[Corollary 5.2]{QSb}), this is a contradiction.
We conclude that $v>0$ on $(0,\infty)$, which proves the desired result.
\end{proof}

\section{Proof of  Theorems~\ref{thm-proportional2}, \ref{thmGSstarCor} and \ref{thmGSstar}}
\label{ProofGS}

\begin{proof}[Proof of Theorem~\ref{thm-proportional2}]
By Theorem~\ref{thm-proportional}, we have $u=v$ and $-\Delta u=(1-\lambda)h(u)$.
In the case $\lambda\in[0,1)$,
the conclusion then follows from Theorem~B for $D=\Rn$,
and from \cite{CLZ} or \cite{DSS} for $D=\Rn_+$.

Let us consider the case $\lambda>1$. Let $M=\frac12\sup_D u\in[0,\infty)$ 
and assume for contradiction that $M>0$.
Choose $x_0\in D$ such that $u(x_0)\ge M$ and, for any $\delta>0$, set $u_\delta(x)=u(x)-\delta |x-x_0|^2$.
Since $u_\delta(x)\to -\infty$ as $|x|\to\infty$ (and $u_\delta<0$ on $\partial D$ in case $D=\Rn_+$),
 there exists $x_\delta\in D$ such that $u_\delta(x_\delta)=\sup u_\delta$.
Let $\tau=\inf_{s\in[M,2M]}h(s)>0$. Since $u(x_\delta)\ge u_\delta(x_\delta)\ge u_\delta(x_0)=u(x_0)\in[M,2M]$,
we have $0\ge \Delta u_\delta(x_\delta)=(\lambda-1)h(u(x_\delta))-2n\delta\ge (\lambda-1)\tau-2n\delta>0$
for $\delta>0$ small: a contradiction.
\end{proof}

In view of the proof of Theorem~\ref{thmGSstar}, 
we first recall the following lemma (see~\cite[Lemma 9.2]{SouDCDS}):

\begin{lemma} \label{lemGS1}
Let $\Omega$ be an arbitrary domain in $\Rn$,
$0\leq \varphi\in \mathcal{D}(\Omega)$, and let
$u\in W^{2,r}_{loc}(\Omega)$ for all finite $r$, with $u>0$.
Fix $q\in \R$ and denote
\be{IJK}
I=\int \varphi\,u^{q-2}|\nabla u|^4,\quad
J=\int \varphi\,u^{q-1}|\nabla u|^2 \Delta u,\quad
K=\int \varphi\,u^q(\Delta u)^2,
\ee
where $\int g=\int_\Omega g(x)\,dx$.
Then, for any $k\in \R$ with $k\neq -1$, there holds
\be{estIJK}
\alpha I+\beta J+\gamma K\leq
{1\over 2}  \int \,u^q|\nabla u|^2\Delta \varphi
+ \int u^q\bigl[\Delta u+(q-k)u^{-1}|\nabla u|^2\bigr]\nabla u\cdot\nabla\varphi,
\ee  
where
$$\alpha=-{n-1\over n}k^2+(q-1)k-{q(q-1)\over 2}, \quad
\beta={n+2\over n}k-{3q\over 2}, \quad
\gamma=-{n-1\over n}.
$$
\end{lemma}

\begin{proof}[Proof of Theorem~\ref{thmGSstar}]
 Assume for contradiction that \eqref{scaleq} admits a nontrivial solution. 
Since $f\ge 0$, the strong maximum principle implies $u>0$ in $\Rn$.

 In view of applying Lemma~\ref{lemGS1}, let us prepare a suitable test-function  $\varphi$. We take $\xi\in 
\mathcal{D}(B_1)$ such that $\xi=1$ in $B_{1/2}$ and $0\leq \xi\leq 1$.
 Due to $m_i>2/3$,
we may fix $a$ such that 
$$\max\Bigl\{\frac{m_i+2}{4m_i}, \frac12\Bigr\}<a<1,\quad i\in\{1,2\}.$$  
If $q>-2$, owing to $\gamma_i>1$, we may also assume
$$a>\frac{\gamma_i+1}{2\gamma_i},\quad i\in\{1,2\}.$$  
By taking $\varphi(x)=\varphi_R(x)=\xi^b(x/R)$ with $b=b(a)>2$ sufficiently large, we have
\be{estimtest}
|\nabla\varphi_R|\leq CR^{-1}\varphi^a, \quad
|\Delta \varphi_R|+\varphi_R^{-1}|\nabla\varphi_R|^2\leq CR^{-2}\varphi^a
\ee 
(where $\varphi_R^{-1}|\nabla\varphi_R|^2$ is defined to be $0$ whenever $\varphi=0$).

 Let $R\ge 1$, $\Omega=B_R$, $I,J,K$ be defined by \eqref{IJK}, and set also
$$ L=\int \varphi\,f(u)F_q(u). $$
The equation  $-\Delta u=f(u)$ yields 
$$ 
J = -\int\varphi u^{q-1} f(u)|\nabla u|^2 
   = -\int\varphi \nabla F_q(u)\cdot\nabla u \\ 
  = -\int\varphi F_q(u) f(u) 
  +\int F_q(u)\nabla u\cdot\nabla\varphi$$
  i.e.,  
\be{eqJL}
J =-L +\int F_q(u)\nabla u\cdot\nabla\varphi, 
\ee
 and \eqref{hypGS20} implies
\be{eqK}
K = \int\varphi u^q f^2(u)  
  \leq c_q L.
\ee
 Combining \eqref{estIJK}, \eqref{eqJL}, \eqref{eqK} and using \eqref{hypGS20} again, we obtain
\be{eqILI}
c_0(I+L)\leq I_1+C(I_2+I_3),
\ee
where $c_0=\min(\alpha,-\beta+c_q\gamma)>0$,
$$
I_1= \int \,u^q|\nabla u|^2|\Delta \varphi|,\quad
 I_2= \int F_q(u)|\nabla u\cdot\nabla\varphi|,\quad
 I_3=\int u^{q-1}|\nabla u|^2|\nabla u\cdot\nabla\varphi|,
$$
 and $C\geq1$ ($C$ also satisfies \eqref{hypGS3}, \eqref{hypGS4}, \eqref{estimtest}).
Set also
$$
 I_4=\int\frac1\varphi u^{2+q}\frac{\varphi^{2a}}{R^4},\quad
 I_5=\int u^q|\nabla u|^2\frac{|\nabla\varphi|^2}\varphi,\quad
 I_6=\int \bigl(u^{\frac{2-q}4}F_q(u)\varphi^{-\frac14}|\nabla\varphi|\bigr)^{4/3} 
$$
and fix $\eta\in(0,c_0/(10C))$. Then there exists $C_\eta\geq1$ such that
$$ \begin{aligned}
 I_3 &=\int u^{q-1}|\nabla u|^2|\nabla u\cdot\nabla\varphi|
 \leq\int u^{q-1}|\nabla u|^2\Bigl(\eta\varphi\frac{|\nabla u|^2}u
 +C_\eta u\frac{|\nabla\varphi|^2}\varphi\Bigr) 
 = \eta I+C_\eta  I_5, \\
   I_2 &=\int F_q(u)|\nabla u\cdot\nabla\varphi| 
  \leq\int\bigl(\varphi^{\frac14}u^{\frac{q-2}4}|\nabla u|\bigr)
 \bigl(u^{\frac{2-q}4}F_q(u)\varphi^{-\frac14}|\nabla\varphi|\bigr)  
 \leq\eta I+C_\eta  I_6.
\end{aligned}
$$
Fix $\eta'\in(0,\eta/(CC_\eta))$.  Then, using \eqref{hypGS4} and \eqref{estimtest}, there exists $C_{\eta'}\geq1$ such that 
$$ \begin{aligned}
I_1&+ I_5 =
 \int u^q|\nabla u|^2\Bigl(|\Delta\varphi|+\frac{|\nabla\varphi|^2}{\varphi}\Bigr)
\leq\int\bigl(\varphi^{\frac12}u^{\frac{q-2}2}|\nabla u|^2\bigr)
 \bigl(\varphi^{-\frac12}u^{\frac{2+q}2}C\frac{\varphi^a}{R^2}\bigr)
\leq \eta' I+C_{\eta'}  I_4, \\
  I_6 &=\int \bigl(u^{\frac{2-q}4}F_q(u)\varphi^{-\frac14}|\nabla\varphi|\bigr)^{4/3}
    = \sum_i \int_{B_R^i} \bigl(\varphi^{\frac1{2m_i}}u^{\frac{2-q}4}F_q(u)\bigr)^{4/3}
            \bigl(\varphi^{-\frac14-\frac1{2m_i}}|\nabla\varphi|\bigr)^{4/3} \\
   &\leq \eta'\sum_i\int_{B_R^i}\varphi\bigl(u^{\frac{2-q}4}F_q(u)\bigr)^{2m_i}
 + C_{\eta'}\sum_i\int_{B_R^i}\bigl(\varphi^{-\frac{m_i+2}{4m_i}}|\nabla\varphi|\bigr)^{\frac{4m_i}{3m_i-2}} 
 \leq \eta'CL+CC_{\eta'}R^{\theta_1},
\end{aligned} $$
where $\theta_1=n-\frac{4m_i}{3m_i-2}<0$ and $B_R^i=\{x\in B_R:u(x)\in\omega_i\}$, $i=1,2$.
Finally, fix $\eta''\in(0,\eta/(CC_\eta C_{\eta'}))$.  If $q>-2$ then, using \eqref{hypGS3}, there exists $C_{\eta''}\geq1$ such that
$$  
 I_4 \leq\eta''\sum_i\int_{B_R^i} 
 \varphi u^{(2+q)\gamma_i}
 +C_{\eta''}\sum_i\int_{B_R^i}\Bigl(\varphi^{-\frac{\gamma_i+1}{\gamma_i}+2a}
      R^{-4}\Bigr)^\frac{\gamma_i}{\gamma_i-1} 
 \leq \eta'' CL+C_{\eta''} CR^{\theta_2},
$$
where $\theta_2=n-\frac{4\gamma_i}{\gamma_i-1}<0$.
 If $q=-2$ and $n=3$, since $a>1/2$, we just notice that
$$I_4 \leq CR^{\theta_2},\quad\hbox{with $\theta_2=-1$}.$$
 Estimate \eqref{eqILI} and the above inequalities for $I_1,\dots,I_6$ imply
$$ I+L\leq C'R^\theta,$$ 
where $\theta=\max(\theta_1,\theta_2)<0$ and $C'=C'(C,c_0,C_\eta,C_{\eta'},C_{\eta''})$.
Consequently,
\be{fueps}
 \int_{B_{R/2}}f(u)F_q(u)\leq L 
\leq C'R^\theta. 
\ee
Letting $R\to\infty$ and using \eqref{hypGS3} again, we conclude that $u\equiv0$:  a contradiction.
\end{proof}
 
\begin{proof}[Proof of Theorem~\ref{thmGSstarCor}]
In the case $p_1\le \kappa-1$, the conclusion follows from known results for elliptic inequalities; see \cite[Theorem~2.1]{AS}.
Let us thus assume $p_1>\kappa-1$.
We shall use Theorem~\ref{thmGSstar} with $q=1-\kappa =-2/(n-2)$ and $c_q=\bar\kappa$.
Assumptions \eqref{hypDefq}-\eqref{hypGS200}
with $\e>0$ small guarantee \eqref{f-Lip} with $p=p_1-\eps>\kappa-1=-q$
 and \eqref{hypGS20}-\eqref{hypGS4}, as well as \eqref{hypGS3} if $q>-2$ (i.e.,~$n\ge 4$).
Set $k_\delta=-\frac{n}{n-2}+\delta$ with $\delta>0$ and let $\alpha=\alpha_\delta,\beta=\beta_\delta$ and $\gamma$ be defined by \eqref{alphabetagamma}
with $k=k_\delta$.
Then $-\beta_0+c_q\gamma=\frac{n-1}{n-2}-\frac{n-1}{n}c_q>0$, $\alpha_0=0$ and 
$(\frac{\partial\alpha}{\partial k})_{|k=k_0}=-\frac{n-1}{n}2k_0+q-1=1$.
Consequently, $\alpha_\delta>0$ and $-\beta_\delta+c_q\gamma>0$ for $\delta>0$ small
and the conclusion follows from Theorem~\ref{thmGSstar}.
\end{proof}

\section{Proof of Proposition~\ref{propdecay}, Theorems~\ref{thmLEpert1}, \ref{thmLEpert1B} and \ref{thmUB}}

\subsection{Proof of Proposition~\ref{propdecay}}
Assume that the assertion fails. Then there exist $\delta>0$, 
a sequence $R_j\to\infty$ and solutions $U_j$ of $-\Delta U_j=f(U_j)$ on $B_{R_j}$, such that $\|U_j\|_\infty\le\Lambda$
and $|U_j(0)|\ge \delta$.
Since $|f(U_j)|\le M:=\sup_{|\xi|\le\Lambda}|f(\xi)|<\infty$, 
it follows from interior elliptic $L^q$ estimates that some subsequence of $U_j$ converges locally uniformly to a 
solution $U_\infty$ of $-\Delta U_\infty=f(U_\infty)$ on $\Rn$.
Moreover, $\|U_\infty\|_\infty\le\Lambda$ and $|U_\infty(0)|\ge\delta$. This contradicts our assumption.

\subsection{Proof of Theorem~\ref{thmLEpert1}}
The proof of Theorem~\ref{thmLEpert1}
is based on a doubling-rescaling argument for systems from \cite[Theorem 4.3]{PQS1},
combined with the rescaling approach from \cite{SouDCDS} for scalar problems with a regularly varying nonlinearity.
However, it requires a rather delicate rescaling procedure to handle regularly varying nonlinearities
of different indices in the two equations.
To this end, for a given solution $(u,v)$ of \eqref{LEsystem1} in a domain $\Omega$,
we shall introduce the key auxiliary function $\phi(N(u,v))$, where
\be{defLEphi}
\phi(t)={t\over h_1^{-1}(t)h_2^{-1}(t)}>0, \quad t\ge s_0,
\ee
\be{defLEphiN}
N(t_1,t_2)=2A+h_1(t_2)+h_2(t_1),\quad t_1, t_2\ge 0,
\ee
and the constant $A=A(f_1,f_2)>0$ is given by the following lemma.
We also set 
\be{defLEphiN2}
A_i=\inf_{s\ge 0}h_i(s)
\ee
(which is finite since $\lim_{s\to\infty}h_i(s)=\infty$).

\begin{lemma} \label{LemphiMonot}
There exists $A=A(f_1,f_2)>2\max(s_0,-A_1,-A_2)$ such that 
\be{infphi0}
\hbox{$\phi$ is increasing on $[A,\infty)$}
\ee
and
\be{infphi00}
\inf_{t\ge A}{g(\lambda t)\over g(t)}\to\infty,
\ \hbox{as $\lambda\to\infty$, \quad for $g\in\{h_1,h_2,\phi\}$.}
\ee
Moreover, we have 
\be{infphi001}
N(t_1,t_2)\ge \max\bigl(A,h_1(t_2),h_2(t_1)\bigr),\quad t_1,t_2\ge 0.
\ee
\end{lemma}

\begin{proof}
Fix $\eps\in(0,1)$ to be chosen later. 
By \eqref{hypLE1}-\eqref{hypLE2}, the function
$z(s):=s^{\eps-p-1}h_1(s)=s^\eps L_1(s)$ is $C^1$ for $s$ large enough with 
$$z'(s)=\Bigl(\eps+{sL_1'(s)\over L_1(s)}\Bigr)s^{\eps-1}L_1(s)>0.$$
This and the similar argument for $h_2$
guarantee the existence of $s_1\ge s_0$ such that
\be{infphi1}
\hbox{the functions ${h_1(s)\over s^{p+1-\eps}}$
and ${h_2(s)\over s^{q+1-\eps}}$ are increasing on $[s_1,\infty)$.}
\ee
It follows that
${h_1(\lambda t)\over h_1(t)}\ge \lambda^{p+1-\eps}$
and ${h_2(\lambda t)\over h_2(t)}\ge \lambda^{q+1-\eps}$ for all $t\ge s_1$
and $\lambda\ge1$,
hence \eqref{infphi00} for $g\in\{h_1,h_2\}$ and any $A\ge s_1$.

\smallskip

Set $t_1:= h_1(s_1)\vee h_2(s_1)$ and pick any $t'>t\ge t_1$.
Then $s':=h_1^{-1}(t')\ge h_1^{-1}(t)=:s$ and it follows from \eqref{infphi1} that
${(h_1(s'))^{1/(p+1-\eps)}\over s'}\ge {(h_1(s))^{1/(p+1-\eps)}\over s}$,
hence
\be{infphi2}
{{t'}^{1/(p+1-\eps)}\over h_1^{-1}(t')}\ge {t^{1/(p+1-\eps)}\over h_1^{-1}(t)}
\quad\hbox{and (similarly)}\quad {{t'}^{1/(q+1-\eps)}\over h_2^{-1}(t')}\ge {t^{1/(q+1-\eps)}\over h_2^{-1}(t)}.
\ee
Set $\theta=\theta_\eps=1-\bigl({1\over p+1-\eps}+{1\over q+1-\eps}\bigr)$.
Since $\theta_\eps\to {pq-1\over (p+1)(q+1)}>0$ as $\eps\to 0$, we may choose $\eps=\eps(p,q)>0$ small enough so that $\theta>0$.
Setting $\psi(t)=t^{-\theta}\phi(t)$ and multiplying the two inequalities in \eqref{infphi2} yields $\psi(t')\ge \psi(t)$,
i.e.~$\psi$ is nondecreasing on $[t_1,\infty)$.
This in particular implies \eqref{infphi0} for any $A\ge t_1$.

\smallskip

Also, for all $t\ge t_1$ and $\lambda>1$, we have 
${\phi(\lambda t)\over \phi(t)}=\lambda^\theta{\psi(\lambda t)\over \psi(t)}\ge \lambda^\theta$,
hence \eqref{infphi00} for $g=\phi$ and any $A\ge t_1$. 
Property \eqref{infphi001} is then immediate, and the proof is complete.
\end{proof}

\begin{proof}[Proof of Theorem~\ref{thmLEpert1}(i)]
{\bf Step 1.} {\it Equivalent estimate.} Recall \eqref{h1h2mon}, \eqref{defLEphi} and~\eqref{defLEphiN}.
We shall show that, for any given 
solution $(u,v)$ of \eqref{LEsystem1} in a domain $\Omega$, we have
\be{defMkLE0}
\phi(N(u,v))\le C(1+d^{-2}(x)),\quad x\in \Omega,
\ee
with $C=C(n,f_1,f_2)>0$
(note that $\phi(N(u,v))$ is well defined, owing to \eqref{infphi0}, \eqref{infphi001}).

We claim that \eqref{defMkLE0} implies \eqref{estimLE1}.
Indeed, let $x\in \Omega_1$.
If $h_2(u(x))\ge A$, then at this point
$2A+h_1(v)+h_2(u)\ge h_2(u)\ge A$ and \eqref{infphi0}, \eqref{defMkLE0} imply
$${f_2(u)\over h_1^{-1}(h_2(u))}=\phi(h_2(u))\le \phi(N(u,v))\le C(1+d^{-2}(x)).$$
If $h_2(u(x))\le A$ then, since $\lim_\infty h_2=\infty$, we have $u(x)\le C$ with $C=C(f_1,f_2)>0$.
The argument is similar for $x\in \Omega_2$. This proves the claim.

\smallskip

{\bf Step 2.} {\it Doubling-rescaling argument.} 
Thus assume for contradiction that estimate \eqref{defMkLE0} fails. 
Then, there
exist sequences
$\Omega_k$, $(u_k,v_k)$, $z_k\in\Omega_k$, such that $(u_k,v_k)$ solves
\eqref{LEsystem1} on $\Omega_k$ and
\be{defMkLE}
M_k:=\sqrt{\phi(N(u_k,v_k))}
\ee
satisfies
$$M_k(z_k)>2k\,\bigl(1+{\rm dist}^{-1}(z_k,\partial\Omega_k)\bigr).$$
By the doubling lemma in \cite[Lemma 5.1]{PQS1},
it follows that there exists $x_k\in \Omega_k$ such that
$$M_k(x_k)>2k\,{\rm dist}^{-1}(x_k,\partial\Omega_k),
\quad M_k(x_k)\ge M_k(z_k)\ge 2k$$
and
\be{LEresc0}
M_k(z)\leq 2M_k(x_k),\qquad |z-x_k|\leq k\,M_k^{-1}(x_k).
\ee
Now we rescale $(u_k,v_k)$ by setting
\be{LEresc1}
N_k=N\bigl(u_k(x_k),v_k(x_k)\bigr),\quad 
\theta_k={1\over M_k(x_k)}={1\over\sqrt{\phi(N_k)}},
\ee
\be{LEresc2}
\lambda_k=h_2^{-1}(N_k),\quad \mu_k=h_1^{-1}(N_k)
\ee
and
\be{LErescuv}
\left.
\begin{aligned}
\tilde u_k(y)&:=\lambda_k^{-1}u_k(x_k+\theta_k y),\\
\tilde v_k(y)&:=\mu_k^{-1}v_k(x_k+\theta_k y),
\end{aligned}
\quad\right\}
\qquad |y|\leq k.\\
\ee
Since $M_k(x_k)\to \infty$, we see that $N_k, \lambda_k, \mu_k\to \infty$
and $\theta_k\to 0$ as $k\to\infty$.
The rescaled functions $(\tilde u_k,\tilde v_k)$ solve
\be{LEresc3}
\left.\begin{aligned}
-\Delta_y \tilde u_k(y)&=\lambda_k^{-1}\theta_k^2 f_1(\mu_k\tilde v_k(y))=:g_{1,k}(\tilde v_k(y)),\\
\noalign{\vskip 1mm}
-\Delta_y \tilde v_k(y)&=\mu^{-1}\theta_k^2 f_2(\lambda_k\tilde u_k(y))=:g_{2,k}(\tilde u_k(y)),\\
\end{aligned}\quad\right\}
\qquad |y|\leq k.
\ee
Notice that, by \eqref{defLEphi}, \eqref{LEresc1}, \eqref{LEresc2}, we have
$$\lambda_k^{-1}\theta_k^2 f_1(\mu_k)=
(\lambda_k \mu_k)^{-1}{h_1(\mu_k)\over\phi(N_k)}=
(\lambda_k \mu_k)^{-1}{N_k\over\phi(N_k)}=
(\lambda_k \mu_k)^{-1}h_1^{-1}(N_k)h_2^{-1}(N_k)=1$$
and similarly $\mu_k^{-1}\theta_k^2 f_2(\lambda_k)=1$,
hence
\be{LEresc4}
g_{1,k}(\tilde v_k(y))={f_1(\mu_k\tilde v_k(y))\over f_1(\mu_k)},\quad
g_{2,k}(\tilde u_k(y))={f_2(\lambda_k\tilde u_k(y))\over f_2(\lambda_k)}.
\ee
Moreover, by assumption \eqref{hypLE2} and Lemma~\ref{LemUnifConv0},  
we have
\be{LEresc5}
{f_1(\mu_k s)\over f_1(\mu_k)}\to s^p,\quad 
{f_2(\lambda_k s)\over f_2(\lambda_k)}\to s^q,
\quad\hbox{uniformly for $s\ge 0$ bounded.}
\ee

\smallskip

{\bf Step 3.} {\it Local bounds.} 
We now proceed to establish the local bound
\be{LErescBound}
\tilde u_k(y), \tilde v_k(y)\le M_0,\quad |y|\le k,\quad k\ge 1,
\ee
with some constant $M_0>0$.
By \eqref{LEresc0}, we have
\be{LEresc5b}
\phi\bigl(N(u_k(x),v_k(x))\bigr)={M^2_k(x)\over M^2_k(x_k)}\phi(N_k)\le 4\phi(N_k),\quad |x-x_k|\le k\theta_k.
\ee
On the other hand, owing to \eqref{infphi00} for $g=\phi$ and \eqref{infphi001},
there exists $\lambda>1$ such that $4\phi(N_k)\le\phi(\lambda N_k)$
for all $k$. This along with \eqref{LEresc5b} and \eqref{infphi0} guarantees that
\be{LEresc5c}
N(u_k(x),v_k(x))\le\lambda N_k,\quad |x-x_k|\le k\theta_k.
\ee
Also, restricting to $k$ sufficiently large without loss of generality, we may assume that $h_1^{-1}(N_k)=\mu_k\ge A$.
Next combining \eqref{LEresc5c} with \eqref{infphi00} for $g=h_1$ and \eqref{infphi001}, we obtain $\bar\lambda>1$ such that
\be{LEresc6}
h_1(v_k(x))\le\lambda N_k=\lambda h_1(h_1^{-1}(N_k))\le h_1(\bar\lambda h_1^{-1}(N_k)),\quad |x-x_k|\le k\theta_k.
\ee
It follows from \eqref{h1h2mon} and \eqref{LEresc6} that
$$v_k(x)\le A\vee \bar\lambda h_1^{-1}(N_k)= \bar\lambda \mu_k,\quad |x-x_k|\le k\theta_k.$$
Recalling \eqref{LErescuv} and arguing similarly for $u_k$, we thus obtain \eqref{LErescBound}.

\smallskip

{\bf Step 4.} {\it Nonvanishing and conclusion.} 
Going back to \eqref{LEresc3}-\eqref{LEresc5} and using interior elliptic estimates, we deduce that 
$(\tilde u_k,\tilde v_k)$ converges locally uniformly to a (bounded) 
solution $(\tilde u,\tilde v)$ of 
 \eqref{LEsystem2} in~$\Rn$.
 
 We claim that $(\tilde u,\tilde v)$ is nontrivial.
Indeed, from \eqref{LEresc2}, we have
\be{LEresc7}
h_1(\mu_k)=h_2(\lambda_k)=N_k=2A+h_1(v_k(x_k))+h_2(u_k(x_k)).
\ee
By passing to a subsequence, we may assume that
$h_1(v_k(x_k))\ge h_2(u_k(x_k))$ (the case $h_1(v_k(x_k))\ge h_2(u_k(x_k))$ being similar)
and \eqref{LEresc7} then guarantees that $v_k(x_k)\to\infty$ and $h_1(\mu_k)\le 3h_1(v_k(x_k))$ for $k$ large.
Applying again \eqref{infphi00} for $g=h_1$, we deduce that $\mu_k\le Cv_k(x_k)$, 
hence $\tilde v_k(0)={v_k(x_k)\over \mu_k}\ge 1/C$.
Consequently, $(\tilde u,\tilde v)$ is nontrivial.
But this contradicts our nonexistence assumption.
\end{proof}

\begin{proof}[Proof of Theorem~\ref{thmLEpert1}(ii)]
This follows by modifying the proof of Theorem~\ref{thmLEpert1}(i)
along the lines of the proof of \cite[Theorem~4.1]{PQS2},
also using the Liouville property for system \eqref{LEsystem2} in $\R^n_+$ with zero boundary conditions.
The latter, for given $(p,q)$, is a consequence of the Liouville property in $\R^n$
(see \cite{BiMi} and \cite[Theorem~4.2]{PQS1}).
\end{proof}

\subsection{Proof of Theorem~\ref{thmLEpert1B}}
It is based on suitable modifications of the proof of Theorem~\ref{thmLEpert1}.
We first have the following analogue of Lemma~\ref{LemphiMonot}.

\begin{lemma} \label{LemphiMonotB}
Let $\phi$ be given by \eqref{defLEphi} for $t>0$ small. There exist $0<\eps_2<\eps_1<\eps_0$
depending only on $f_1, f_2$ such that 
\be{infphi0B}
\hbox{$\phi$ is well defined and increasing on $[0,\eps_1]$},
\ee
\be{infphi00B}
\inf_{t\in(0,\eps_1]}{g(t)\over g(\delta t)}\to\infty,
\ \hbox{as $\delta\to 0$, \quad for $g\in\{h_1,h_2,\phi\}$},
\ee
\be{infphi000B}
\sup_{t\in(0,\eps_1]}{g(t)\over g(\sigma t)}<\infty,
\ \hbox{for each $\sigma \in(0,1)$ and $g\in\{f_1,f_2,h_1^{-1}\circ h_2,h_2^{-1}\circ h_1\}$}
\ee
and
\be{Nt1t2B}
0\le h_1(t), h_2(t)\le \eps_1/2\quad\hbox{for all $t\in(0,\eps_2]$.}
\ee
\end{lemma}

\begin{proof}
Fix $\eps\in(0,1)$ to be chosen later. 
By \eqref{hypLE1}-\eqref{hypLE2} as $s\to 0$, the function
$z(s):=s^{\eps-p-1}h_1(s)=s^\eps L_1(s)$ is $C^1$ for $s$ small 
enough with 
$$z'(s)=\Bigl(\eps+{sL_1'(s)\over L_1(s)}\Bigr)s^{\eps-1}L_1(s)>0.$$
This and the similar argument for $h_2$
guarantee the existence of $\bar\eps_0\in(0,\eps_0]$ such that
\be{infphi1B}
\hbox{the functions ${h_1(s)\over s^{p+1-\eps}}$
and ${h_2(s)\over s^{q+1-\eps}}$ are increasing on $(0,\bar\eps_0]$.}
\ee
It follows that
${h_1(t)\over h_1(\delta t)}\ge \delta^{-(p+1-\eps)}$
and ${h_2(t)\over h_2(\delta t)}\ge \delta^{-(q+1-\eps)}$ for all $t\in(0,\bar\eps_0]$  
and $\delta\in(0,1]$,
hence \eqref{infphi00B} for $g\in\{h_1,h_2\}$
 and any $\eps_1\in(0,\bar\eps_0]$.

\smallskip

Set $t_1:= h_1(\bar\eps_0)\wedge h_2(\bar\eps_0)$ and pick any $0<t<t'\le t_1$.
Then $s':=h_1^{-1}(t')\ge h_1^{-1}(t)=:s$ and it follows from \eqref{infphi1B} that
${(h_1(s'))^{1/(p+1-\eps)}\over s'}\ge {(h_1(s))^{1/(p+1-\eps)}\over s}$,
hence
\be{infphi2B}
{{t'}^{1/(p+1-\eps)}\over h_1^{-1}(t')}\ge {t^{1/(p+1-\eps)}\over h_1^{-1}(t)}
\quad\hbox{and (similarly)}\quad {{t'}^{1/(q+1-\eps)}\over h_2^{-1}(t')}\ge {t^{1/(q+1-\eps)}\over h_2^{-1}(t)}.
\ee
Set $\theta=\theta_\eps=1-\bigl({1\over p+1-\eps}+{1\over q+1-\eps}\bigr)$.
Since $\theta_\eps\to {pq-1\over (p+1)(q+1)}>0$ as $\eps\to 0$, we may choose $\eps=\eps(p,q)>0$ small enough so that $\theta>0$.
Setting $\psi(t)=t^{-\theta}\phi(t)$ and multiplying the two inequalities in \eqref{infphi2B} yields $\psi(t')\ge \psi(t)$,
i.e.~$\psi$ is nondecreasing on $(0,t_1]$. 
This in particular implies \eqref{infphi0B} for any $\eps_1\in(0,t_1]$.

Also, for all $t\in(0,t_1]$ and $\delta\in(0,1]$, we have 
${\phi(t)\over \phi(\delta t)}=\delta^{-\theta}{\psi(t)\over \psi(\delta t)}\ge \delta^{-\theta}$,
hence \eqref{infphi00B} for $g=\phi$ and any $\eps_1\in(0,t_1]$.

 We next prove \eqref{infphi000B}. If a function $g$ is $C^1$ and $>0$ near $0^+$ and $L(s):=s^{-p} g(s)$ satisfies
$\lim_{s\to 0^+} {sL'(s)\over L(s)}=0$ for some $p>0$, then
\be{limsgprimeg}
\lim_{s\to 0^+}\frac{sg'(s)}{g(s)}=p.
\ee
Taking $\eps_1>0$ smaller if necessary, property \eqref{limsgprimeg} guarantees that, for each $\sigma \in(0,1)$,
$$\log\left(\frac{g(t)}{g(\sigma t)}\right)
=\int_{\sigma t}^t \frac{g'(s)}{g(s)}ds
=\int_{\sigma t}^t (p+o(1))s^{-1} ds \le 2p\log(\sigma^{-1}),\quad t\in (0,\eps_1],$$
hence $\sup_{t\in(0,\eps_1]}{g(t)\over g(\sigma t)}<\infty$.
Now elementary computations show that property \eqref{limsgprimeg} is stable under composition 
and passage to the reciprocal ($p$ becoming the product or the inverse, respectively).
In view of our assumption that \eqref{hypLE1}, \eqref{hypLE2}  are satisfied as $s\to 0$, this implies \eqref{infphi000B}.

As for \eqref{Nt1t2B}, it is immediate for $\eps_2\in(0,\eps_1)$ small. The proof is complete.
\end{proof}

\begin{proof}[Proof of Theorem~\ref{thmLEpert1B}(i)]
{\bf Step 1.} {\it Proof of \eqref{estimLE1BD} and equivalent estimate.} 
Let $\eps_2$ be given by Lemma~\ref{LemphiMonotB}.
By Proposition~\ref{propdecay}, we may choose $D=D(n,f,\Lambda)>0$ such that, for any $U$ satisfying the assumptions of the theorem, we have
\be{uveps2}
u,v\le\eps_2\quad\hbox{in $\tilde\Omega_D$,}
\ee
which in particular implies \eqref{estimLE1BD}.
Let
$$N(t_1,t_2):=h_1(\sigma t_2)+h_2(\sigma t_1),\quad \sigma=\eps_2\Lambda^{-1}.$$
We note that, owing to $\|U\|_\infty\le\Lambda$, \eqref{infphi0B} and \eqref{Nt1t2B}, we have
\be{defMkLE0B2}
N(u,v)\le \eps_1
\ee
and the function $\phi(N(u,v))$ (where $\phi$ is given by \eqref{defLEphi}) is well defined and continuous on~$\Omega$.
We shall show that
\be{defMkLE0B}
\phi(N(u,v))\le Cd^{-2}(x),\quad x\in \tilde\Omega_D,
\ee
with $C=C(n,f,\Lambda)>0$.
Estimate \eqref{defMkLE0B} will imply \eqref{estimLE1B} since,
 by \eqref{infphi0B}, \eqref{infphi000B}, \eqref{Nt1t2B} and \eqref{uveps2}, we will then have
$${f_2(u)\over h_1^{-1}(h_2(u))}\le C{f_2(\sigma u)\over h_1^{-1}(h_2(\sigma u))}= C\phi(h_2(\sigma u))\le C\phi(N(u,v))\le Cd^{-2}(x),
\quad x\in \tilde\Omega_D,$$
and similarly for $v$.

\smallskip

{\bf Step 2.} {\it Doubling-rescaling argument.} 
Thus assume for contradiction that estimate \eqref{defMkLE0B} fails. 
By arguing exactly as in Step 2 of the proof of Theorem~\ref{thmLEpert1}(i), we obtain solutions of
$$
\left.\begin{aligned}
-\Delta_y \tilde u_k(y)&={f_1(\mu_k\tilde v_k(y))\over f_1(\mu_k)},\\
\noalign{\vskip 1mm}
-\Delta_y \tilde v_k(y)&={f_2(\lambda_k\tilde u_k(y))\over f_2(\lambda_k)},\\
\end{aligned}\quad\right\}
\qquad |y|\leq k,
$$
with
\be{LEresc0BB}
M_k(z)\le 2M_k(x_k),\qquad |z-x_k|\leq k\,M_k^{-1}(x_k),
\ee
where $M_k, N_k, \lambda_k, \mu_k$ are defined by \eqref{defMkLE}, \eqref{LEresc1}, \eqref{LEresc2}.
The main difference is that, owing to \eqref{infphi0B} and \eqref{defMkLE0B2}, we now have
$$\bar M:=\phi^{1/2}(\eps_1)\ge M_k(x_k)>2k\,{\rm dist}^{-1}(x_k,\partial\Omega_k)$$
for $k$ large, so that 
${\rm dist}(x_k,\partial\Omega_k)\ge 2k\bar M^{-1}\to \infty$ as $k\to\infty$,
hence $|U_k(x_k)|\to 0$ by Proposition~\ref{propdecay}.
Consequently, $N_k, \lambda_k, \mu_k\to 0$. Also \eqref{LEresc5} is satisfied.

\smallskip

Then arguing as in Steps 3 and 4 of the proof of Theorem~\ref{thmLEpert1}(i)
and using \eqref{defMkLE0B2}, \eqref{LEresc0BB} and Lemma~\ref{LemphiMonotB}, instead of Lemma~\ref{LemphiMonot},
we obtain the local bound $\tilde u_k(y), \tilde v_k(y)\le M_0$ for $|y|\le k$ and $k$ large, with some constant $M_0>0$,
and we conclude by reaching a contradiction similarly as before. 
 \end{proof}

\begin{proof}[Proof of Theorem~\ref{thmLEpert1B}(ii)]
By Theorem~\ref{thmLEpert1}(i), there exists $\Lambda=\Lambda(n,f)>0$ such that
any solution of~\eqref{LEsystem1} in $\Omega$ satisfies $|U|\le \Lambda$ in
the open set $\tilde \Omega_1$.
It follows from assertion (i) that 
\eqref{estimLE1BD} and \eqref{estimLE1B} are true with $\tilde\Omega_D$ replaced by 
$\bigl\{x\in\tilde\Omega_1;\ {\rm dist}(x,\partial\tilde\Omega_1)>D\bigr\}$.
Since ${\rm dist}(x,\partial\Omega)\le 2\,{\rm dist}(x,\partial\tilde\Omega_1)$ in $\tilde\Omega_D$ whenever $D\ge 2$, the conclusion follows
(with possibly different constants $C,D$, now depending only on $n,f$).
 \end{proof}

\subsection{Proof of Theorem~\ref{thmUB}}
Let us recall that in the scalar case $m=1$, the generalized rescaling limits
of any positive regularly varying function are power functions. Namely, if
$f(s)>0$ for $s>0$ large, 
and $\lim_{\lambda\to\infty}\frac{f(\lambda s)}{f(\lambda)}= f_\infty(s)\in(0,\infty)$ for all $s>0$,
then there exists $p\in\R$ such that $f_\infty(s)=s^p$
(see \cite[Theorem~1.3]{Se} for a more general statement).
If $p>0$, then Lemma~\ref{LemUnifConv0}(i) guarantees 
$$\lim_{\lambda\to\infty}\frac{f(\lambda s)}{f^+(\lambda)}=s^p \quad\hbox{locally uniformly for } s\in[0,\infty).$$
The proof of Theorem~\ref{thmUB} will use the following lemma, which provides analogous properties in the vector valued case $m\geq1$,
and whose proof is given in appendix.

\begin{lemma} \label{lem-theta}
Let $f\in C(K)$, $\Lambda>0$. 
\smallskip

(i) Assume 
\be{assh}
\hbox{$f(U)\ne0$ for $|U|\geq\Lambda$},\quad
\lim_{\lambda\to\infty}\frac{f(\lambda U)}{f^+(\lambda)}=  f_\infty(U) \quad\hbox{locally uniformly for } U\in K,
\ee
where $f_\infty$ is finite.
Then there exists $p\ge0$ such that $f_\infty$ is homogeneous of degree $p$,
\be{asshplus}
\lim_{\lambda\to\infty}\frac{\+f(\lambda s)}{\+f(\lambda)}=  \+f_\infty(s)
                    \quad\hbox{ locally uniformly for } s\ge 0, 
\ee
$\+f_\infty(s)=s^p$, and we have
\be{fpinfty}
(\forall\theta>0)\,(\exists\lambda_\theta>\Lambda)\quad
\lambda>\lambda_\theta\ \Rightarrow\ 
 \lambda^{p-\theta}\leq\+f(\lambda)\leq\lambda^{p+\theta}.
\end{equation}

(ii) Assume 
\be{assg}
\hbox{$f(U)\ne0$ for $0<|U|\leq\Lambda$},\quad
\lim_{\lambda\to0+}\frac{f(\lambda U)}{f^+(\lambda)}= f_0(U) \quad\hbox{locally uniformly for } U\in K,
\ee
where $f_0$ is finite.
Then there exists $q\geq0$ such that $f_0$ is homogeneous of degree $q$,
\be{assgplus}
\lim_{\lambda\to0+}\frac{\+f(\lambda s)}{\+f(\lambda)}= \+f_0(s)
                    \quad\hbox{ locally uniformly for } s\ge 0, 
\ee
$\+f_0(s)=s^q$,
and we have
\be{fp0}
(\forall\theta>0)\,(\exists\lambda_\theta\in(0,\Lambda))\quad
\lambda\in(0,\lambda_\theta)\ \Rightarrow\ 
 \lambda^{q+\theta}\leq\+f(\lambda)\leq\lambda^{q-\theta}.
\end{equation}
\end{lemma}

\begin{remark} \label{rem-superlin}
(i) Assume \eqref{assh}. Then \eqref{asshplus} is true with $f_\infty(s)=s^p$ for some $p\geq0$ and
the proof of \cite[Lemma~8.2(i)]{SouDCDS} implies  
$$ (\forall\theta>0)(\exists c_\theta>0)\quad
\inf_{\lambda\geq\Lambda}\frac{f^+(\lambda s)}{f^+(\lambda)}\geq c_\theta s^{p-\theta}
\qquad\hbox{for all } s\geq1. $$
Similarly if we assume \eqref{assg}, then \eqref{assgplus} is true with $f_0(s)=s^q$ for some $q\geq0$ and
$$ (\forall\theta>0)(\exists c_\theta>0)\quad 
\frac{f^+(\lambda s)}{f^+(\lambda)}\geq c_\theta s^{q-\theta}
\qquad\hbox{provided } s\geq1 \hbox{ and } \lambda s\leq\Lambda.$$
Finally, if $f(U)\ne0$ for $U\ne0$, then
\eqref{asshplus}, \eqref{assgplus} with $f_\infty(s)=s^p$, $f_0(s)=s^q$, $p,q>1$, 
and the proof of \cite[Lemma~8.2(ii)]{SouDCDS} imply
the existence of $r>1$ and $c>0$ such that
$$  \inf_{\lambda>0}\frac{f^+(\lambda s)}{f^+(\lambda)}\geq c s^{r}
\qquad\hbox{for all } s\geq1. $$

(ii) Although the function $f^+$ depends on the choice of the norm $|\cdot|$ in $\R^m$, Lemma 7.3 % \ref{lem-theta}  
remains true for any choice of this norm.
In addition, if $|\cdot|,|\cdot|\tilde{\phantom{w}}$ are two norms in $\R^m$ and
$\tilde f^+(\lambda)=\max_{|U|\tilde{\phantom{w}}=\lambda}|f(U)|\tilde{\phantom{w}}$, then
$$ \tilde f_0(U):=\lim_{\lambda\to0+}\frac{f(\lambda U)}{\strut\tilde f^+(\lambda)}
     =\lim_{\lambda\to0+}\frac{f(\lambda U)}{f^+(\lambda)}\,\frac{f^+(\lambda)}{\strut\tilde f^+(\lambda)}
   =f_0(U) \lim_{\lambda\to0+}\frac{f^+(\lambda)}{\strut\tilde f^+(\lambda)}
   = f_0(U) \frac{f^+_0(1)}{\strut\tilde f^+_0(1)}. $$
\end{remark}

\begin{proof}[Proof of Theorem~\ref{thmUB}]
The nonexistence of nontrivial bounded entire solutions of $-\Delta V=\varphi(V)$ implies $\varphi(U)\ne0$ whenever $U\ne0$.
In particular, if $\varphi=f$, then $f^+(\lambda)>0$ for $\lambda>0$.

If $U\ne0$, then we set 
$$M(U):=\sqrt{\frac{\+f(|U|)}{\sigma+|U|}},$$
where $\sigma=1$ in Case (i) and $\sigma=0$ otherwise.
Assuming $|U|\geq\Lambda$ in Case (i) and $|U|\leq\Lambda$ in Case (ii),
we will prove
\begin{equation} \label{estM}
 M(U)\leq C(\sigma+\hbox{\rm dist}(x,\partial \Omega)^{-1}).
\end{equation}
Notice that \eqref{estM} implies estimates \eqref{UBh} and \eqref{UB}.

Remark~\ref{rem-superlin}(i) implies the existence of $c>0$ and $r>1$ such that
$$f^+(\lambda s)\geq cs^rf^+(\lambda)
\quad\hbox{whenever }s\geq1\hbox{ and }
\begin{cases} \lambda\geq\Lambda &\hbox{ in Case (i)},\\
              \lambda s\in(0,\Lambda] &\hbox{ in Case (ii)},\\
              \lambda>0 &\hbox{ in Case (iii)}.
\end{cases}$$
Fix $s_0>1$ such that $\kappa_0:=cs_0^{r-1}\ge 4$.
Then the inequality $f^+(\lambda s)\geq cs^rf^+(\lambda)$
implies 
\begin{equation} \label{fsup2}
f^+(\lambda s)> \kappa_0 sf^+(\lambda)\quad\hbox{for }s > s_0
\end{equation}
and inequality \eqref{fsup2} guarantees the implication
\begin{equation} \label{s0bound}
M^2(U)\leq \kappa_0 M^2(V)\ \Rightarrow\ |U|\leq s_0|V|,
\end{equation}
provided also $|U|,|V|\geq\Lambda$ in Case (i)
and $|U|,|V|\leq\Lambda$ in Case (ii). 

We will use the same arguments as in the proof of \cite[Theorem 4.3]{SouDCDS}. 
Assume to the contrary that there exist $\Omega_k$,
solutions $U_k$ of $-\Delta U=f(U)$ in $\Omega_k$ and $x_k\in \Omega_k$ such that
\begin{equation} \label{flarge}
 M_k(x_k)>2k 
 \bigl({\sigma+}\hbox{dist}(x_k,\partial \Omega_k)^{-1}\bigr),
\end{equation}
where $M_k:=M(U_k)$, with also
$|U_k(x_k)|\geq\Lambda$ in Case (i) and $\|U_k\|_\infty\leq\Lambda$ in Case (ii).
 The doubling lemma in \cite[Lemma~5.1]{PQS1}
guarantees that, changing the points $x_k$ if necessary, we may also assume
$M_k(x)\le 2M_k(x_k) \le \sqrt{\kappa_0} M_k(x_k)$ in $\hat \Omega_k=\{x:|x-x_k|\leq \frac{k}{M_k(x_k)}\}$.
(If $k$ is large, then the inequality $|U_k(x_k)|\geq\Lambda$ in Case~(i) remains true
 for the new points $x_k$ given by
the doubling lemma, since $M_k\to\infty$ 
in Case~(i) and $f^+(U)$ is bounded for $|U|\leq\Lambda$.)
Consequently, \eqref{s0bound} implies $|U_k(x)|\leq s_0|U_k(x_k)|$ for $x\in \hat \Omega_k$.
Set
$$ m_k:=|U_k(x_k)|,\quad\alpha_k:=\frac1{M_k(x_k)},\quad
V_k(y)=\frac1{m_k}U_k(x_k+\alpha_k y), $$
where $y\in\tilde \Omega_k:=B_{k/2}$.
Then $|V_k(0)|=1$, $|V_k|\leq s_0$ and
\begin{equation} \label{eqVk} 
-\Delta V_k=\frac{\alpha_k^2}{m_k}f(m_k V_k)={\frac{\sigma+m_k}{m_k}}\frac1{\+f(m_k)}f(m_k V_k)\quad\hbox{in }\tilde \Omega_k.
\end{equation}
Since $m_k\to\infty$ in Case~(i) and $m_k\leq\Lambda$ in Case~(ii),
 our assumptions guarantee that
the right-hand side of \eqref{eqVk} is bounded,
and we can suppose that $V_k\to W$ locally uniformly,
where $|W(0)|=1$ and $|W|\leq s_0$.
If $m_k\to\infty$ or $m_k\to0$, then $W$ is a nontrivial bounded entire solution 
of $-\Delta W=\varphi(W)$, where $\varphi=f_\infty$ or $\varphi=f_0$, respectively,
which yields a contradiction.
Hence we may assume $m_k\to a\in(0,\infty)$.
Then $W$ is a nontrivial bounded entire solution of 
$-\Delta W=\frac{f(aW)}{\+f(a)}$, hence $\tilde W(y):=aW(\lambda y)$
with $\lambda:=\sqrt{\+f(a)/a}$ is a nontrivial bounded entire solution of 
$-\Delta\tilde W=f(\tilde W)$, which yields a contradiction again.
\end{proof}

\goodbreak

\section{Appendix: Some properties of regularly varying functions}

It is well known (see \cite[Theorem~1.1]{Se}) that, for any regularly varying function of index $p\in\R$, the generalized rescaling limits 
are uniform on compact subsets of $(0,\infty)$.
We have actually used the following lemma, which gives a stronger property: When $p>0$,
 the convergence is uniform on compact subsets of $[0,\infty)$.
 This is probably known but we haven't found the result in the literature.

\begin{lemma} \label{LemUnifConv0}
Let $f\in C([0,\infty))$ and $p>0$.

\strut\hbox{\rm(i)} If $f(s)>0$ for $s>0$ large 
and $\lim\limits_{\lambda\to\infty}\frac{f(\lambda s)}{f(\lambda)}=s^p$
for all $s>0$, then
\be{limLocUnif}
\lim_{\lambda\to\infty}\frac{f(\lambda s)}{f(\lambda)}=s^p \quad\hbox{locally uniformly for } s\in[0,\infty).
\ee

\strut\hbox{\rm(ii)} If $f(s)>0$ for $s>0$ small
and $\lim\limits_{\lambda\to0+}\frac{f(\lambda s)}{f(\lambda)}=s^p$
for all $s>0$, then
\be{limLocUnif0}
\lim_{\lambda\to0+}\frac{f(\lambda s)}{f(\lambda)}=s^p \quad\hbox{locally uniformly for } s\in[0,\infty).
\ee
\end{lemma}

\begin{proof}
(i) Set $L(s):=s^{-p}f(s)$ for $s>0$.
Then $L$ satisfies 
\be{Lslow}
\lim_{\lambda\to\infty}\frac{L(\lambda s)}{L(\lambda)}=1\quad\hbox{for } s>0,
\ee
hence $L$ is slowly varying (see \cite{Se}), and \cite[Theorem~1.1]{Se} guarantees
that the limit in \eqref{Lslow} is locally uniform for $s\in(0,\infty)$.
Since $\frac{f(\lambda s)}{f(\lambda)}=s^p\frac{L(\lambda s)}{L(\lambda)}$,
assertion \eqref{limLocUnif} is true for $s\in(0,\infty)$.
Consequently, it is sufficient to prove that given $\eps>0$, there exist $\delta>0$
and $\Lambda>0$ such that 
\be{estLocUnif}
\Big|\frac{f(\lambda s)}{f(\lambda)}-s^p\Big|<\eps
\quad\hbox{whenever } s\in[0,\delta)\ \hbox{ and } \lambda>\Lambda.   
\ee
Fix $\eps>0$ and $\theta\in(0,p)$.
Then the representation theorem \cite[Theorem 1.2]{Se} implies
the existence of $\Lambda_1>0$ such that $L(\lambda)\geq\lambda^{-\theta}$ for $\lambda>\Lambda_1$.
In particular, $f(\lambda)\geq\lambda^{p-\theta}$ for $\lambda>\Lambda_1$,
hence there exists $\tilde\Lambda_1\ge\Lambda_1$ such that the inequality in \eqref{estLocUnif} 
is true for $s=0$ and $\lambda>\tilde\Lambda_1$.
The representation theorem \cite[Theorem 1.2]{Se} also implies the existence of $\delta_2\in(0,1)$ and $\Lambda_2>0$
such that $\frac{L(\lambda s)}{L(\lambda)}\leq s^{-\theta}$ provided $s\in(0,\delta_2)$ and $\lambda s>\Lambda_2$.
Consequently, there exists $\delta_1\in(0,\delta_2]$ such that
$$ \Big|\frac{f(\lambda s)}{f(\lambda)}-s^p\Big|=s^p\Big|\frac{L(\lambda s)}{L(\lambda)}-1\Big|
\leq s^p(s^{-\theta}+1)<\eps$$
whenever $s\in(0,\delta_1)$ 
and $\lambda s>\Lambda_2$.

Next assume $0<\lambda s\leq\Lambda_2$, $s<\delta_3$, $\lambda>\Lambda_3\geq\Lambda_1$.
Since $|f(t)|\leq C_2$ for $t\in[0,\Lambda_2]$, we obtain  
$$ \Big|\frac{f(\lambda s)}{f(\lambda)}-s^p\Big|=s^p\Big|\frac{L(\lambda s)}{L(\lambda)}-1\Big|
\leq s^p(C_2(\lambda s)^{-p}\lambda^\theta+1)=s^p+C_2\lambda^{\theta-p}<\eps$$
provided $\delta_3$ is small enough and $\Lambda_3$ is large enough.
This concludes the proof of \eqref{estLocUnif}. 

\smallskip

(ii) Set $\tilde L(s):=s^{-p}f(s)$ and $L(s):=\tilde L(1/s)$ for $s>0$.
Then \eqref{Lslow} is true and
$\frac{f(\lambda s)}{f(\lambda)}=s^p\frac{L(1/(\lambda s))}{L(1/\lambda)}$,
hence, as in (i), the limit in \eqref{limLocUnif0} is locally uniform for $s\in(0,\infty)$
and it is sufficient to show
\be{estLocUnif0}
\Big|\frac{f(s/\lambda)}{f(1/\lambda)}-s^p\Big|<\eps
\quad\hbox{whenever }\ s\in[0,\delta)\ \hbox{ and }\ \lambda>\Lambda.   
\ee
Since, by assumption, $\lim\limits_{\lambda\to0+}\frac{f(\lambda/2)}{f(\lambda)}=2^{-p}<1$
hence $f(0)=0$, the inequality in \eqref{estLocUnif0}
is true if $s=0$.
As in (i) we obtain $\frac{L(\lambda/s)}{L(\lambda)}\leq s^{-\theta}$ if $s<\delta_2$ and $\lambda>\Lambda_2$.
Consequently,
$$ \Big|\frac{f(s/\lambda)}{f(1/\lambda)}-s^p\Big|=s^p\Big|\frac{L(\lambda/s)}{L(\lambda)}-1\Big|
\leq s^p(s^{-\theta}+1)<\eps$$
if $s>0$ is small enough and $\lambda>\Lambda_2$, hence \eqref{estLocUnif0} is true.
\end{proof}

We next give the proof of Lemma~\ref{lem-theta}.

\begin{proof}[Proof of Lemma~\ref{lem-theta}]
We will only prove (ii); the proof of (i) is analogous.

It is easily seen that $f_0$ is continuous.
We first prove \eqref{assgplus}.
Set $K_s=\{U\in K:\ |U|=s\}$. 
Fix $\eps,M>0$. By \eqref{assg} 
there exists $\lambda_0>0$ such that, for all $0<\lambda\le\lambda_0$,
\be{F1abs}
\Bigl|\frac{f(\lambda U)}{\+f(\lambda)}-f_0(U)\Bigr|\le\eps
                    \quad\hbox{ for } U\in K,\ |U|\le M.
\ee
In particular we have
$$
\frac{|f(\lambda U)|}{\+f(\lambda)}\le |f_0(U)|+\eps\le \+f_0(s)+\eps
                    \quad\hbox{ for } U\in K_s,\ 0\le s\le M.
$$
Since $K_{\lambda s}=\lambda K_s$, by taking supremum over $K_s$, we get
\begin{equation} \label{F1plusupper}
\frac{\+f(\lambda s)}{\+f(\lambda)}\le \+f_0(s)+\eps
                    \quad\hbox{ for } 0\le s\le M.
\end{equation}
On the other hand, for any $s\ge 0$ we may choose $U_s\in K_s$ such that $\+f_0(s)=|f_0(U_s)|$,
hence \eqref{F1abs} with $U=U_s$ implies
\begin{equation} \label{F1pluslower}
\frac{\+f(\lambda s)}{\+f(\lambda)}\ge\frac{|f(\lambda U_s)|}{\+f(\lambda)}\ge |f_0(U_s)|-\eps=\+f_0(s)-\eps
                    \quad\hbox{ for } 0\le s\le M.
\end{equation}
Inequalities \eqref{F1plusupper} and \eqref{F1pluslower} together imply \eqref{assgplus}.

Notice that $\+f_0(1)=1$ and $\+f_0$ is continuous, hence $\+f_0$ is positive on a neighborhood of~$1$.
Consequently,
by standard results from regular variation theory (see, e.g.,~\cite[Theorem~1.3]{Se}), 
it follows from \eqref{assgplus} that $f_+$ has regular variation at $0$,
i.e. there exists $q\in\R$ such that
$$
L(s):= s^{-q} \+f(s) \hbox{ satisfies } 
\lim_{\lambda\to 0+} \frac{L(\lambda s)}{L(\lambda)}=1\ \hbox{ for each $s>0$.}
$$
This implies $\+f_0(s)=s^q$.
Moreover we have $q\ge 0$ owing to the continuity of $f$.
Then, for each $\mu>0$, going back to \eqref{assg}, we obtain
\begin{equation} 
f_0(\mu U)=\lim_{\lambda\to0+}\frac{f(\lambda \mu U)}{\+f(\lambda)}=
\lim_{\lambda\to0+}\frac{\+f(\mu\lambda)}{\+f(\lambda)}\lim_{\lambda\to0+}\frac{f(\lambda \mu U)}{\+f(\lambda\mu)}
=\mu^q f_0(U)
\end{equation}
i.e.,~$f_0$ is homogeneous of degree $q$.

On the other hand, by the integral representation property for slow variation functions (see~\cite[Section 1.2]{Se}),
there exist continuous functions $\eta,\xi$ such that
 $\lim_{s\to 0} \eta(s) = c$, $\lim_{z\to 0} \xi(z)=0$ and
$$  L(s)=\exp\Bigl[\eta(s)+\int_1^s z^{-1}\xi(z)\,dz\Bigr],\quad s\in(0,\Lambda). 
$$
As a consequence, we have $\log L(s)=o(|\log s|)$ as $s\to 0$,
which readily implies \eqref{fp0}. 
\end{proof}

\noindent{\bf Acknowledgement.}
The first author was supported in part by the Slovak Research and Development Agency
under the contract No.~APVV-23-0039 and by VEGA grant 1/0245/24.
Part of this work was done during visits of the first author at Universit\'e Sorbonne Paris Nord
and of the second author at Comenius University. The authors thank these institutions for their support.

\medskip

\noindent{\bf Declarations.}
The authors have no competing interests to declare regarding the work in this paper.

\font\pc=cmcsc9
\font\rmn=cmr9
\font\sln=cmsl9
\font\rmb=cmbx8 scaled 1125 \rm
\font\it=cmti9
\font\spa=cmr9

\end{document}